\documentclass[11pt,a4paper,reqno]{amsart}
\usepackage{amsfonts,amssymb,amsmath,amsthm,graphicx,color,amscd,xspace,verbatim,mathrsfs}
\makeatletter

\usepackage{hyperref}

\makeatletter
\def\@tocline#1#2#3#4#5#6#7{\relax
  \ifnum #1>\c@tocdepth 
  \else
    \par \addpenalty\@secpenalty\addvspace{#2}%
    \begingroup \hyphenpenalty\@M
    \@ifempty{#4}{%
      \@tempdima\csname r@tocindent\number#1\endcsname\relax
    }{%
      \@tempdima#4\relax
    }%
    \parindent\z@ \leftskip#3\relax \advance\leftskip\@tempdima\relax
    \rightskip\@pnumwidth plus4em \parfillskip-\@pnumwidth
    #5\leavevmode\hskip-\@tempdima
      \ifcase #1
       \or\or \hskip 1em \or \hskip 2em \else \hskip 3em \fi%
      #6\nobreak\relax
      \dotfill
      \hbox to\@pnumwidth{\@tocpagenum{#7}}
    \par
    \nobreak
    \endgroup
  \fi}
\makeatother

\newtheorem{theorem}{Theorem}[section]
\newtheorem{lemma}[theorem]{Lemma}
\newtheorem{proposition}[theorem]{Proposition}

\theoremstyle{definition}
\newtheorem{definition}[theorem]{Definition}

\theoremstyle{remark}
\newtheorem{remark}[theorem]{Remark}

\newcommand{\N}{{\mathbb N}}
\newcommand{\R}{{\mathbb R}}

\newcommand{\loc}{\mathrm{loc}}

\newcommand{\beqn}{\begin{eqnarray}}
\newcommand{\eeqn}{\end{eqnarray}}   
\newcommand{\beq}{\begin{eqnarray*}}
\newcommand{\eeq}{\end{eqnarray*}}

\newcommand{\be}{\begin{equation}}
\newcommand{\bel}[1]{\begin{equation}\label{#1}}
\newcommand{\ee}{\end{equation}}

\newcommand{\BA}{\begin{array}}
\newcommand{\EA}{\end{array}}
\newcommand{\BAN}{\renewcommand{\arraystretch}{1.2}
\setlength{\arraycolsep}{2pt}\begin{array}}
\newcommand{\BAV}[2]{\renewcommand{\arraystretch}{#1}
\setlength{\arraycolsep}{#2}\begin{array}}

\newcommand{\BSA}{\begin{subarray}}
\newcommand{\ESA}{\end{subarray}}
\newcommand{\BAL}{\begin{aligned}}
\newcommand{\EAL}{\end{aligned}}

\newcommand{\norm}[1]{\left \|#1\right \|}

\newcommand{\supp}{\mathrm{supp}\,}
\newcommand{\dist}{\mathrm{dist}\,}
\newcommand{\sign}{\mathrm{sign}}

\def\dist{\mathrm{dist}}

\def\gm{\mu}

\def\CA{{\mathcal A}}   \def\CB{{\mathcal B}}

   \def\CK{{\mathcal K}}   \def\CL{{\mathcal L}}
\def\CT{{\mathcal T}}

   \def\BBN {\mathbb N}    
   \def\BBR {\mathbb R}

\def\GTM {\mathfrak M}

\newcommand{\xa}{\alpha}
\newcommand{\xb}{\beta}
\newcommand{\xg}{\gamma}

\newcommand{\xd}{\delta}
\newcommand{\xD}{\Delta}
\newcommand{\xe}{\varepsilon}
\newcommand{\xz}{\zeta}

\newcommand{\xl}{\lambda}

\newcommand{\xm}{\mu}
\newcommand{\xn}{\nu}

\newcommand{\xs}{\sigma}

\newcommand{\xf}{\phi}
\newcommand{\xF}{\Phi}

\newcommand{\xO}{\Omega}

\newcommand{\tp}{{\tau_+(s,\xm)}}

\newcommand{\Hsg}{H_0^s(\Omega;|x|^\gamma)}

\newcommand{\Hm}{\mathbf{H}_{\mu,0}^s(\Omega)}

\newcommand{\scJ}{\mathscr{J}}

\newcommand{\1}{\textbf{1}}
\def\bal#1\eal{\small\begin{align*}#1\end{align*}\normalsize}
\def\ba#1\ea{\small\begin{align}#1\end{align}\normalsize}

\numberwithin{equation}{section}

\def\Xint#1{\mathchoice
{\XXint\displaystyle\textstyle{#1}}%
{\XXint\textstyle\scriptstyle{#1}}%
{\XXint\scriptstyle\scriptscriptstyle{#1}}%
{\XXint\scriptscriptstyle\scriptscriptstyle{#1}}%
\!\int}
\def\XXint#1#2#3{{\setbox0=\hbox{$#1{#2#3}{\int}$}
\vcenter{\hbox{$#2#3$}}\kern-.5\wd0}}

\def\dashint{\Xint-}

\begin{document}

\title[Semilinear fractional Hardy equations ]{Semilinear elliptic equations involving fractional Hardy operators}
\author[H. Chen]{Huyuan Chen}
\address{Huyuan Chen, Department of Mathematics, Jiangxi Normal University, Nanchang 330022, China}
\email{chenhuyuan@yeah.net}

\author[K. T. Gkikas]{Konstantinos T. Gkikas}
\address{Konstantinos T. Gkikas, Department of Mathematics, National and Kapodistrian University of Athens, 15784 Athens, Greece}
\email{kugkikas@math.uoa.gr}

\author[P.T. Nguyen]{Phuoc-Tai Nguyen}
\address{Phuoc-Tai Nguyen, Department of Mathematics and Statistics, Masaryk University, Brno, Czech Republic}
\email{ptnguyen@math.muni.cz}

\date{\today}

\begin{abstract}
 Our aim in this article is to study semilinear elliptic equations involving a fractional Hardy operator, an absorption and a Radon source in a weighted distributional sense. We show various scenarios, produced by the combined effect of the fractional Hardy potential, the growth of the absorption term and the concentration of the measure, in which existence and uniqueness results holds.

\bigskip

\noindent{\footnotesize Key words:  \textit{Semilinear elliptic problem; Fractional Hardy operators; Fractional Hardy potential; Radon measures. }
	
	\smallskip
	
}
	
\end{abstract}

\maketitle
\tableofcontents

\section{Introduction}
 For $s \in(0,1)$, $\mu \in \R$ and $x\in\R^N$ with $2 \leq N \in \N$, let $\CL_\mu^s$ be the fractional Hardy operator defined by
\bal 
\CL_\mu^s  := (-\Delta)^s    +\frac{\mu}{|x|^{2s}}.
\eal
Here $ (-\Delta)^s $ denotes the fractional Laplacian  defined by
\bal
(-\Delta)^s  u(x):= C_{N,s}\lim_{\epsilon\to 0^+} \int_{\R^N\setminus B_\epsilon(x) }\frac{ u(x)-
	u(y)}{|x-y|^{N+2s}}  dy,
\eal
where $B_\epsilon(x)$ is the ball  centered $x \in \R^N$ with radius $\epsilon>0$,   
\bal 
C_{N,s}=2^{2s}\pi^{-\frac N2}s\frac{\Gamma(\frac{N+2s}2)}{\Gamma(1-s)}>0
\eal 
with $\Gamma$ being the Gamma function. In this article, we consider the semilinear problem
\ba \label{eq:g(u)-kdirac-sum} \left\{\BAL
\CL_\xm^s u+g(u)&= \tilde \nu  &&\quad\text{in}\;\xO,\\
u&=0&&\quad\text{in}\;\BBR^N\setminus\xO,
\EAL \right.
\ea
where $g: \R \to \R$ is a nondecreasing continuous function such that $g(0)=0$ and $\tilde \nu$
is a bounded measure on $\Omega$.

\subsection{Review of the literature} When $s=1$, the operator $\CL_\mu^s$ becomes the local Hardy operator $\CL_\xm^1:=-\Delta +\frac{\mu}{|x|^2}$ and problem \eqref{eq:g(u)-kdirac-sum} reduces to the following semilinear elliptic equation
 \ba \label{eq:g(u)-kdirac-sum-loc}
\left\{\BAL
\CL_\xm^1  u+g(u)&= \tilde \nu   &&\quad\text{in}\;\,\xO,\\
u&=0&&\quad\text{on}\;\, \partial\xO.
\EAL \right.
 \ea
In the free-potential case, namely $\mu=0$, fundamental contributions were due to Brezis
\cite{B12} and B\'enilan and Brezis \cite{BB11}. When $N\geq 3$, it was proved that if $g:\R\to\R$   satisfies the {\it integral  subcritical assumption} 
\bal
\int_1^{+\infty}(g(s)-g(-s))s^{-1-\frac{N}{N-2}}ds<+\infty
\eal
then problem \eqref{eq:g(u)-kdirac-sum-loc} admits a unique weak solution. When $N=2$,
 Vazquez \cite{Va1} imposed a condition expressed in terms of  exponential growth of $g$ under which there exists a unique weak solution of \eqref{eq:g(u)-kdirac-sum-loc} with $\mu=0$.
Later on, Baras and Pierre \cite{BP2} studied \eqref{eq:g(u)-kdirac-sum-loc})
with $g(u)=|u|^{p-1}u$ for $p>1$ and they discovered that for $p\geq\frac{N}{N-2}$,
the problem is solvable if and only if $\tilde \nu$ is absolutely continuous
with respect to the Bessel capacity $c_{2,p'}$ with $p'=\frac{p}{p-1}$.  Since then, significant developments on problem \eqref{eq:g(u)-kdirac-sum-loc} in different directions have been established; see e.g.  \cite{BMP, MV-book,Ve-HandBook}. In case $\mu\not=0$, problem \eqref{eq:g(u)-kdirac-sum-loc} was studied in \cite{GuVe,D,BP} in connection with the Hardy inequality \cite{BV1,VazZua-00}. Thanks to a new notion of weak solutions of $\CL_\gm^1 u=0$ combined with a  dual formulation of the equation introduced in \cite{CQZ}, the authors of the paper
 \cite{ChVe} investigated problem \eqref{eq:g(u)-kdirac-sum-loc}and proved a existence and uniqueness result provided  that $g$ satisfies a subcritical integral assumption and the weak $\Delta_2$-condition. Moreover, when $g(u)=|u|^{p-1}u$ with $p>1$, they gave necessary and sufficient conditions for the existence of a weak solution to \eqref{eq:g(u)-kdirac-sum-loc}. For semilinear elliptic equations with more general potentials, we refer to \cite{GkiNg-absorption}.

When $s\in(0,1)$ and $\mu=0$, problem \eqref{eq:g(u)-kdirac-sum} turns to the fractional semilinear elliptic equation
\bal
 \left\{\BAL
(-\Delta)^s u+g(u)&= \tilde \nu  &&\quad\text{in}\;\xO,\\
u&=0&&\quad\text{in}\;\BBR^N\setminus\xO,
\EAL \right.
\eal
which has been studied  in \cite{CV} with Radon measure $\tilde \nu$.  The readers are referred to \cite{KMS} for the case $g=0$ and to \cite{NV} for boundary measures.

Now we return to problem \eqref{eq:g(u)-kdirac-sum}, which is driven by the fractional Hardy operator $\CL_\mu^s$. This operator arises in physical models related to
relativistic Schr\"odinger operator with Coulomb potential (see \cite{NRS,FLS1}), in the study of Hardy inequalities and Hardy-Lieb-Thirring inequalities (see, e.g., \cite{FLS,Fra-2009,tz}) and is closely related to the fractional Hardy inequality
\ba \label{mu_00}
\frac{C_{N,s}}{2}\int_{\R^N}\int_{\R^N} \frac{|\varphi(x)-\varphi(y)|^2}{|x-y|^{N+2s}}dydx + \mu_0 \int_{\R^N} \frac{|\varphi(x)|^2}{|x|^{2s}}dx \geq 0, \quad \forall \varphi \in C_0^\infty(\R^N),
\ea
where the best constant in \eqref{mu_00} is explicitly determined by  (see, e.g., \cite{FLS}) 
\bal
\mu_0 =-2^{2s}\frac{\Gamma^2(\frac{N+2s}4)}{\Gamma^2(\frac{N-2s}{4})}.
\eal 
Related inequalities could be found in  \cite{BF,FLS,Fra-2009}. Note that when $\mu\geq \mu_0$, the fractional Hardy operator $\CL^s_\mu$ is positive definite. The potential $\mu |x|^{-2s}$ is of the same homogeneity $-2s$ as $(-\Delta)^s$, hence it cannot be understood as a lower order perturbation of $(-\Delta)^s$. Moreover, in non-relativistic quantum mechanics, this potential exhibits the borderline between regular potentials (for which ordinary stationary states exist) and singular potentials (for which the energy is not bounded from below), therefore it may lead to disclose anomalous phenomena (see \cite{FLS-71}). Further properties of $\CL_\mu^s$ were obtained in \cite{B,JW,MY}.

Recent years have witnessed a growing interest in the study of elliptic equations with fractional Hardy potentials, evidenced by \cite{F,M,NRS,MW,WYZ,CW,CGN}.   
Let us recall relevant results on the linear equation  with a fractional Hardy potential. It was shown in \cite{CW} that  for $\mu\geq \mu_0$ the equation
\bal \CL_\mu^s u=0\quad\, {\rm in}\ \, \R^N\setminus \{0\}
\eal
has two distinct radial solutions
\ba\label{fu} \Phi_{s,\mu}(x)=\left\{
\BAL
&|x|^{\tau_-(s,\mu)}\quad
	&&\text{if }  \mu>\mu_0\\
&|x|^{-\frac{N-2s}{2}}|\ln|x|| \quad  &&\text{if } \mu=\mu_0
\EAL
\right.\quad   \text{and}\ \ \Gamma_{s,\mu}(x)=|x|^{\tau_+(s,\mu)} \ \ \text{for } x \in \R^N \setminus \{0\},
\ea
where
$\tau_-(s,\mu)  \leq  \tau_+(s,\mu)$. Additional properties of $\tau_-(s,\mu)$ and  $\tau_+(s,\mu)$ were given in \cite[Proposition 1.2]{CW}. 
\textit{In the remaining of the paper, when there is no ambiguity, we write for short $\tau_+$ and $\tau_-$ instead of $\tau_+(s,\mu)$ and $\tau_-(s,\mu)$.}

Note that for any $\xi\in C^2_c(\R^N)$,
 \bal
\int_{\R^N}\Phi_{s,\mu}   (-\Delta)^s_{\tau_+}\xi \, dx  =c_{s,\mu}\xi(0),
\eal
where $c_{s,\mu}>0$ and $(-\Delta)^s_{\tau_+}$ denotes the dual of the operator of $\CL^s_\mu$, which is  a weighted fractional Laplacian  given by
\bal
(-\Delta)^s_{\tau_+} v(x):=
C_{N,s}\lim_{\epsilon\to0^+} \int_{\R^N\setminus B_\epsilon(x) }\frac{v(x)-
v(y)}{|x-y|^{N+2s}} \, |y|^{\tau_+}dy.
\eal

By \cite[Theorem 4.14]{CW}, there exist a positive constant $c=c(N,s,\xm,\Omega)$ and a nonnegative function $\xF_{s,\xm}^\xO\in W^{s,2}_{\loc}(\BBR^N\setminus\{0\})$ such that $\xF_{s,\xm}^\xO=0$  in $\BBR^N\setminus\xO,$
\ba \label{PhiOmega}
\lim_{|x|\to 0^+}\frac{\xF_{s,\xm}^\xO(x)}{\xF_{s,\xm}(x)}=1\quad \text{and}\quad\xF_{s,\xm}^\xO(x)\leq c |x|^{\tau_-}, \quad \forall x\in \xO\setminus\{0\}.
\ea
Moreover,  for any $v\in C_0^\infty(\xO\setminus\{0\})$
 \bal
\ll \Phi_{s,\mu}^\Omega,v \gg_{\mu} :=\frac{C_{N,s}}{2}\int_{\R^N}\int_{\R^N}\frac{\big(u(x)-u(y)\big)\big(v(x)-v(y)\big)}{|x-y|^{N+2s}} dy   dx+\mu  \int_\Omega \frac{u(x) v(x)}{|x|^{2s}}dx= 0
\eal  
and for any $\psi\in C^{1,1}_0(\xO)$
\bal
\int_\xO \xF_{s,\xm}^\xO(-\xD)^s_{\tau_+}\psi\, dx=c_{s,\xm}\,\psi(0), 
\eal
where $c_{s,\xm}$ is a constant given in \cite[(1.15)]{CW}. 

In a recent paper \cite{CGN}, the authors of the present paper set up an appropriate distributional framework to investigate  of the form 
\ba\label{eq 1.0}\left\{\BAL
\CL_\xm^s u&= \tilde \nu  &&\quad\text{in}\;\xO,\\
u&=0&&\quad\text{in}\;\BBR^N\setminus\xO,
\EAL\right.
\ea
where $\tilde \nu$ is a bounded measure on $\Omega$.
The approach in \cite{CGN} is to analyze the associated weighted fractional Laplace operator $(-\Delta)^s_{\gamma}$ which is defined by
\ba \label{L}
	(-\Delta)^s_{\gamma} v(x)=C_{N,s}\lim_{\xd\to 0^+} \int_{\R^N\setminus B_\xd(x) }\frac{v(x)-
		v(y)}{|x-y|^{N+2s}} \, |y|^\gamma dy,
\ea
where  $\gamma\in\big[ -\frac{N-2s}{2}, 2s\big)$.  Particularly,  $(-\Delta)^s_0$ reduces to the fractional Laplacian.
From the integral-differential form of the weighted fractional Laplacian, a natural restriction for the  function $v$ is
\bal \|v\|_{L_{2s-\gamma}(\R^N)}:=\int_{\R^N}\frac{|v(x)|}{(1+|x|)^{N+2s-\gamma}} dx<+\infty.
\eal
The weighted Sobolev space associated to $(-\Delta)_{\gamma}^s$ is  $H_0^s(\xO;|x|^\gamma)$, which is defined as the closure of the functions in $C^\infty(\R^N)$ with the compact support in $\Omega$  under the norm
\ba 
\norm{u}_{H_0^s(\Omega ;|x|^\gamma)}:=\sqrt{\int_{\R^N}\int_{\R^N}\frac{|u(x)-u(y)|^2}{|x-y|^{N+2s}}|y|^\xg dy |x|^\xg dx}.
\ea
Note that $H_0^s(\xO;|x|^\gamma)$ is a Hilbert space with the inner product
\ba
\langle u,v\rangle_{s,\gamma}:=\int_{\R^N}\int_{\R^N}\frac{\big(u(x)-u(y)\big)\big(v(x)-v(y)\big)}{|x-y|^{N+2s}}|y|^\xg dy |x|^\xg dx.
\ea
{\it The results in \cite{CGN}, which will be recalled in Section \S 2, provide a  basic framework for the study of problem \eqref{eq:g(u)-kdirac-sum} with Radon measures.}  

Semiliear equations with a fractional Hardy potential have been a research objective in numerous papers; see e.g. \cite{F,M,NRS,MW,WYZ}. However, the above-mentioned works focus only on the case of source nonlinearities and rely on variational methods. To the best of our knowledge, to date, a profound understanding of the absorption case, namely equation \eqref{eq:g(u)-kdirac-sum}, is still lacking. The interplay between the fractional Hardy operator, the absorption nonlinearity and the measure datum generates different types of substantial difficulties and requires a different approach. In this paper, we perform a deep analysis of this interplay and develop a theory for \eqref{eq:g(u)-kdirac-sum} in measure frameworks. The main ingredients includes recent results in  \cite{CGN} (see subsection \ref{basic}), the local behavior near the origin of the solutions to the Poisson equation (see Proposition \ref{est1})  obtained by adapting ideas in \cite{DKP} and in \cite{KMS}, two-sided estimates of approximating solutions expressed in term of the fundamental solution $\Phi_{s,\mu}^{\Omega}$ (see Lemma \ref{aproxsol-sum}).    

\subsection{Framework, notion of solutions and main results}
Let $\Omega$ be a bounded domain in $\R^N$ containing the origin and $d(x)=\dist(x,\partial \Omega)$. For $q \in [1,+\infty)$, $\alpha, \beta \in \mathbb{R}$, we denote by $L^q(\Omega;d(x)^{\alpha}|x|^{\beta})$ the weighted Lebesgue space of functions $v: \Omega \to \mathbb{R}$ such that
\bal 
\| v \|_{L^q(\Omega;d(x)^{\alpha}|x|^{\beta})}:= \left( \int_{\Omega} |v|^q d(x)^{\alpha}|x|^{\beta} dx \right)^{\frac{1}{q}} < +\infty.
\eal
We denote by  {\it $\GTM(\Omega;d(x)^\alpha|x|^{\beta})$ (resp. $\GTM(\Omega\setminus\{0\};d(x)^\alpha|x|^{\beta})$) the space of Radon measures $\nu$ on $\Omega$ (resp. $\Omega\setminus\{0\}$) such that
\bal
\| \nu\|_{\GTM(\Omega;d(x)^\alpha|x|^{\beta})} &:=\int_{\Omega}d(x)^\alpha|x|^{\beta} \, d|\nu|<+\infty, \\
\big(\text{resp. } \| \nu\|_{\GTM(\Omega\setminus\{0\};d(x)^\alpha|x|^{\beta})} &:=\int_{\Omega\setminus\{0\}}d(x)^\alpha|x|^{\beta} \, d|\nu|<+\infty\big)
\eal
and by $\GTM^+(\Omega;d(x)^\alpha|x|^{\beta})$ (resp. $\GTM^+(\Omega\setminus\{0\};d(x)^\alpha|x|^{\beta})$) its positive cone. }

In order to specify the notion of weak solutions, we first introduce the space of test functions.

\begin{definition} \label{def:weaksol} {\it
	Assume $\xO\subset\BBR^N$ is a bounded domain satisfying the exterior ball condition and containing the origin.  For $b<2s-\tau_+$, we denote by $\mathbf{X}_\xm(\xO;|x|^{-b})$ the space of functions $\psi$ with the following properties:
	
	(i) $\psi\in H^{s}_0(\xO;|x|^{\tau_+});$
	
	(ii) $(-\xD)^s_{\tau_+}\psi $ exists a.e. in $\xO\setminus\{0\}$ and $\displaystyle\sup_{x\in \xO\setminus\{0\}}\big||x|^b (-\xD)^s_{\tau_+}\psi(x)\big| <+\infty$;
	
	(iii) for any compact set $K\subset\xO\setminus\{0\}$, there exist $\xd_0>0$ and $w\in L^1_{\loc}(\xO\setminus \{0\})$ such that
	\bal
	\sup_{0<\xd\leq\xd_0}|(-\Delta)^s_{\tau_+,\xd}\psi|\leq w\;\; \text{a.e. in}\;K
	\eal
	where
	\bal
	(-\Delta)^s_{\tau_+,\delta} \psi(x):=
	C_{N,s} \int_{\R^N\setminus B_{\delta}(x) }\frac{\psi(x)-
		\psi(y)}{|x-y|^{N+2s}} \, |y|^{\tau_+}dy, \quad x \in \Omega \setminus \{0\}.
	\eal }
\end{definition}

The concentration of the measure datum $\tilde \nu$ plays an important role in the study of problem (\ref{eq:g(u)-kdirac-sum}). For any $\tilde \nu \in \GTM(\Omega;d(x)^s|x|^{\alpha})$, we decompose $\tilde \nu = \nu + \ell \delta_0$ where $\nu \in \GTM(\Omega \setminus \{0\};d(x)^s|x|^{\alpha})$, $\ell \in \R$ and $\delta_0$ denotes the Dirac measure at the origin.

\begin{definition} \label{sol:semieq-Dirac-sum} {\it
	Assume  $g: \R \to \R$ is a nondecreasing continuous function such that $g(0)=0$, $\tilde \nu = \nu + \ell \delta_0$ where $\nu \in \GTM(\Omega \setminus \{0\}; |x|^{\tau_+})$ and $\ell \in \R$. A function $u$ is called a \textit{weak  solution} of problem \eqref{eq:g(u)-kdirac-sum} if for any $b<2s-\tau_+$, $u\in L^1(\xO;|x|^{-b}),$ $g(u)\in L^1(\xO;|x|^{\tau_+})$ and
	\ba   \label{sol:g(u)-kdiracdef}
	\int_\xO u(-\xD)^s_{\tau_+}\psi dx+\int_\xO g(u)\psi|x|^{\tau_+} dx= \int_{\Omega \setminus \{0\} } \psi |x|^{\tau_+}d\nu +  \ell \int_{\xO}\xF_{s,\xm}^\xO (-\xD)^s_{\tau_+}\psi dx,\ \, \forall\psi\in \mathbf{X}_\xm(\xO;|x|^{-b}).
	\ea 	}
%
%
%
%
\end{definition}

For $q>1$, we define
\ba\label{subcond}
\Lambda_{g,q}:=\int_1^\infty (g(t)-g(-t))t^{-1-q}dt.
\ea
We also put
\ba \label{psmu*}
p_{s,\mu}^* := \min\left\{ \frac{N}{N-2s}, \frac{N+\tau_+}{-\tau_-}  \right\}.
\ea
We note that $p_{s,\mu}^*$ is the Serrin type critical exponent. If $\mu \geq 0$ then $p_{s,\mu}^*=\frac{N}{N-2s}$, and if $\mu_0\leq \mu<0$ then $p_{s,\mu}^*=\frac{N+\tau_+}{-\tau_-}$.

The first main result deals with the case when the datum $\tilde \nu$ is concentrated away from the origin.
\begin{theorem} \label{existence-semi-1}
	Assume $\mu \geq \mu_0$, $\tilde \nu = \xn\in\GTM(\Omega\setminus\{0\};|x|^{\tau_+})$ and $g\in C(\xO)$ is a nondecreasing function such that $g(0)=0$
	and $\Lambda_{g,\frac{N}{N-2s}}<+\infty$.
	Then there exists a unique weak solution $u_{\nu,0}$ to problem \eqref{eq:g(u)-kdirac-sum}.
	Moreover, for any $b<2s-\tau_+$,
	\ba \label{apriori-1}
	\|  u_{\nu,0} \|_{L^1(\Omega;|x|^{-b})} \leq C(N,\Omega,s,\mu,b)\| \nu \|_{\GTM(\Omega \setminus \{0\};|x|^{\tau_+} )}.
	\ea
\end{theorem}

Let us point out that the weak solution stated in the above theorem is constructed by using an approximation procedure in which the local behavior near the origin of the weak solution to the Poisson problem plays an essential role.  It is worth noting that ODE techniques, which are efficient in the local case, no longer fit our setting well. Instead, we use a different approach based on the adaption of ideas in \cite{DKP} and \cite{KMS}, the assumption that $\supp \nu \cap \{0\} = \emptyset$, and a performance of tedious calculations. The desired local behavior allows us to derive weak Lebesgue estimates, which in turn ensures the convergence of  approximating nonlinearities.   

\medskip

The next result treats the case when $\mu>\mu_0$ and $\tilde \nu$ is concentrated at the origin.

\begin{theorem} \label{semi-dirac>}
Assume $\xm>\xm_0$, $\tilde \nu=\ell \delta_0$ for some $\ell \in\BBR$  and $\Lambda_{g,\frac{N+\tau_+}{\tau_-}}<+\infty$.  Then there exists a unique weak solution $u_{0,\ell}$ of problem \eqref{eq:g(u)-kdirac-sum}.
\end{theorem}

Let us sketch the idea of the proof. Since the datum $\tilde \nu$ is concentrated at the origin, we first construct approximating solutions $u_{\varepsilon}$ in $\Omega \setminus B_{\varepsilon}(0)$ by using the standard theory of monotone operators. Then we establish two-sided estimates on $u_{\varepsilon}$ expressed in term of $\Phi_{s,\mu}^{\Omega}$, which allow us to derive the existence and the asymptotic behavior of a weak solution to problem \eqref{eq:g(u)-kdirac-sum}.  

When $\mu=\mu_0$, a logarithmic correction is involved, which is reflected in the following result. For $q>1$, set
\ba\label{critcond}
\tilde \Lambda_{g,q}:=\int_1^\infty (g(|\ln t|t) - g(-|\ln t|t))t^{-1-q}dt.
\ea
By an analog argument as in the proof of Theorem \ref{semi-dirac>}, we can show that
\begin{theorem} \label{semi-dirac=}
	Assume $\xm=\xm_0$, $\tilde \nu=\ell \delta_0$ for some $\ell \in\BBR,$  and $\tilde \Lambda_{g,\frac{N+2s}{N-2s}}<+\infty$.  Then there exists a unique weak solution $u_{0,\ell}$ of \eqref{eq:g(u)-kdirac-sum}.
\end{theorem}

Next we treat the case where the measures may have support on the whole domain $\Omega$.

\begin{theorem} \label{dirac>-sum}
	Assume $\xm>\xm_0$, $\tilde \nu = \nu + \ell \delta_0$ for some $\nu \in \GTM^+(\Omega \setminus \{0\};|x|^{\tau_+})$, $\ell \geq 0$ and $\Lambda_{g,p_{s,\mu}^*}<+\infty$.  Then there exists a unique weak solution of problem \eqref{eq:g(u)-kdirac-sum}. 

\end{theorem}

The existence of the weak solution stated in Theorem \ref{dirac>-sum} is also based on an approximation procedure, however since $\tilde \nu$ is supported on the whole domain $\Omega$, the analysis is more complicated and requires the treatment near the origin and far from the origin at the same time. We stress that in the proof of Theorem \ref{dirac>-sum}, no $\Delta_2$-condition is needed. Therefore, Theorem \ref{dirac>-sum} does not only extend \cite[Theorem B]{CV} to the local case, but also improve that result in the sense that the  $\Delta_2$-condition can be relaxed.

Finally, in case $\mu=\mu_0$, we also obtain the existence an uniqueness result.

\begin{theorem} \label{dirac=-sum}
Let $\xm=\xm_0$, $\tilde \nu = \nu + \ell \delta_0$ for some $\nu \in \GTM^+(\Omega \setminus \{0\};|x|^{\tau_+})$, $\ell \geq 0$ and $\Lambda_{g,\frac{N}{N-2s}}<\infty$.  Then there exists a unique weak  solution of \eqref{eq:g(u)-kdirac-sum}.
\end{theorem}

\noindent \textbf{Organization of the paper.} The rest of this paper is organized as follows. In Section 2,  we recall basic properties and derive local behavior of solutions to the Poisson problem with a fractional Hardy and a source.  Section 3 is devoted to the study of  weak solution to semilinear problems involving Radon measures supported away from the origin.   In Section 4, we address weak solutions of  semilinear problem involving  Dirac masses concentrated at the origin. Finally, in Section 5, we construct weak solutions in case measures have support in the whole domain $\Omega$.  \medskip

\noindent \textbf{Notation.} Throughout this paper, unless otherwise specified, we assume that $\Omega \subset \R^N$ ($N \geq 2)$ is a bounded domain containing the origin and $d(x)$ is the distance from $x \in \Omega$ to $\R^N \backslash \Omega$. We denote by $c, C, c_1, c_2, \ldots$ positive constants that may vary from one appearance to another and depend only on the data. The notation $c = c(a,b,\ldots)$ indicates the dependence of the constant $c$ on $a,b,\cdots$. For a function $u$, we denote that $u^+=\max\{u,0\}$ and $u^- = \max\{-u,0\}$. For a set $A \subset \R^N$, the function $\1_A$ denotes the indicator function of $A$. \medskip

\noindent \textbf{Acknowledgements.}  {\small \it
	
	H. Chen is supported by NNSF of China, No: 12071189 and 11431005, by the Jiangxi Provincial Natural Science Foundation, No: 20212ACB211005. 
	
	K. T. Gkikas acknowledges the support by the Hellenic Foundation for Research and Innovation (H.F.R.I.) under the “2nd Call for H.F.R.I. Research Projects to support Post-Doctoral Researchers” (Project Number: 59). 
	
	P.-T. Nguyen is supported by Czech Science Foundation, Project GA22-17403S.}

 \section{Preliminary}
\subsection{Basic study of Poisson problem } \label{basic}
In this subsection,  we recall the basic properties of Poisson problem with the fractional
Hardy operator and Radon measures in \cite{CGN}.

For $\xg\in[ \frac{2s-N}{2}, 2s),$ we denote by  $H_0^s(\xO;|x|^\gamma)$ the closure of functions in $C^\infty(\R^N)$ with compact support in $\Omega$ under the norm
\bal 
\norm{u}_{H_0^s(\Omega ;|x|^\gamma)}:= \left(\int_{\R^N}\int_{\R^N}\frac{|u(x)-u(y)|^2}{|x-y|^{N+2s}}|y|^\xg dy |x|^\xg dx\right)^{\frac{1}{2}}.
\eal
Note that $H_0^s(\xO;|x|^\gamma)$ is a Hilbert space with the inner product
\bal 
\langle u,v\rangle_{s,\gamma}:=\int_{\R^N}\int_{\R^N}\frac{\big(u(x)-u(y)\big)\big(v(x)-v(y)\big)}{|x-y|^{N+2s}}|y|^\xg dy |x|^\xg dx. 
\eal


  For $\mu \geq \mu_0$, let  $\mathbf{H}^{s}_{\mu,0}(\xO)$ be the closure of functions in $C^\infty(\R^N)$ with compact support in $\bar \Omega$ under the norm
\bal 
\norm{u}_{\mu}:=\left(\frac{C_{N,s}}{2}\int_{\R^N}\int_{\R^N}\frac{|u(x)-u(y)|^2}{|x-y|^{N+2s}}  dy  dx+\mu\int_\Omega \frac{u^2}{|x|^{2s}}dx\right)^{\frac{1}{2}}.
\eal
This is a Hilbert space with the inner product
\bal 
\ll u,v\gg_{\mu}:=\frac{C_{N,s}}{2}\int_{\R^N}\int_{\R^N}\frac{\big(u(x)-u(y)\big)\big(v(x)-v(y)\big)}{|x-y|^{N+2s}} dy   dx+\mu  \int_\Omega \frac{u(x) v(x)}{|x|^{2s}}dx.
\eal
The next result provides a relation between the above spaces.
 \begin{theorem}[{\cite[Theorem 1.1]{CGN}}] \label{th:main-1} 
Assume $\xO$ is a  bounded Lipschitz domain containing the origin. 

$(i)$ For any $\xg\in[ \frac{2s-N}{2}, 2s)$ and $\mu\geq\mu_0$,  the space  $C_0^\infty(\Omega\setminus\{0\})$ is dense in $\Hsg$ and in $\mathbf{H}_{\mu,0}^s(\Omega)$.

$(ii)$ For any $\xg\in[ \frac{2s-N}{2}, 2s)$, there is $\mu\geq \mu_0$ such that $\tp=\gamma$ and
\ba 
\Hsg =\big\{|x|^{-\gamma}u: u \in \Hm \big\}.
\ea

$(iii)$ 
Let $\gamma\in[\frac{2s-N}{2},2s)$, $\beta<2s$ and $1\leq q<\min\Big\{  \frac{2N-2\beta}{N-2s},\ \frac{2N }{N-2s}\Big\}$. Then there exists a positive constant $c=c(N,\Omega,s,\gamma,\beta,q)$ such that
\ba 
\big\| |\cdot|^{\gamma}v \big\|_{L^q(\Omega;|x|^{-\beta})} \leq c\, \| v \|_{s,\gamma}, \quad \forall v \in \Hsg.
\ea
\end{theorem}

Firstly, a sharp condition for the source is considered for variational solution of the Poisson problem
\ba \label{eq:1-1}
	\left\{
	\BAL
		(-\xD)^s_\gamma u &=f  \quad && \text{in }  \Omega,  \\[0.5mm]
		\qquad\quad u &=0  &&  \text{in } \R^N\setminus \Omega.
	\EAL
	\right.
\ea
Here $u$ is called a variational solution of (\ref{eq:1-1})
 if  $v\in \Hsg$ and
\bal
	\langle v,\xi\rangle_{s,\gamma} =(f,\xi)_\gamma \quad \forall \, \xi\in \Hsg,
\eal
	where
$ 
(f,\xi)_\gamma:=\int_{\xO}f\xi |x|^\gamma dx.
$ 

\begin{theorem}[{\cite[Theorem 1.3]{CGN}}] \label{th:main-2} 
Assume  $\xg\in[ \frac{2s-N}{2}, 2s)$, $\alpha\in\R$ and 
\ba \label{exs ex-0}
p>\max\Big\{  \frac{2N}{N+2s},\, \frac{2N+2\alpha}{N+2s},\ 1+\frac{\alpha}{2s}\Big\}.
\ea
For any $  f\in L^p(\Omega;|x|^{\alpha})$, problem \eqref{eq:1-1} has a unique variational solution $u$. Moreover, there exists a constant $c=c(N,\Omega,s,\gamma,\alpha,p)$ such that 
\ba \label{varisol:est-1}
\norm{u}_{s,\gamma} \leq c\, \| f\|_{L^p(\Omega;|x|^{\alpha})}.
\ea
In addition, the following Kato type inequality holds
\ba\label{kato1}
\langle u^+,\xi \rangle_{s,\gamma}
\leq (f\sign^+(u),\xi)_{\gamma}, \quad \forall \, 0\leq \xi\in \Hsg.
\ea
\end{theorem}

It is easy to check that
for $\alpha<2s$,    \eqref{exs ex-0} could be reduced to
\bal p>\max\Big\{  \frac{2N}{N+2s},\, \frac{2N+2\alpha}{N+2s}\Big\}.
\eal


The Poisson problem involving $\CL_\mu^s$ of the form 
\ba\label{veryweaksolution} \left\{ \BAL
	\CL_\xm^s u&=f&&\quad\text{in}\;\xO, \\
	u&=0&&\quad\text{in}\;\BBR^N\setminus\xO,
	\EAL \right.
	\ea
where $f \in L^1(\xO;d^s|x|^{\tau_+})$, was studied in \cite{CGN} in which the authors introduced a notion of weak solutions in the following sense: A function $u$ is called a weak solution of \eqref{veryweaksolution}
 if for any $b <2s-\tau_+$, $u\in L^1(\xO;|x|^{-b})$  and $u$ satisfies
\bal 
	\int_\xO u(-\xD)^s_{\tau_+}\psi dx=\int_\xO f\psi |x|^{\tau_+} dx,\quad\forall\psi\in \mathbf{X}_\xm(\xO;|x|^{-b}).
\eal

The solvability for \eqref{veryweaksolution} and Kato type inequalities were established in \cite{CGN}.
\begin{theorem}[{\cite[Theorem 1.7]{CGN}}] \label{existence2}
 Assume  $f\in L^1(\xO;d(x)^s|x|^{\tau_+})$. Then problem \eqref{veryweaksolution} admits a unique weak solution $u$.
For any $b<2s-\tau_+$, there exists a positive constant $C=C(N,\Omega,s,\mu,b)$ such that 
	\bal
	\| u \|_{L^1(\Omega;|x|^{-b})} \leq C\| f \|_{L^1(\Omega;d(x)^s|x|^{\tau_+})}.
	\eal
	Furthermore, there holds
	\ba \label{Kato:+-1}
	\int_\xO u^+(-\xD)^s_{\tau_+}\psi dx\leq\int_\xO f\sign^+(u)\psi |x|^{\tau_+} dx,\quad\forall \, 0\leq\psi\in \mathbf{X}_\xm(\xO;|x|^{-b})
	\ea
	and
	\ba \label{Kato||-1}
	\int_\xO |u|(-\xD)^s_{\tau_+}\psi dx\leq\int_\xO f\sign(u)\psi |x|^{\tau_+} dx,\quad\forall \, 0\leq\psi\in \mathbf{X}_\xm(\xO;|x|^{-b}).
	\ea
	As a consequence, the mapping $f \mapsto u$ is nondecreasing. In particular, if $f \geq 0$ then $u \geq 0$ a.e. in $\Omega \setminus \{0\}$.
\end{theorem}	

Note that (\ref{Kato:+-1}) is the Kato's inequality in our weighted distributional sense, which plays a pivotal role in the proof of the
uniqueness for weak solution for our problems involving the Radon measures.

In case of measure data, weak solutions of 
	\ba\label{veryweaksolutionmesure} \left\{ \BAL
	\CL_\xm^s u&=\xn + \ell \delta_0&&\quad\text{in}\;\xO,\\
	u&=0&&\quad\text{in}\;\BBR^N\setminus\xO,
	\EAL \right.
	\ea
are understood in the sense of Definition \ref{sol:semieq-Dirac-sum} with $g\equiv 0$. Existence, uniqueness and a priori estimates of weak solutions are recalled below.
 
\begin{theorem}[{\cite[Theorem 1.10]{CGN}}] \label{existence3-dirac}
Assume $\ell \in \R$ and $\xn\in\mathfrak{M}(\Omega\setminus\{0\};|x|^{\tau_+})$. Then problem \eqref{veryweaksolutionmesure} admits a unique weak solution $u$. For any $b < 2s-\tau_+$, there exists a positive constant $C=C(N,\Omega,s,\mu,b)$ such that
\bal 
	\| u \|_{L^1(\Omega;|x|^{-b})} \leq C\bigg(\| \nu \|_{\GTM(\Omega \setminus \{0\};|x|^{\tau_+} )}+\ell\bigg).
\eal
	Moreover, the mapping $(\nu,\ell) \mapsto u$ is nondecreasing. In particular, if $\nu \geq 0$ and $\ell \geq 0$ then $u \geq 0$ a.e. in $\Omega \setminus \{0\}$.
\end{theorem}

 \subsection{Local behavior at the origin}
Our aim in this subsection is to establish the local behavior of the solution to the  Poisson problem \eqref{veryweaksolutionmesure} when the measure has the support away from the origin. More precisely, the main result of this subsection is stated below. 

 \begin{theorem}\label{est1}
 Assume  $\xO\subset\BBR^N$ is a bounded domain satisfying the exterior ball condition and containing the origin, $\xm_0\leq\xm\leq0$ and $\xn\in\mathfrak{M}(\Omega\setminus\{0\};|x|^{\tau_+})$ with $\dist(\supp|\xn|,\{0\})=r>0$. Let $u$ be the unique solution of \eqref{veryweaksolutionmesure} with $\ell=0$. Then there exists a positive constant $C=C(N,\Omega,s,\mu,r)$ such that
 \ba \label{supu|x|}
\sup_{x\in B_{\frac{r}{4}}(0) \setminus \{0\}}(|x|^{-\tau_+}|u(x)|)\leq C \| \nu \|_{\GTM(\Omega \setminus \{0\};|x|^{\tau_+})}.
\ea
\end{theorem}

The proof of Theorem \ref{est1} consists of several intermediate technical lemmas which allow us to bound local subsolutions to $\CL_\mu^s u \leq 0$ from above in terms of  their tails and mass in neighborhood of the origin. 

The first lemma, which is inspired by \cite[Theorem 1.4]{DKP}, is the following.

\begin{lemma}
Let $x_0\in \xO$ and $r>0$ such that $B_r(x_0)\subset\xO.$ We assume that $v\in H^{s}_0(\xO;|x|^{\tau_+})$ satisfies
\ba\label{subsolution}
\int_{\BBR^N}\int_{\BBR^N}\frac{(v(x)-v(y))(\xf(x)-\xf(y))}{|x-y|^{N+2s}}|y|^{\tau_+}dy |x|^{\tau_+}dx
\leq0, \quad \forall 0 \leq \xf\in C_0^\infty(B_r(x_0)).
\ea
For any $k \in \R$, put $w=v-k$. Then for any nonnegative $\xf\in C_0^\infty(B_r(x_0))$, there holds
\ba\label{cacc}\BAL
&\quad \int_{B_r(x_0)}\int_{B_r(x_0)}\frac{\big(w^+(x)\xf(x)-w^+(y)\xf(y)\big)^2}{|x-y|^{N+2s}}|y|^{\tau_+}dy |x|^{\tau_+}dx\\
&\leq
6 \int_{B_r(x_0)}\int_{B_r(x_0)}\max\{w^+(x)^2,w^+(y)^2\}\frac{(\xf(x)-\xf(y))^2}{|x-y|^{N+2s}}|y|^{\tau_+}dy |x|^{\tau_+}dx\\
&\quad +8\int_{B_r(x_0)}w^+(x)\xf(x)^2|x|^{\tau_+} dx\Big(\sup_{y\in\supp\xf}\int_{\BBR^N\setminus B_r(x_0)}\frac{w^+(x)}{|x-y|^{N+2s}}|x|^{\tau_+} dx\Big).
\EAL
\ea
\end{lemma}
\begin{proof}[\textbf{Proof}]
Let $\phi \in C_0^\infty(B_r(x_0))$ and due to the standard density argument, we can take $w^+\xf^2$ as a test function in \eqref{subsolution} to obtain
\ba \label{test-wphi-1} \BAL
0&\geq \int_{\BBR^N}\int_{\BBR^N}\frac{(v(x)-v(y))(w^+(x)\xf(x)^2-w^+(y)\xf(y)^2)}{|x-y|^{N+2s}}|y|^{\tau_+}dy |x|^{\tau_+}dx\\
&=\int_{B_r(x_0)}\int_{B_r(x_0)}\frac{(v(x)-v(y))(w^+(x)\xf(x)^2-w^+(y)\xf(y)^2)}{|x-y|^{N+2s}}|y|^{\tau_+}dy |x|^{\tau_+}dx\\
&\quad +2\int_{\BBR^N\setminus B_r(x_0)}\int_{B_r(x_0)}\frac{(v(y)-v(x))w^+(y)\xf(y)^2}{|x-y|^{N+2s}}|y|^{\tau_+}dy |x|^{\tau_+}dx.
\EAL \ea
We note that  $(v(y)-v(x))w^+(y) \geq -w^+(x)w^+(y)$ for any $x,y \in \R^N$,  hence
\ba\label{27}\BAL
&\quad \int_{\BBR^N\setminus B_r(x_0)}\int_{B_r(x_0)}\frac{(v(y)-v(x))w^+(y)\xf(y)^2}{|x-y|^{N+2s}}|y|^{\tau_+}dy |x|^{\tau_+}dx\\
&\geq -\int_{\BBR^N\setminus B_r(x_0)}\int_{B_r(x_0)}\frac{w^+(x)w^+(y)\xf(y)^2}{|x-y|^{N+2s}}|y|^{\tau_+}dy |x|^{\tau_+}dx\\
&=-\int_{B_r(x_0)}w^+(y)\xf(y)^2\int_{\BBR^N\setminus B_r(x_0)}\frac{w^+(x)}{|x-y|^{N+2s}}|x|^{\tau_+}dx|y|^{\tau_+}dy\\
&\geq-\int_{B_r(x_0)}w^+(y)\xf(y)^2dy \Big(\sup_{y\in\supp\xf}\int_{\BBR^N\setminus B_r(x_0)}\frac{w^+(x)}{|x-y|^{N+2s}}|x|^{\tau_+} dx\Big).
\EAL
\ea
Next, for any $x,y \in \R^N$, we can show that
\ba \label{27-aa}
(v(x)-v(y))(w^+(x)\xf(x)^2-w^+(y)\xf(y)^2)\geq (w^+(x)-w^+(y))(w^+(x)\xf(x)^2-w^+(y)\xf(y)^2).
\ea
Combining \eqref{test-wphi-1}--\eqref{27-aa} yields
\ba \label{27-a} \BAL
&\int_{B_r(x_0)}\int_{B_r(x_0)}\frac{(w^+(x)-w^+(y))(w^+(x)\xf(x)^2-w^+(y)\xf(y)^2)}{|x-y|^{N+2s}}|y|^{\tau_+}dy |x|^{\tau_+}dx \\ \leq& 2\int_{B_r(x_0)}w^+(y)\xf(y)^2dy\Big(\sup_{y\in\supp\xf}\int_{\BBR^N\setminus B_r(x_0)}\frac{w^+(x)}{|x-y|^{N+2s}}|x|^{\tau_+} dx\Big).
\EAL \ea

Take arbitrary $x,y \in \R^N$. If $\xf(x)\leq\xf(y)$ then we obtain
\bal\BAL
&(w^+(x) -w^+(y))(w^+(x)\xf(x)^2-w^+(y)\xf(y)^2)\\[1mm]
= &(w^+(x)-w^+(y))^2\xf(y)^2+w^+(x)(w^+(x)-w^+(y))(\xf(x)^2-\xf(y)^2)\\
\geq &\frac{1}{2}(w^+(x)-w^+(y))^2\xf(y)^2-2(w^+(x))^2(\xf(x)-\xf(y))^2.
\EAL\eal
If $\xf(x)\geq\xf(y)$ and $w^+(x)\geq w^+(y)$ then
\bal
(w^+(x)&-w^+(y))(w^+(x)\xf(x)^2-w^+(y)\xf(y)^2)\geq (w^+(x)-w^+(y))^2\xf(x)^2.
\eal
Hence, in any case we have that
\bal
(w^+(x) -w^+(y))(w^+(x)\xf(x)^2-w^+(y)\xf(y)^2) &\geq \frac{1}{2}(w^+(x)-w^+(y))^2\max\{\xf(x)^2,\xf(y)^2\}\\[1mm]&\quad -2\max\{(w^+(x))^2,(w^+(y))^2\}(\xf(x)-\xf(y))^2.
\eal
On the other hand, we find that
\bal
(w^+(x)\xf(x) -w^+(y)\xf(y))^2 &\leq 2(w^+(x)-w^+(y))^2\max\{\xf(x)^2,\xf(y)^2\}
\\[1mm]&\quad +2\max\{(w^+(x))^2,(w^+(y))^2\}(\xf(x)-\xf(y))^2.
\eal
Combining the last two displays leads to
\ba\label{26}\BAL
 (w^+(x)\xf(x)-w^+(y)\xf(y))^2
&\leq 4(w^+(x)-w^+(y))(w^+(x)\xf(x)^2-w^+(y)\xf(y)^2)
\\[1mm]&\quad + 6\max\{(w^+)^2(x),(w^+)^2(y)\}(\xf(x)-\xf(y))^2.
\EAL
\ea

Finally, plugging \eqref{27-a} into \eqref{26}, we derive \eqref{cacc}.
\end{proof}

In order to proceed further, we recall that if $\xm_0<\xm$ then for $u\in C^\infty_0(\xO)$
\ba  \label{subcritsobolev0} \BAL
\frac{C_{N,s}}{2}\int_{\BBR^N}\int_{\BBR^N}\frac{|u(x)-u(y)|^2}{|x-y|^{N+2s}}dydx+\xm \int_{\BBR^N}\frac{|u|^2}{|x|^{2s}}dx &\geq\frac{C_{N,s}}{2}\frac{\xm_0-\xm}{\xm_0}\int_{\BBR^N}\int_{\BBR^N}\frac{|u(x)-u(y)|^2}{|x-y|^{N+2s}}dydx\\
&\geq c \left(\int_{\xO}|u|^\frac{2N}{N-2s}dx\right)^{\frac{N-2s}{N}},
\EAL
\ea
where $c=c(N,s,\xm)>0$. Setting $u=|x|^{\tau_+}v$, for any $\xm>\xm_0$, we have that for $u\in C^\infty_0(\xO)$
\ba \label{subcritsobolev1}
\BAL
\frac{C_{N,s}}{2}\int_{\BBR^N}\int_{\BBR^N}\frac{|v(x)-v(y)|^2}{|x-y|^{N+2s}}|y|^{\tau_+}dy |x|^{\tau_+}dx
&\geq c \left(\int_{\xO}(|x|^{\tau_+}|v|)^\frac{2N}{N-2s}dx\right)^{\frac{N-2s}{N}}.
\EAL
\ea

	
	If $\xm=\xm_0$ then by \cite[Theorem 3]{tz}, there exists a positive constant $C=C(N,s)$ such that for   $u\in C^\infty_0(\xO)$,
	\bal
	\frac{C_{N,s}}{2}\int_{\BBR^N}\int_{\BBR^N}\frac{|u(x)-u(y)|^2}{|x-y|^{N+2s}}dydx+\xm_0 \int_{\BBR^N}\frac{|u|^2}{|x|^{2s}}dx\geq C(N,s) \left(\int_{\xO}|u|^\frac{2N}{N-2s}X(|x|)^{\frac{2(N-s)}{N-2s}}dx\right)^{\frac{N-2s}{N}},
	\eal
	where $X(t)=1+\ln (\sup_{x\in\Omega} |x|)-\ln t$. Setting $u=|x|^{\frac{2s-N}{2}}v,$  Theorem \ref{th:main-1} implies that
	\ba \label{critsobolev1}
	\BAL
	\int_{\BBR^N}\int_{\BBR^N}\frac{|v(x)-v(y)|^2}{|x-y|^{N+2s}}|y|^{\frac{2s-N}{2}}dy |x|^{\frac{2s-N}{2}}dx &\geq C(N,s) \left(\int_{\xO}|v|^\frac{2N}{N-s}|x|^{-N}X(|x|)^{\frac{2(N-s)}{N-2s}}dx\right)^{\frac{N-2s}{N}}\\&\geq C(N,s,\xO) \left(\int_{\BBR^N}|v|^\frac{2N}{N-2s}|x|^{\frac{2s-N}{2}}dx\right)^{\frac{N-2s}{N}}.
	\EAL
	\ea
	By Theorem \ref{th:main-1}, inequalities \eqref{subcritsobolev1}-\eqref{critsobolev1} are valid for any $v\in H^{s}_0(\xO;|x|^{\tau_+} )$.

Next, we introduce a notion of the tail  adapted to the context of fractional Hardy potential. For any $v\in H^{s}_0(\xO;|x|^{\tau_+})$, the nonlocal tail of $v$ in the ball $B_r(x_0)$ is defined by
\ba \label{def:tail}
T(v;x_0,r):=r^{2s}\max\{|x_0|,r\}^{-\tau_+}\int_{\BBR^N\setminus B_r(x_0)}\frac{|v||x|^{\tau_+}}{|x-x_0|^{N+2s}}dx.
\ea
\begin{proposition}
Assume  $\xm_0 \leq \mu<0$,   $B_r(x_0)\subset\xO$ and $v\in H^{s}_0(\xO;|x|^{\tau_+})$ satisfy \eqref{subsolution}. We additionally assume that $|x_0|\geq \theta r $ with $\theta\in(1,2)$    if $x_0\neq0$. Then there exists a positive constant $C=C(\xO,\xm,N,s,\xg)$ such that
\ba\label{supest}
\sup_{ B_{\frac{r}{2}}(x_0)}(v-k)^+\leq \xd T((v-k)^+;x_0,\frac{r}{2})+C\xd^{-\frac{N}{4s}}\left(\dashint_{B_r(x_0)}((v-k)^+)^2|x|^{\tau_+} dx\right)^\frac{1}{2}
\ea
for any $\xd\in(0,1]$ and $k\in \BBR.$
\end{proposition}
\begin{proof}[\textbf{Proof}]
The proof, which is based on a combination of the idea in \cite[Theorem 1.1]{DKP} and the fractional Hardy inequality, consists of several steps. \medskip

\noindent \textit{Step 1.}
If $x_0 \neq 0$ then by the assumption  $|x_0|>\theta r$, we obtain
\ba \label{xx_0} \frac{\theta-1}{\theta}|x_0|\leq|x|\leq\frac{\theta+1}{\theta}|x_0|, \quad \forall x \in B_r(x_0).
\ea
Consequently,
\ba \label{Brxtp}
\int_{B_r(x_0)}|x|^{\tau_+} dx\approx r^N|x_0|^{\tau_+}.
\ea
Let $r'<r$, for any $x \in \R^N \setminus B_r(x_0)$ and $y \in B_{r'}(x_0)$, we have
\ba \label{xx0-1}
\frac{|x-x_0|}{|x-y|}\leq 1+\frac{|y-x_0|}{|x-y|}\leq 1+\frac{r'}{r-r'}=\frac{r}{r-r'}.
\ea
 Let $\xz\in C_0^\infty(B_{r'}(x_0))$, then by \eqref{xx_0}, standard fractional Sobolev inequality and \eqref{xx0-1}, we obtain
\ba\label{28}\BAL
\Big(\int_{B_r(x_0)}|\xz|^{\frac{2N}{N-2s}}|x|^{\tau_+} dx\Big)^{\frac{N-2s}{N}}
&\leq C|x_0|^{\frac{(N-2s)\tau_+}{N}}\int_{\BBR^N}\int_{\BBR^N}\frac{|\xz(x)-\xz(y)|^2}{|x-y|^{N+2s}}dy dx\\
&=C|x_0|^{\frac{(N-2s)\tau_+}{N}}\int_{B_r(x_0)}\int_{B_r(x_0)}\frac{|\xz(x)-\xz(y)|^2}{|x-y|^{N+2s}}dy dx\\
&\quad+2C|x_0|^{\frac{(N-2s)\tau_+}{N}}\int_{\BBR^N\setminus B_r(x_0)}\int_{B_r(x_0)}\frac{|\xz(y)|^2}{|x-y|^{N+2s}}dy dx\\
&\leq C|x_0|^{\frac{(N-2s)\tau_+}{N}-2\tau_+}\int_{B_r(x_0)}\int_{B_r(x_0)}\frac{|\xz(x)-\xz(y)|^2}{|x-y|^{N+2s}}|y|^{\tau_+} dy |x|^{\tau_+} dx\\
&\quad+ C|x_0|^{\frac{(N-2s)\tau_+}{N}-\tau_+}\left(\frac{r}{r-r'}\right)^{N+2s}r^{-2s}\int_{B_r(x_0)}|\xz(x)|^2|x|^{\tau_+}dx,
\EAL
\ea
where $C=C(N,s,\mu)$. 

If $x_0=0,$ then by using a scaling argument, \eqref{subcritsobolev1} and \eqref{critsobolev1}, we can show that
\ba\label{29}\BAL
&\left(\int_{B_r(0)}|\xz|^{\frac{2N}{N-2s}}|x|^{\tau_+} dx\right)^{\frac{N-2s}{N}}\\
&\leq C r^{\frac{N-2s}{N}(N + \tau_+)-N+2s-2\tau_+}\int_{\BBR^N}\int_{\BBR^N}\frac{|\xz(x)-\xz(y)|^2}{|x-y|^{N+2s}}|y|^{\tau_+} dy |x|^{\tau_+} dx\\
&\leq C r^{\frac{N-2s}{N}(N + \tau_+)-N+2s-2\tau_+}\int_{B_r(0)}\int_{B_r(0)}\frac{|\xz(x)-\xz(y)|^2}{|x-y|^{N+2s}}|y|^{\tau_+} dy |x|^{\tau_+} dx\\
&\quad +C r^{\frac{N-2s}{N}(N+ \tau_+)-N-\tau_+}\left(\frac{r}{r-r'}\right)^{N+2s}\int_{B_r(0)}|\xz(x)|^2|x|^{\tau_+}dx,
\EAL
\ea
where $C=C(N,s,\mu,\Omega)$. 

For a function $\varphi \in L^1(B_r(x_0))$, set
\bal
\dashint_{B_r(x_0)}\varphi(x) |x|^{\tau_+} dx:=\left(\int_{B_r(x_0)}|x|^{\tau_+}dx\right)^{-1}\int_{B_r(x_0)}\varphi |x|^{\tau_+} dx.
\eal
Then, by \eqref{Brxtp}, \eqref{28} and \eqref{29} can be written as follows
\ba\label{30}\BAL
\left(\dashint_{B_r(x_0)}|\xz|^{\frac{2N}{N-2s}}|x|^{\tau_+} dx\right)^{\frac{N-2s}{N}}
&\leq C r^{2s}\max\{|x_0|,r\}^{-\tau_+}\dashint_{B_r(x_0)}\int_{B_r(x_0)}\frac{|\xz(x)-\xz(y)|^2}{|x-y|^{N+2s}}|y|^{\tau_+} dy |x|^{\tau_+} dx\\
&\quad+ C\left(\frac{r}{r-r'}\right)^{N+2s}\dashint_{B_r(x_0)}|\xz(x)|^2|x|^{\tau_+}dx.
\EAL
\ea

\noindent \textit{Step 2.} For any $j\in \BBN\cup\{0\}$ we set
\bal
r_j=\frac{1}{2}(1+2^{-j})r,\quad\tilde r_j=\frac{r_j+ r_{j+1}}{2},\quad B_j=B_{r_j}(x_0)\quad\text{and}\;\;\tilde B_j=B_{\tilde r_j}(x_0).
\eal
Let $k\in \BBR,$ $\tilde k\in\BBR_+$, put \bal k_j=k+(1-2^{-j})\tilde k \quad \text{and} \quad \tilde k_j=\frac{k_{j}+k_{j+1}}{2}.
\eal
Set
\bal
w_j=(v-k_j)^+ \quad \text{and} \quad
\tilde w_j=(v-\tilde k_j)^+.
\eal

Let $\phi_j \in C_0^\infty(\tilde B_j)$ such that $0 \leq \phi_j \leq 1$ in $\tilde B_j$, $\xf_j=1$ in $B_{j+1}$ and $|\nabla \xf_j|\leq \frac{2^{j+3}}{r}$ in $\tilde B_j$. By \eqref{30}, we have
\ba \label{30-a}\BAL
&\bigg(\dashint_{B_j}|\tilde w_j\xf_j |^{\frac{2N}{N-2s}}|x|^{\tau_+}dx\bigg)^{\frac{N-2s}{N}}\\
&\qquad \leq C r^{2s}\max\{|x_0|,r\}^{-\tau_+}\dashint_{B_j}\int_{B_j}\frac{|\tilde w_j(x)\xf_j(x)-\tilde w_j(y)\xf_j(y)|^2}{|x-y|^{N+2s}}|y|^{\tau_+}dy |x|^{\tau_+}dx\\
&\qquad\qquad+C2^{j(N+2s)}\dashint_{B_j} |\tilde w_j(y)\xf_j(y)|^2|y|^{\tau_+}dy.
 \EAL
\ea
We will estimate the first term on the right hand side of \eqref{30-a} by applying \eqref{cacc} with $B_j$, $\tilde w_j$, $\phi_j$ in place of  $B_r(x_0)$, $w^+$, $\phi$ respectively. Hence, we obtain
\ba\label{30-b}\BAL
&\dashint_{B_j}\int_{B_j}\frac{(\tilde w_j(x)\xf_j(x)-\tilde w_j(y)\xf_j(y))^2}{|x-y|^{N+2s}}|y|^{\tau_+}dy |x|^{\tau_+}dx\\
&\leq
C \dashint_{B_j}\int_{B_j}\max\{\tilde w_j(x)^2, \tilde w_j(y)^2\}\frac{(\xf_j(x)-\xf_j(y))^2}{|x-y|^{N+2s}}|y|^{\tau_+}dy |x|^{\tau_+}dx\\
&\quad +C\dashint_{B_j}\tilde w_j(x)\xf(x)^2|x|^{\tau_+} dx\left(\sup_{y\in\supp\xf_j}\int_{\BBR^N\setminus B_j(x_0)}\frac{\tilde w_j(x)}{|x-y|^{N+2s}}|x|^{\tau_+} dx\right).
\EAL
\ea
We note that for $y \in \supp \phi_j \subset \tilde B_j$ and $x \in \R^N \setminus B_j$,
\ba \label{x-x_0} \frac{|x-x_0|}{|x-y|} \leq 1 + \frac{\tilde r_j}{r_j -\tilde r_j} \leq 2^{j+4}.
\ea
Thanks to the estimates $\tilde w_j\leq \frac{w_j^2}{\tilde k_j-k_j}$, $\tilde w_j \leq w_j$, estimate\eqref{x-x_0} and the definition of nonlocal tail of $w_0$ in $B_{\frac{r}{2}}(x_0)$ in \eqref{def:tail}, we have
\ba\label{31}\BAL
&r^{2s}\max\{|x_0|,r\}^{-\tau_+}\dashint_{B_j}\tilde w_j(x)\xf_j^2|x|^{\tau_+} dx\left(\sup_{y\in\supp\xf_j}\int_{\BBR^N\setminus B_j}\frac{\tilde w_j(x)}{|x-y|^{N+2s}}|x|^{\tau_+} dx\right)\\
&\leq   Cr^{2s}\max\{|x_0|,r\}^{-\tau_+}2^{j(N+2s)}\dashint_{B_j}\frac{w_j^2(x)}{\tilde k_j-k_j}|x|^{\tau_+} dx \int_{\BBR^N\setminus B_j}\frac{ w_j(x)}{|x-x_0|^{N+2s}}|x|^{\tau_+} dx\\
&\leq C\frac{2^{j(N+2s+1)}}{\tilde k}\left(\dashint_{B_j}w_j^2(x)|x|^{\tau_+} dx\right) T(w_0;x_0,\frac{r}{2}).
\EAL
\ea
Next we estimate the first term on the right hand side of \eqref{30-b}. If $x_0 \neq 0$ then by using the estimate  $|\nabla \phi_j| \leq \frac{2^{j+3}}{r}$ in $\tilde B_j$,\eqref{xx_0} and $\tilde w_j \leq w_j$, we obtain
\ba\label{32}\BAL
&r^{2s}|x_0|^{-\tau_+}\dashint_{B_j}\int_{B_j}\max\{\tilde w_j(x)^2,\tilde w_j(y)^2\}\frac{(\xf_j(x)-\xf_j(y))^2}{|x-y|^{N+2s}}|y|^{\tau_+}dy |x|^{\tau_+}dx \\
&\leq C 2^{2j}r^{2s-2}|x_0|^{-\tau_+}\dashint_{B_j}\tilde w_j(x)^2\left(\int_{B_j} |x-y|^{-N-2s+2}|y|^{\tau_+} dy\right)|x|^{\tau_+} dx\\
&\leq  C 2^{2j}r^{2s-2}|x_0|^{-\tau_+} r^{-2s+2} |x_0|^{\tau_+} \dashint_{B_j}|w_j(x)|^2|x|^{\tau_+} dx \\
&= C2^{2j}\dashint_{B_j}|w_j(x)|^2|x|^{\tau_+} dx.
\EAL
\ea

Now let us consider the case $x_0 =0$.

By using the assumption that $\phi_j=1$ in $B_{j+1}\Supset B_{\frac{r}{4}}$, we can split the first term on the right hand side of \eqref{30-b} as
\ba\label{32-b}\BAL
&r^{2s-\tau_+}\dashint_{B_j}\int_{B_j}\max\{\tilde w_j(x)^2,\tilde w_j(y)^2\}\frac{(\xf_j(x)-\xf_j(y))^2}{|x-y|^{N+2s}}|y|^{\tau_+}dy |x|^{\tau_+}dx\\
&=S_jr^{2s-\tau_+} \left(2\int_{B_j \setminus B_{\frac{r}{4}}}\int_{B_{\frac{r}{4}}}\ldots  dx  +\; \int_{B_j \setminus B_{\frac{r}{4}}}\int_{B_j \setminus B_{\frac{r}{4}}}\ldots dx \right) \\
&=:I_{j,1} + I_{j,2},
\EAL \ea
where
\bal S_j: = \left( \int_{B_j}|x|^{\tau_+}dx \right)^{-1}.
\eal

We will estimate $I_{j,1}$. By using the fact that $\phi_j=1$ in $B_{j+1}$, the estimate  $|\nabla \phi_j| \leq \frac{2^{j+3}}{r}$ in $\tilde B_j$ and the estimate $\tilde w_j \leq w_j$, we obtain
\ba \label{32-c-0} \BAL
I_{j,1}&=2S_jr^{2s-\tau_+} \int_{B_j\setminus B_{j+1}}\int_{B_{\frac{r}{4}}}\max\{\tilde w_j(x)^2,\tilde w_j(y)^2\}\frac{(\xf_j(x)-\xf_j(y))^2}{|x-y|^{N+2s}}|y|^{\tau_+}dy |x|^{\tau_+}dx\\
&\leq C 2^{2j} S_j  r^{2s-2-\tau_+}\int_{B_j \setminus B_{j+1}} w_j(x)^2\left(\int_{B_{\frac{r}{4}}} |x-y|^{-N-2s+2}|y|^{\tau_+} dy\right)|x|^{\tau_+} dx\\
&\leq C 2^{2j} S_j \int_{B_j \setminus B_{j+1}} w_j(x)^2|x|^{\tau_+} dx,
\EAL \ea
since $|x-y|\geq \frac{r}{4}$ for any $x\in B_j \setminus B_{j+1}$ and $y\in B_{\frac{r}{4}}.$

By a similar, and simpler, argument as above, we can show that
\ba \label{32-f} \BAL
{ I_{j,2}}
\leq C2^{2j} S_j  \int_{B_j \setminus B_{j+1}}w_j(x)^2|x|^{\tau_+} dx.
\EAL
\ea
From \eqref{32-b}-\eqref{32-f}, we obtain
\ba\label{32-g}\BAL
&r^{2s-\tau_+}\dashint_{B_j}\int_{B_j}\max\{\tilde w_j(x)^2,\tilde w_j(y)^2\}\frac{(\xf_j(x)-\xf_j(y))^2}{|x-y|^{N+2s}}|y|^{\tau_+}dy |x|^{\tau_+}dx\\
&\leq C2^{2j} S_j\dashint_{B_j \setminus B_{j}}w_j(x)^2|x|^{\tau_+} dx.
\EAL \ea
Combining \eqref{30-a}, \eqref{31}, \eqref{32} and \eqref{32-g}, we derive
\ba\label{33}\BAL
&\bigg(\dashint_{B_j}|\tilde w_j\xf_j |^{\frac{2N}{N-2s}}|x|^{\tau_+}dx\bigg)^{\frac{N-2s}{N}} \\
&\leq C 2^{j(N+2s+1)}\left(1+\frac{1}{\tilde k}T(w_0;x_0,\frac{r}{2})\right)\dashint_{B_j}w_j(x)^2 |x|^{\tau_+} dx.
\EAL
\ea

\noindent \textit{Step 3.} Since
\bal
|\tilde w_j|^{\frac{2N}{N-2s}}\geq ( k_{j+1}-\tilde k_j)^{\frac{4s}{N-2s}} w_{ j+1}^2\geq (2^{-j-2}\tilde k)^\frac{4s}{N-2s}w_{j+1}^2,
\eal
it follows that (noticing that $\phi_j=1$ in $B_{j+1}$)
\ba\label{34}\BAL
\bigg(\dashint_{B_j}|\tilde w_j\xf_j |^{\frac{2N}{N-2s}}|x|^{\tau_+}dx\bigg)^{\frac{N-2s}{N}}\geq (2^{-j-2}\tilde k)^\frac{4s}{N}\bigg(\dashint_{B_{j+1}}w_{j+1}^2|x|^{\tau_+}dx\bigg)^{\frac{N-2s}{N}}.
\EAL
\ea

Set
\bal
A_j:=\bigg(\dashint_{B_{j}} w_j ^2|x|^{\tau_+}dx\bigg)^{\frac{1}{2}},
\eal
then by \eqref{34} and \eqref{33}, we deduce

\bal
(2^{-j-2}\tilde k)^\frac{4s}{N} A_{j+1}^{\frac{2(N-2s)}{N}}\leq C 2^{j(N+2s+1)}\left(1+\frac{T(w_0;x_0,\frac{r}{2})}{\tilde k}\right)A_j^2.
\eal

Assuming $\tilde k\geq \xd T(w_0;x_0,\frac{r}{2})$ for some $\xd\in (0,1],$ we obtain
\ba\label{35}
\left(\frac{A_{j+1}}{\tilde k}\right)\leq C 2^{j(\frac{N+2s+1}{2}+\frac{2s}{N})\frac{N}{N-2s}}\xd^{-\frac{N}{2(N-2s)}}\left(\frac{A_j}{\tilde k}\right)^{1+\frac{2s}{N-2s}}.
\ea
We will show by induction that there exists $\xs>1$ such that
\ba \label{Aj<A0}
A_{j}\leq \sigma^{-j} A_0,\quad\forall j\in \BBN\cup\{0\}.
\ea
Obviously is valid for $j=0.$  We assume that it is true for $j=l-1.$ By \eqref{35} we have
\bal
A_{l} &\leq C 2^{l(\frac{N+2s+1}{2}+\frac{2s}{N})\frac{N}{N-2s}}\delta^{-\frac{N}{2(N-2s)}} \left(\frac{A_{l-1}}{\tilde k}\right)^{\frac{2s}{N-2s}} A_{l-1}\\
&\leq C 2^{l(\frac{N+2s+1}{2}+\frac{2s}{N})\frac{N}{N-2s}}\xd^{-\frac{N}{2(N-2s)}}\xs^{(1-l)\frac{N}{N-2s}}\left(\frac{A_{0}}{\tilde k}\right)^{\frac{2s}{N-2s}}
A_{0}\\
&\leq C\xs^{\frac{N}{N-2s}}\xs^{-l} 2^{l(\frac{N+2s+1}{2}+\frac{2s}{N})\frac{N}{N-2s}}\xd^{-\frac{N}{2(N-2s)}}\xs^{-\frac{2s}{N-2s}l}
\left(\frac{A_{0}}{\tilde k}\right)^{\frac{2s}{N-2s}}A_{0}.
\eal
Taking $\xs=2^{(\frac{N+2s+1}{2}+\frac{2s}{N})\frac{N}{2s}}>1$ we obtain
\bal
A_{l} &\leq C\xs^{\frac{N}{N-2s}}\xs^{-l}\xd^{-\frac{N}{2(N-2s)}}\left(\frac{A_{0}}{\tilde k}\right)^{\frac{2s}{N-2s}}A_{0}.
\eal
Hence it is enough to choose $\tilde k$ such that
\ba \label{A0k}
C\xs^{\frac{N}{N-2s}}\xd^{-\frac{N}{2(N-2s)}}\left(\frac{A_{0}}{\tilde k}\right)^{\frac{2s}{N-2s}}\leq1.
\ea
We choose
\bal
\tilde k = \xd T(w_0;x_0,\frac{r}{2})+C^\frac{N-2s}{2s}A_0 \xd^{\frac{-N}{4s}}\xs^{\frac{N}{2s}},
\eal
where $C$ is the constant in \eqref{A0k}, then $A_l \leq \sigma^{-l} A_0$. Then by induction, we deduce \eqref{Aj<A0}.
Letting $j\to\infty$ in \eqref{Aj<A0}, we have that $\lim_{j\to\infty} A_j=0$, namely
\bal
(v-k-\tilde k)^+=0,\;\;a.e.\;\;\text{in}\;\;B_{\frac{r}{2}}(x_0),
\eal
which implies \eqref{supest}. The proof is complete.
\end{proof}


Employing the above result and adapting the idea in \cite[Corollary 2.1]{KMS} to our setting, we obtain the following lemma. 

\begin{lemma}\label{estimateorigin}
Let $\xm_0\leq\xm\leq0,$ $r>0$ be such that $B_r(0)\subset\xO$ and $v\in H^{s}_0(\xO;|x|^{\tau_+})$ satisfy \eqref{subsolution}. Then there exists a positive constant $C=C(\xO,N,s,\xm)$ such that
\bal
\sup_{B_{\frac{r}{2}}}v^+\leq C\left(T(v^+;0,\frac{r}{2})+\dashint_{B_r}v^+(x)|x|^{\tau_+} dx\right).
\eal
\end{lemma}
\begin{proof}[\textbf{Proof}]
Let $\frac{1}{2}\leq t < \gamma \leq1$ and $z\in B_{t r}\setminus B_{\frac{r}{4}}$. We note that $B_{\frac{(\xg-t)r}{100}}(z) \subset B_{\gamma r}$.  Then by applying \eqref{supest} with $B_{\frac{(\xg-t)r}{100}}(z)$ and $k=0$, we obtain
\ba \label{36-a} \BAL
\sup_{x\in B_{\frac{(\xg-t)r}{200}}(z)}v^+&\leq C T(v^+;z,\frac{(\xg-t)r}{200})+C\left(\dashint_{B_{\frac{(\xg-t)r}{100}}(z)}(v^+(x))^2|x|^{\tau_+} dx\right)^\frac{1}{2}\\
&\leq C T(v^+;z,\frac{(\xg-t)r}{200})+\frac{C}{(\xg-t)^{\frac{N}{2}}}\left(\dashint_{B_{\xg r}}(v^+(x))^2|x|^{\tau_+} dx\right)^\frac{1}{2}.
\EAL \ea

Now, since $\frac{|x|}{|x-z|}\leq \frac{C}{\xg-t}$ for any $x\in\BBR^N\setminus B_{\frac{(\xg-t)r}{200}}(z),$  by \eqref{def:tail} and H\"older inequality, we have
\ba\label{36}\BAL
 T(v^+;z,\frac{(\xg-t)r}{200})&\leq C(\xg-t)^{2s}
 r^{2s-\tau_+}\int_{\BBR^N\setminus B_{\frac{(\xg-t)r}{200}}(z)} \frac{v^+}{|x-z|^{N+2s}}|x|^{\tau_+} dx\\
&\leq
 C(\xg-t)^{-N}r^{2s-\tau_+}\int_{(\BBR^N\setminus B_{\frac{(\xg-t)r}{200}}(z))\setminus B_{\frac{r}{2}}} \frac{v^+}{|x|^{N+2s}}|x|^{\tau_+} dx\\ &+C(\xg-t)^{-N}r^{-N-\tau_+}\int_{(\BBR^N\setminus B_{\frac{(\xg-t)r}{200}}(z))\cap B_{\frac{r}{2}}} v^+|x|^{\tau_+} dx\\
&\leq C(\xg-t)^{-N}\bigg(T(v^+;0,\frac{r}{2})+\left(\dashint_{B_{\xg r}}(v^+(x))^2|x|^{\tau_+} dx\right)^\frac{1}{2}\bigg).
\EAL
\ea
Combining \eqref{36-a} and \eqref{36} yields
\ba\label{100}
\sup_{x\in B_{tr}}v^+\leq C(\xg-t)^{-N}\bigg(T(v^+;0,\frac{r}{2})+\left(\dashint_{B_{\xg r}}(v^+(x))^2|x|^{\tau_+} dx\right)^\frac{1}{2}\bigg).
\ea
By Young's inequality, we have
\bal
  (\xg-t)^{-N}\left(\dashint_{B_{\xg r}}(v^+(x))^2|x|^{\tau_+} dx\right)^\frac{1}{2}  
&\leq  (\xg-t)^{-N}\left(\sup_{x\in B_{\xg r}}v^+\dashint_{B_{\xg r}}v^+(x)|x|^{\tau_+} dx\right)^\frac{1}{2}\\
&\leq \frac{1}{2C}\sup_{x\in B_{\xg r}}v^++2C(\xg-t)^{-2N}\dashint_{B_{\xg r}}v^+(x)|x|^{\tau_+} dx,
\eal
where $C$ is the constant in \eqref{100}. Therefore
\ba \label{supv+}
\sup_{x\in B_{tr}}v^+\leq\frac{1}{2}\sup_{x\in B_{\xg r}}v^++2C^2(\xg-t)^{-2N}\left(\dashint_{B_{\xg r}}v^+(x)|x|^{\tau_+} dx+T(v^+;0,\frac{r}{2})\right).
\ea

For any $t\in[\frac{1}{2},1],$ we set $f(t)=\sup_{x\in B_{tr}}v^+$ and
\bal
c_0=2C^2\left(\dashint_{B_{\xg r}}v^+(x)|x|^{\tau_+} dx+T(v^+;0,\frac{r}{2})\right).
\eal
Then \eqref{supv+} can be written as
\bal
f(t)\leq \frac{1}{2}f(\gamma)+\frac{c_0}{(\xg-t)^{2N}}\quad \text{for all } \frac{1}{2}\leq t<\xg\leq1.
\eal
Let $\tau\in (0,1)$ and $\xa=2N.$ We consider the sequence $t_0=t$ and
\bal
t_{i+1}=t_i+(1-\tau)\tau^i(\xg-t)=t_0+(1-\tau)(\xg-t)\sum_{j=0}^i \tau^{j}\to \xg.
\eal
By iteration, we have
\bal
f(t)=f(t_0)=2^{-k}f(t_k)+\frac{c_0}{(1-\tau)^\xa(\xg-t)^\xa}\sum_{i=0}^{k-1}2^{-i}\tau^{-i\xa}.
\eal
Choosing $2^{-\frac{1}{\xa}}<\tau<1$ and letting $k\to\infty$ in the above expression, we obtain
\bal
f(t)\leq C(\xa,\tau)\frac{c_0}{(1-\tau)^\xa(\xg-t)^\xa},
\eal
which yields the desired inequality.
\end{proof}

Now, we turn to

\begin{proof}[\textbf{Proof of Theorem \ref{est1}}]
Let $\{\zeta_{\delta}\}_{\delta>0}$ be the sequence of standard mollifiers and denote $\nu_{1,\delta} = \zeta_{\delta} \ast \nu^+$ and $\nu_{2,\delta} = \zeta_{\delta} \ast \nu^-$. For $\delta>0$,
 $\nu_{i,\delta} \in C_0^\infty(\xO\setminus\{0\})$, $i=1,2$, and
 $\dist(\supp \nu_{1,\delta},\{0\})>\frac{r}{2},$ $\dist(\supp \nu_{2,\delta},\{0\})>\frac{r}{2}$. Moreover,
\ba \label{fngn}
\int_{\xO}\nu_{1,\delta}|x|^{\tau_+} dx\to \int_{\xO\setminus\{0\}}|x|^{\tau_+} d\xn^+\quad\text{and}\quad \int_{\xO} \nu_{2,\delta}|x|^{\tau_+} dx\to \int_{\xO\setminus\{0\}}|x|^{\tau_+} d\xn^-.
\ea
Let $u_\delta$ be the weak solution of
\bal \left\{ \BAL
\CL_\xm^s u&=\nu_{1,\delta}- \nu_{2,\delta}&&\quad\text{in}\;\,\xO,\\
u&=0&&\quad\text{in}\;\,\BBR^N\setminus\xO.
\EAL \right.
\eal
{Since $\Omega$ is a bounded domain satisfying the exterior ball condition and containing the origin, by \cite[Theorem 1.6]{CGN}, for any $b<2s-\tau_+$,} there holds
\ba \label{ud12}
\int_{\xO}|u_\delta||x|^{-b}dx\leq C(N,\xO,\xm,s,b)\int_{\xO}(\nu_{1,\delta}+\nu_{2,\delta})|x|^{\tau_+} dx.
\ea
Furthermore, in view of the proof of {\cite[Theorem 1.7]{CGN}}, $u_\delta \to u$ a.e. in $\xO$   and   in   $L^1(\xO;|x|^b)$. Put $v_\delta=|x|^{-\tau_+}u_\delta$, then $v_\delta\in H^{s}_0(\xO;|x|^{\tau_+})$ and $v_\delta$ satisfies
\bal
\langle v_\delta,\phi \rangle_{s,\tau_+}
=\int_{\xO}\big(\nu_{1,\delta}(x)-\nu_{2,\delta}(x)\big)\xf(x)|x|^{2\tau_+}dx, \quad \forall \xf\in H^{s}_0(\xO;|x|^{\tau_+}).
\eal
Hence,
\bal
\langle v_\delta,\phi \rangle_{s,\tau_+}
=0, \quad \forall \xf\in C_0^\infty(B_{\frac{r}{2}}).
\eal
{ By Lemma \ref{estimateorigin}, the definition of the tail in \eqref{def:tail} and the fact that $v_{\delta}^+ =0$ in $\Omega$, we deduce that 
\bal \BAL
&\quad \sup_{x\in B_{\frac{r}{4}} \setminus \{0\} }v_\delta^+(x)\\  
&\leq C(N,\Omega,s,\mu)\Big(T(v_\delta^+;0,\frac{r}{4})+\dashint_{B_{\frac{r}{2} }}v_\delta^+(x)|x|^{\tau_+} dx\Big) \\
&\leq C(N,\Omega,s,\mu) \Big( 4^{N+2s}r^{-(N+\tau_+)}\int_{\R^N\setminus B_{\frac{r}{4} } }v_\delta^+(x)|x|^{\tau_+} dx +  \big(\int_{B_{\frac{r}{2} }} |x|^{\tau_+}dx \big)^{-1} \int_{B_{\frac{r}{2} }}v_\delta^+(x)|x|^{\tau_+} dx  \Big) \\
&\leq C(N,\Omega,s,\mu,r)\int_{\Omega}v_\delta^+(x)|x|^{\tau_+}dx.
\EAL \eal
Similarly, we can show that
\bal
\sup_{x\in B_{\frac{r}{4}} \setminus \{0\} }v_\delta^-(x) \leq C(N,\Omega,s,\mu,r)\int_{\Omega}v_\delta^-(x)|x|^{\tau_+}dx.
\eal
Adding two preceding estimates and using estimate \eqref{ud12}, for any $b \in [0,2s-\tau_+)$, we obtain 
\bal 
\BAL
\sup_{x\in B_{\frac{r}{4}}(0) \setminus \{0\} }|v_\delta(x)| 
&\leq C(N,\xO,s,\xm,r)\int_{\xO}|v_\delta||x|^{\tau_+} dx \\
&\leq C(N,\xO,s,\xm,r)\int_{\xO}|u_\delta||x|^{-b}dx\   \leq C(N,\Omega,s,\mu,r,b)\int_{\xO}(\nu_{1,\delta}+\nu_{2,\delta})|x|^{\tau_+} dx.
\EAL \eal
By letting $\delta \to 0$ and employing the convergence \eqref{fngn} and the fact that $v_\delta \to |x|^{-\tau_+}u$ a.e. in $\Omega$ as $\delta \to 0$, we derive \eqref{supu|x|}.  }
\end{proof}

\section{The semilinear problem with interior measure data}

\subsection{Estimates in the weighted weak Lebesgue space} We start this subsection by recalling the definition of weak Lebesgue space. 
For $\alpha \in \R$  and $1 \leq q <\infty$, we denote by $L_w^q(\Omega \setminus \{0\};|x|^{\alpha})$ the weighted weak Lebesgue space (or weighted Marcinkiewicz space) defined by
\bal
L_w^q(\Omega \setminus \{0\};|x|^{\alpha} ):= \left\{ u \in L^1_{\loc}(\Omega \setminus \{0\}): \sup_{\lambda > 0} \lambda^q \int_{\{ x \in \Omega \setminus \{0\}: |u(x)| \geq \lambda \} } |x|^{\alpha} dx < +\infty \right\}.
\eal
Denote
\ba \label{semi}
\norm{u}^*_{L_w^q(\Omega \setminus \{0\}; |x|^{\alpha})}:=\left(\sup_{\lambda > 0} \lambda^q \int_{\{ x \in \Omega \setminus \{0\}: |u(x)| \geq \lambda \} } |x|^{\alpha} dx\right)^{\frac{1}{q}}.
\ea
Note that $\norm{\cdot}_{L_w^q(\Omega \setminus \{0\}; |x|^{\alpha})}^*$ is not a norm, but for $\kappa>1$, it is
equivalent to the norm
\bal \norm{u}_{L_w^q(\Omega \setminus \{0\}; |x|^{\alpha})}:=\sup\left\{
\frac{\int_{A}|u||x|^{\alpha}dx}{(\int_A |x|^{\alpha}dx)^{1-\frac{1}{q}}}: A \subset D, \, A \text{
	measurable},\, 0<\int_A |x|^{\alpha}dx <\infty \right\}. \eal
More precisely,
\ba \label{equinorm} \norm{u}_{L_w^q(\Omega \setminus \{0\}; |x|^{\alpha})}^* \leq \norm{u}_{L_w^q(\Omega \setminus \{0\}; |x|^{\alpha})}
\leq \frac{q}{q-1}\norm{u}_{L_w^q(\Omega \setminus \{0\}; |x|^{\alpha})}^*.
\ea
The following strict continuous embedding  hold
\bal
L^q(\Omega \setminus \{0\}; |x|^{\alpha}) \hookrightarrow L_w^q(\Omega \setminus \{0\}; |x|^{\alpha}) \hookrightarrow L^m(\Omega \setminus \{0\}; |x|^{\alpha})
\eal
for $1 \leq m < q < \infty$.

\begin{proposition} \label{Marcin-3}
Assume $\mu \geq \mu_0$ and $\nu \in \GTM(\Omega \setminus \{0\};|x|^{\tau_+})$ such that
\bal
\dist(\supp|\nu|,\{0\})>4r_0
.\eal
Let $u$ be the unique solution of 	\eqref{veryweaksolutionmesure}. Then $u \in L_w^{\frac{N}{N-2s} }(\Omega \setminus \{0\};|x|^{\tau_+})$ and
\ba \label{Marcin-2} \BAL
\| u \|_{L_w^{\frac{N}{N-2s} }(\Omega \setminus \{0\};|x|^{\tau_+})} \leq C(N,\Omega,s,\mu,r_0) \| \nu \|_{\GTM(\Omega \setminus \{0\}; |x|^{\tau_+} )}.
\EAL \ea
\end{proposition}

To prove Proposition \ref{Marcin-3}, we need the following estimates on the level sets of $u$.

\begin{lemma}\label{est2}
Assume $\xm\geq\xm_0$, $\xn\in\mathfrak{M}(\Omega\setminus\{0\};|x|^{\tau_+})$  and  let $u$ be the unique solution of \eqref{veryweaksolutionmesure}.

(i) If $\mu_0\leq \mu<0$ then, for any $\lambda>1$ and any $r>0$ such that $B_{4r}(0)\subset\xO,$ there holds
\ba \label{u>k-1}
|\{x\in\xO\setminus B_r(0):\;|u(x)|>\lambda\}|\leq C(N,\xO,s,\xm,r)\lambda^{-\frac{N}{N-2s}}\| \nu \|_{\GTM(\Omega \setminus \{0\};|x|^{\tau_+})}^{\frac{N}{N-2s} }.
\ea

(ii) If $\mu_0\geq 0$ then, for any $\lambda>1$,
\ba \label{u>k-1+}
|\{x\in\xO:\;|u(x)|>\lambda\}|\leq { C(N,s,\xm)}\lambda^{-\frac{N}{N-2s}}\| \nu \|_{\GTM(\Omega \setminus \{0\};|x|^{\tau_+})}^{\frac{N}{N-2s} }.
\ea
\end{lemma}
\begin{proof}[\textbf{Proof}]
In view of the proof of Theorem \ref{est1}, we may assume that $\xn\in C_0^\infty(\xO).$ Then $v=|x|^{-\tau_+}u \in H^{s}_0(\xO;|x|^{\tau_+}) $ and satisfies
\bal
\langle v,\phi \rangle_{s,\tau_+}
=\int_{\xO}\xf(x)|x|^{\tau_+}\xn(x)dx, \quad  \forall \xf\in C_0^\infty(B_{\frac{r}{2}}(0)).
\eal
Let $\lambda>0,$ taking $v_\lambda=\max\{-\lambda,\min \{v,\lambda\}\}$ as a test function, we have
\bal
\langle v,v_\lambda \rangle_{s,\tau_+}
&=\int_{\xO} v_\lambda(x)|x|^{\tau_+} \xn(x)dx\leq \lambda\int_{\xO}|x|^{\tau_+} |\xn(x)|dx.
\eal
We see that
\bal
(v(x)-v(y))(v_\lambda(x)-v_\lambda(y))\geq (v_\lambda(x)-v_\lambda(y))^2, \quad \forall x,y \in \R^N.
\eal
Hence from the  two proceeding inequalities, we obtain
\ba\label{101}
\langle v_\lambda,v_\lambda \rangle_{s,\tau_+}
\leq \lambda \int_{\xO}|x|^{\tau_+} |\xn(x)|dx.
\ea

(i) Assume $\xm_0\leq\xm<0$. In this case $\frac{2s-N}{2} \leq \tau_+ <0$. Let $r>0$ is small enough such that $B_{4r}(0) \subset \Omega$ and  set $\xl_1=\lambda r^{-\tau_+}$. Then by  \eqref{subcritsobolev1}, \eqref{critsobolev1} and \eqref{101}, we have
\bal
|\{x\in\xO \setminus B_r(0):\;|u(x)|\geq \lambda\}|&\leq |\{x\in\xO\setminus B_r(0):\;|v_{\xl_1}(x)|\geq \xl_1\}\\
&\leq C(N,s,\mu,r)\lambda^{-\frac{2N}{N-2s}}\int_{\xO\setminus B_r(0)}|v_{\xl_1}|^{\frac{2N}{N-2s}} dx\\
&\leq C(N,s,\mu,\xO,r)\lambda^{-\frac{2N}{N-2s}}\int_{\xO}|v_{\xl_1}|^{\frac{2N}{N-2s}} |x|^{\tau_+} dx\\
&\leq C(N,s,\mu,\xO,r)\lambda^{-\frac{2N}{N-2s}}\langle v_{\lambda_1},v_{\lambda_1} \rangle_{s,\tau_+}^{\frac{N}{N-2s}}\\
&\leq C(N,s,\mu,\xO,r)\lambda^{-\frac{N}{N-2s}}\left(\int_{\xO\setminus \{0\}}|x|^{\tau_+}|\xn(x)|dx\right)^{\frac{N}{N-2s}},
\eal
which implies \eqref{u>k-1}.

(ii) Assume  $\xm\geq 0$. In this case $\tau_+ \geq 0$.  Set $\xl_2=(\sup_{x\in\Omega}|x|)^{-\tau_+}\xl$, then by \eqref{subcritsobolev1} and \eqref{101}, we have
\bal
|\{x\in\xO:\;|u(x)|\geq \lambda\}|&\leq C(N,s,\mu)\lambda^{-\frac{2N}{N-2s}}\int_{\{x\in\xO :\;|v(x)|\geq \lambda|x|^{-\tau_+}\}}\big(|x|^{\tau_+}|v|\big)^{\frac{2N}{N-2s}} dx \\
&\leq  C(N,s,\mu)\lambda^{-\frac{2N}{N-2s}}\int_{\xO}\big(|x|^{\tau_+}|v_{\xl_2}|\big)^{\frac{2N}{N-2s}} dx \\
&\leq { C(N,s,\mu)}\lambda^{-\frac{2N}{N-2s}}\langle v_{\lambda_2},v_{\lambda_2} \rangle_{s,\tau_+}^{\frac{N}{N-2s}}\\
&\leq { C(N,s,\mu)}\lambda^{-\frac{N}{N-2s}}\left(\int_{\xO\setminus \{0\}}|x|^{\tau_+}|\xn(x)|dx\right)^{\frac{N}{N-2s}},
\eal
which implies \eqref{u>k-1+}.
\end{proof}

Now we are ready to prove Proposition \ref{Marcin-3}.

\begin{proof}[\textbf{Proof of Proposition \ref{Marcin-3}}] By the linearity, we may assume that $\| \nu \|_{\GTM(\Omega \setminus \{0\};|x|^{\tau_+})}=1$.\smallskip
	
\noindent \textit{Case 1: $\mu_0 \leq \mu <0$}. In this case, $\frac{2s-N}{2}\leq \tau_+<0$. Fix $0<r<r_0$ small.

Since $\tau_+<0$, by Theorem \ref{est1}, we have { (noting that $\| \nu \|_{\GTM(\Omega \setminus \{0\};|x|^{\tau_+})}=1$)
\bal
|u(x)| \leq C_1 |x|^{\tau_+},\quad\forall x \in B_r(0) \setminus \{0\},
\eal
where $C_1=C_1(N,\xO,s,\xm,r)$.
It follows that
\bal
 \{x \in B_r(0) \setminus \{0\}: |u(x)|>\lambda\}\subset B_{R_1}(0)
\eal
where
\bal
R_1=\left(C_1\xl^{-1}\right)^{-\frac{1}{\tau_+}}.
\eal
Hence,
\bal
\int_{\{ x \in B_r(0) \setminus \{0\}: |u(x)|> \lambda \}} |x|^{\tau_+}dx &\leq \int_{B_{R_1}(0)}|x|^{\tau_+}dx =C(N,\Omega,s,\mu,r) \xl^{\frac{N+\tau_+}{\tau_+} }.
\eal}
Since $\tau_+\geq \frac{2s-N}{2}$, we have $-\frac{N+\tau_+}{\tau_+}  \geq \frac{N}{N-2s}$.  From the above estimate, we can easily deduce that, {for any $\lambda>1$},
\ba \label{Br-1}
\int_{\{ x \in B_r(0) \setminus \{0\}: |u(x)|> \lambda \}} |x|^{\tau_+}dx \leq C(N,\Omega,s,\mu,r){\lambda^{-\frac{N}{N-2s} }}.
\ea

Next by (\ref{u>k-1}) and since $\tau_+<0$, we have, {for any $\lambda>1$},
\ba \label{Br-2} \BAL
\int_{\{ x \in \Omega \setminus B_r(0): |u(x)|> \lambda \}} |x|^{\tau_+}dx  \leq r^{\tau_+}\big|\{ x \in \Omega \setminus B_r(0): |u(x)|> \lambda    \}\big|  
 \leq C(N,\xO,s,\xm,r) {\lambda^{-\frac{N}{N-2s} }.}
\EAL \ea
By combining \eqref{Br-1} and \eqref{Br-2}, we obtain,  for any $\lambda>1$, that
\ba  \label{Marcin-2a}
\int_{\{ x \in \Omega \setminus \{0\}: |u(x)|> \lambda \}} |x|^{\tau_+}dx \leq C(N,\xO,s,\xm,r)  \lambda^{-\frac{N}{N-2s} }. 
\ea
Then we can easily show that \eqref{Br-1} holds for any $\lambda>0$.

\noindent \textit{Case 2: $\mu \geq 0$}. In this case $\tau_+ \geq0$. Put $v=|x|^{-\tau_+}u$ and $D_{\Omega}=\sup_{x \in \Omega}|x|^{\tau_+}$, then
\bal
\int_{\{x\in\xO \setminus \{0\}:\;|u(x)|\geq \lambda \}}|x|^{\tau_+} dx\leq \int_{\{x\in\xO \setminus \{0\}:\;D_\xO|v(x)|\geq \lambda \}}|x|^{\tau_+} dx.
\eal
This and \eqref{u>k-1+} imply \eqref{Marcin-2a} for any $\lambda>0$.

Finally, estimate \eqref{Marcin-2} follows from \eqref{Marcin-2a}, \eqref{semi} and  \eqref{equinorm}.
\end{proof}

\subsection{Existence and uniqueness} In this subsection, we aim to establish the existence and uniqueness of the weak solution to problem \eqref{eq 1.0} with $\ell=0$.

First, we show the solvability for semilinear problem involving the dual operator $(-\Delta)_{\gamma}^s$ in the variational framework. 

We assume that
 \begin{itemize}
\item[$({\bf h}_1)$] $h \in C(\R^N \setminus \{0\} \times \R) \cap L^\infty (\R^N \times \R)$;
\item[$({\bf h}_2)$] the map $s\mapsto h(\cdot,s)$ is nondecreasing and $h(\cdot,0)=0$ in $\R^N \setminus \{0\}$.
\end{itemize}
  In the sequel, we will use the notations:
  $H(t)=\int_0^th(s)ds$,
   $(h \circ u)(x) = h(x,u(x))$,  $(H \circ u)(x) = H(x,u(x))$ for $x \in \R^N \setminus \{0\}$.
\begin{definition}
A function $u$ is called a \textit{variational solution} to problem
\ba \label{prob:semi-var}
\left\{
\BAL
(-\xD)^s_\gamma u +h \circ u &=f  \quad && \text{in }  \Omega,  \\[1mm]
\qquad\quad u &=0  &&  \text{in } \R^N\setminus \Omega,
\EAL
\right.
\ea
if $u \in \Hsg$ and
\bal 
\langle u, \xi \rangle_{s,\gamma} + (h \circ u,\xi)_{\gamma} = (f,\xi)_{\gamma}, \quad\forall\, \xi \in \Hsg.
\eal
\end{definition}
\begin{proposition} \label{exist-semi-varsol}
Assume $\xg\in[ \frac{2s-N}{2}, 2s)$,  $\xa\in\R$,     $f \in L^p(\Omega;|x|^{\alpha})$
with $p$ satisfying \eqref{exs ex-0}
 and  $h$ satisfies $({\bf h}_1)$ and $({\bf h}_2)$. Then problem \eqref{prob:semi-var} admits a unique variational solution.
\end{proposition}
\begin{proof}[\textbf{Proof}] Under the assumptions $({\bf h}_1)$ and $({\bf h}_2)$, we consider the functional
	\bal
	\scJ(\varphi) := \frac{1}{2}\| \varphi \|_{s,\gamma}^2 + \int_{\Omega}(H \circ u) |x|^{\gamma} dx  - \int_{\Omega} f u |x|^\gamma dx, \quad \varphi \in \Hsg.
	\eal
	Since $H$ is a nonnegative function, by proceeding similarly as in the proof of Theorem \ref{th:main-2}, we can show that $\scJ$ is coercive and semi-continuous in $\Hsg$. Thus $\scJ$ has a critical point $v \in \Hsg$, which is a variational solution of \eqref{prob:semi-var}. The uniqueness follows from Kato type inequality \eqref{kato1}.
\end{proof}

\begin{theorem}\label{existence4}
Assume $\xm_0\leq\xm$, $\xn_i\in\mathfrak{M}^+(\Omega\setminus\{0\};|x|^{\tau_+})$,  $i=1,2$, and $g\in  C(\R)\cap L^\infty(\R)$ is a nondecreasing function such that $g(0)=0$. Then there exist unique weak solutions $u,u_1,u_2,v_1,v_2$ of the following problems respectively
\ba \label{prob:gu} \left\{ \BAL
\CL_\xm^s u+g(u)&=\xn_1-\xn_2&&\quad\text{in}\;\, \xO,\\
u&=0&&\quad\text{in}\;\,\BBR^N\setminus\xO,
\EAL \right.
\ea
\ba \label{prob:gu1} \left\{ \BAL
\CL_\xm^s u_1+g(u_1)&=\xn_1&&\quad\text{in}\;\, \xO,\\
u_1&=0&&\quad\text{in}\;\, \BBR^N\setminus\xO,
\EAL \right.
\ea
\ba \label{prob:gu2} \left\{ \BAL
\CL_\xm^s u_2-g(-u_2)&=\xn_2&&\quad\text{in}\;\, \xO,\\
u_2&=0&&\quad\text{in}\;\, \BBR^N\setminus\xO,
\EAL \right.
\ea
and
\ba \label{linear-vi} \left\{ \BAL
\CL_\xm^s v_i&=\xn_i&&\quad\text{in}\;\, \xO,\\
v_i&=0&&\quad\text{in}\;\, \BBR^N\setminus\xO,
\EAL \right.
\ea
such that
\ba \label{order-1}
-v_2\leq-u_2\leq u\leq u_1\leq v_1 \quad \text{in } \Omega \setminus \{0\}.
\ea
\end{theorem}
\begin{proof}[\textbf{Proof}]
\noindent {\sc Existence. }   \medskip

\noindent \textbf{Step 1.} First we assume that $\xn_i$, $i=1,2$, has compact support in $\xO\setminus\{0\}$.
Let $\{\xz_\delta\}_{\delta>0}$ be the sequence of standard mollifiers. Put $\xn_{i,\delta}=\xz_\delta*\xn_i$ then there exists an open set $D\Subset\xO\setminus\{0\}$ such that $\xn_{i,\delta} \in C^\infty_0(D)$  for $\delta>0$ small enough. Then
\ba \label{nuntonu}
\int_{\xO}  d(x)^s |x|^{\tau_+}\xn_{\delta} dx\to \int_{\xO\setminus\{0\}}  d(x)^s |x|^{\tau_+}d\xn \quad \text{as } \delta \to 0
\ea
and
\ba \label{nun<nu}
\| \xn_{i,\delta} \|_{L^1(\Omega;|x|^{\tau_+})} \leq c \| \nu_i \|_{\mathfrak{M}(\Omega\setminus\{0\}; d(x)^s |x|^{\tau_+})}, \quad \forall \delta > 0.
\ea

Define the function $h: \R^N \setminus \{0\} \times \R \to \R$ by $h(x,t)=g(|x|^{\tau_+} t)$ for $x \in \R^N \setminus \{0\}$, $t \in \R$. By Proposition \ref{exist-semi-varsol}, there exist variational solutions $\tilde u_{\delta},\tilde u_{1,\delta},\tilde u_{2,\delta},\tilde v_{1,\delta},\tilde v_{2,\delta}\in H_0^s(\Omega;|x|^{\tau_+})$ of the following problems respectively
\bal \left\{  \BAL
(-\Delta)_{\tau_+}^s \tilde u_{\delta}+h \circ \tilde u_{\delta}&=\xn_{1,\delta}-\xn_{2,\delta}&&\quad\text{in}\;\xO,\\
\tilde u_{\delta}&=0&&\quad\text{in}\;\BBR^N\setminus\xO,
\EAL \right.
\eal
\bal \left\{  \BAL
(-\Delta)_{\tau_+}^s \tilde u_{1,\delta}+h \circ \tilde u_{1,\delta}&=\xn_{1,\delta}&&\quad\text{in}\;\xO,\\
\tilde u_{1,\delta}&=0&&\quad\text{in}\;\BBR^N\setminus\xO,
\EAL \right.
\eal
\bal \left\{ \BAL
(-\Delta)_{\tau_+}^s \tilde u_{2,\delta}-h \circ (-\tilde u_{2,\delta})&=\xn_{2,\delta}&&\quad\text{in}\;\xO,\\
\tilde u_{2,\delta}&=0&&\quad\text{in}\;\BBR^N\setminus\xO,
\EAL \right.
\eal
and
\bal \left\{ \BAL
(-\Delta)_{\tau_+}^s \tilde v_{i,\delta}&=\xn_{i,\delta}&&\quad\text{in}\;\xO,\\
\tilde v_{i,\delta}&=0&&\quad\text{in}\;\BBR^N\setminus\xO.
\EAL \right.
\eal
We infer from Kato type inequality \eqref{kato1} that $\tilde u_{i,\delta} \geq 0$, $\tilde v_{i,\delta} \geq 0$, $i=1,2$, and
\bal
-\tilde v_{2,\delta}\leq-\tilde u_{2,\delta}\leq \tilde u_{n}\leq \tilde u_{1,\delta}\leq \tilde v_{1,\delta} \quad \text{a.e. in } \Omega \setminus \{0\}.
\eal
Put $u_\delta = |x|^{\tau_+}\tilde u_\delta$, $u_{i,\delta}=|x|^{\tau_+}\tilde u_\delta$ and $v_{i,n}=|x|^{\tau_+} \tilde v_{i,\delta}$,
then $u_{i,\delta} \geq 0$, $v_{i,\delta} \geq 0$ and
\ba \label{uivi-1}
-v_{2,\delta}\leq-u_{2,\delta}\leq u_{\delta}\leq  u_{1,\delta}\leq  v_{1,\delta} \quad \text{a.e. in } \Omega \setminus \{0\}.
\ea
This implies that
\bal
 |u_\delta| \leq u_{1,\delta} + u_{2,\delta} \leq v_{1,\delta} + v_{2,\delta} \quad \text{a.e. in } \Omega \setminus \{0\}.
\eal
Moreover, since $g(u_\delta(x)) = (
h \circ\tilde u_\delta)(x)$ and $g(u_{i,\delta}(x)) = (h \circ\tilde u_{i,\delta})(x)$, $i=1,2$, it follows that
\bal 
|g(u_\delta)| \leq g(u_{1,\delta}) - g(- u_{2,\delta}) \quad \text{a.e. in } \Omega \setminus \{0\}.
\eal
From Theorem \ref{existence2}, we find that $u_\delta$, $u_{1,\delta}$, $u_{2,\delta}$, $v_{1,\delta}$ and $v_{2,\delta}$ are respectively weak solutions to the following problems
\ba \label{gun} \left\{  \BAL
\CL_\mu^s u_{\delta}+g( u_{\delta})&=\xn_{1,\delta}-\xn_{2,\delta}&&\quad\text{in}\;\xO,\\
 u_{\delta}&=0&&\quad\text{in}\;\BBR^N\setminus\xO,
\EAL \right.
\ea
\ba \label{guin} \left\{  \BAL
\CL_\mu^s u_{1,\delta}+g(u_{1,\delta})&=\xn_{1,\delta}&&\quad\text{in}\;\xO,\\ u_{1,\delta}&=0&&\quad\text{in}\;\BBR^N\setminus\xO,
\EAL \right.
\ea
\ba \label{gvin} \left\{ \BAL
\CL_\mu^s u_{2,\delta}-g(- u_{2,\delta})&=\xn_{2,\delta}&&\quad\text{in}\;\xO,\\
 u_{2,\delta}&=0&&\quad\text{in}\;\BBR^N\setminus\xO,
\EAL \right.
\ea
and
\bal \left\{ \BAL
\CL_\mu^s v_{i,\delta}&=\xn_{i,\delta}&&\quad\text{in}\;\xO,\\
 v_{i,\delta}&=0&&\quad\text{in}\;\BBR^N\setminus\xO.
\EAL \right.
\eal

Since $g \in C(\R) \cap L^\infty(\R)$, we can modify the approximation argument in \cite[Theorem 5.1]{CGN} in order to show that there exist functions $u,u_1,u_2,v_1,v_2$ such that $\{u_\delta\}$, $\{u_{i,\delta}\}$ and $\{v_{i,\delta}\}$ converge to $u$, $u_i$ and $v_i$, $i=1,2$, a.e. in $\Omega \setminus \{0\}$ and in $L^1(\Omega;|x|^{-b})$, for any $b<2s-\tau_+$, respectively  as $\delta \to 0$. From Theorem \ref{existence2}, $v_i$ is the solution to problem \eqref{linear-vi}, $i=1,2$.

Since $g \in C(\R) \cap L^\infty(\R)$, we deduce from the dominated convergence theorem that $\{g(u_\delta)\}$, $\{g(u_{i,\delta})\}$ and $\{g(v_{i,\delta})\}$ convergence to $g(u)$, $g(u_i)$ and $g(v_i)$, $i=1,2$, a.e. in $\Omega \setminus \{0\}$ and in $L^1(\Omega;|x|^{\tau_+})$ respectively as $\delta \to 0$.

Therefore, by passing to the limit in the weak formulation for problems \eqref{gun}--\eqref{gvin}, we derive that
$u$, $u_1$, $u_2$ are solutions to problems \eqref{prob:gu}--\eqref{prob:gu2} respectively.

The uniqueness for problems \eqref{prob:gu}--\eqref{prob:gu2} follows from Kato type inequality \eqref{Kato:+-1} and the monotonicity of $g$. Moreover, from \eqref{uivi-1}, we obtain \eqref{order-1}. \medskip

\noindent \textbf{Step 2.} Next we drop the assumption that $\nu_i$, $i=1,2$, has compact support. Let $\{O_l\}_{l \in \N}$ be a smooth exhaustion of $\Omega \setminus \{0\}$, i.e. smooth open sets $\{O_l\}_{l \in\N}$ such that
\bal
O_l \Subset O_{l+1}\Subset\xO\setminus \{0\}\quad\text{and}\quad\cup_{l \in \N} O_l=\xO\setminus\{0\}.
\eal
Set $\xn_{i,l} = \1_{\overline{O}_l}\xn_i$, $i=1,2$, and let $u_l$, $u_{1,l}$, $u_{2,l}$, $v_{i,l}$ be respectively the nonnegative weak solutions to \eqref{prob:gu}, \eqref{prob:gu1}, \eqref{prob:gu2} and \eqref{linear-vi} with $\nu_i$ replaced by $\nu_{i,l}$, $i=1,2$. Then we have
\ba  \label{u12lvl}
-v_{2,l} \leq -u_{2,l} \leq u_{l} \leq u_{1,l} \leq v_{1,l} \quad \text{a.e. in } \Omega \setminus \{0\}.
\ea

Let $b<2s-\tau_+$, then $u_l \in L^1(\Omega;|x|^{-b})$ and $u_l$ satisfies the weak formulation
\ba \label{ulsol-a} \BAL
\int_\xO u_l (-\xD)^s_{\tau_+}\psi dx + \int_{\Omega}g(u_{l}) \psi |x|^{\tau_+}dx = \int_{\xO\setminus\{0\}}\psi |x|^{\tau_+} d(\xn_{1,l}- \xn_{2,l})
\EAL \ea
for all $0 \leq \psi\in \mathbf{X}_\xm(\xO;|x|^{-b})$.
We know from Step 1 that $u_l$ is the limit of the sequence of $\{u_{l,\delta}\}_{\delta > 0}$, where $u_{l,\delta}$ is the weak solution of
\bal 
\left\{  \BAL
\CL_\mu^s u+g( u)&=\xn_{1,l,\delta}-\xn_{2,l,\delta}&&\quad\text{in}\;\xO,\\
u_{\delta}&=0&&\quad\text{in}\;\BBR^N\setminus\xO,
\EAL \right.
\eal
where $\xn_{i,l,\delta}:=\zeta_\delta \ast \nu_{i,l}$, $i=1,2$. For $l > l'$, since $\nu_{i,l} \geq \nu_{i,l'}$, it follows that $\nu_{i,l,\delta} \geq \nu_{i,l',\delta}$ for $i=1,2$ and any $\delta>0$.

Let $\xi_b$ be the solution of
\ba \label{xib} \left\{  \BAL
\CL_\mu^s \xi &= |x|^{-b} \quad &&\text{in } \Omega, \\
\xi &=0 \quad &&\text{in }  \R^N \setminus \Omega.
\EAL \right. \ea
Then $\xi_b \in \mathbf{X}_\xm(\xO;|x|^{-b})$ and   $|\xi_b| \leq Cd^s$ a.e. in $\Omega$ (see \cite[estimate (1.16)]{CGN}).

By using Kato type inequality \eqref{Kato||-1}  for  $u_{l,\delta}-u_{l',\delta}$ and  $\psi=\xi_b$ as the test function,  we obtain that
\ba \label{ulsol-aa} \BAL
&\int_\xO |u_{l,\delta}-u_{l',\delta}||x|^{-b} dx + \int_{\Omega}|g(u_{l,\delta}) - g(u_{l',\delta})|\xi_b |x|^{\tau_+}dx \\
&\leq  \int_{\xO\setminus\{0\}}\xi_b \sign(u_{l,\delta} - u_{l',\delta}) |x|^{\tau_+} [(\xn_{1,l,\delta}- \nu_{1,l',\delta}) - (\xn_{2,l,\delta}-\nu_{2,l',\delta})]dx \\
&\leq  C\int_{\xO\setminus\{0\}}  |x|^{\tau_+} (\xn_{1,l,\delta}- \nu_{1,l',\delta})dx  +  C\int_{\xO\setminus\{0\}}  |x|^{\tau_+} (\xn_{2,l,\delta}-\nu_{2,l',\delta})dx.
\EAL \ea
We note that $u_{l,\delta}-u_{l',\delta} \to u_{l}-u_{l'}$ in $L^1(\Omega;|x|^{-b})$,  $g(u_{l,\delta}) - g(u_{l',\delta}) \to g(u_{l}) - g(u_{l'})$ in $L^1(\Omega;|x|^{\tau_+})$ as $\delta \to 0$ and for any $l>0$,
\bal
\int_{\Omega \setminus \{0\}}|x|^{\tau_+}\nu_{i,l,\delta}dx
\to \int_{\Omega \setminus \{0\}}|x|^{\tau_+}d\nu_{i,l} \quad \text{as } \delta \to 0.
\eal
Therefore, by letting $\delta \to 0$ in \eqref{ulsol-aa}, we find
\ba \label{ulsol-ab}
\| u_l - u_{l'} \|_{L^1(\Omega;|x|^{-b})} \leq C( \| \xn_{1,l}- \xn_{1,l'}\|_{\GTM(\Omega \setminus \{0\};|x|^{\tau_+})} + \|\nu_{2,l}-\nu_{2,l'}  \|_{\GTM(\Omega \setminus \{0\};|x|^{\tau_+})}).
\ea
Since $\nu_{i,l} \uparrow \nu_{i}$ as $l \to \infty$, we infer from \eqref{ulsol-ab} that $\{u_l\}_{l \in \N}$ is a Cauchy sequence in $L^1(\Omega;|x|^{-b})$. Therefore there exists $u \in L^1(\Omega;|x|^{-b})$ such that  $u_l \to u$ a.e. in $\Omega \setminus \{0\}$ and in $L^1(\Omega;|x|^{-b})$. Since $g \in C(\R) \cap L^\infty(\R)$, by the dominated convergence theorem, we deduce that $g(u_l) \to g(u)$ a.e. in $\Omega \setminus \{0\}$ and in $L^1(\Omega;|x|^{\tau_+})$. Thus by letting $l \to \infty$ in \eqref{ulsol-a}, we derive that $u$ is a weak solution of \eqref{prob:gu}.

By a similar argument, we can show that $\{u_{i,l}\}_{l \in \N}$ and $\{v_{i,l}\}_{l \in \N}$ converge to $u_{i}$ and $v_i$ in $L^1(\Omega;|x|^{-b})$ respectively as $l \to \infty$ and $\{g(u_{i,l})\}_{l \in \N}$ converges to $g(u_i)$, $i=1,2$, in $L^1(\Omega;|x|^{\tau_+})$ as $l \to \infty$. Moreover, $u_1,u_2,v_i$ are solutions of problems \eqref{prob:gu1}, \eqref{prob:gu2} and \eqref{linear-vi}.

Estimate \eqref{order-1} follows from \eqref{u12lvl}.
 \medskip

\noindent {\sc Uniqueness.} The uniqueness for problem \eqref{linear-vi} was established in Theorem \ref{existence2}. The uniqueness for problems \eqref{prob:gu}--\eqref{prob:gu2} follows from Kato's inequality \eqref{Kato:+-1} and the monotonicity of $g$.
\end{proof}

\begin{lemma} \label{subcrcon} Assume
	\ba \label{subcd0} \int_1^\infty  t^{-q-1}(\ln t)^{m} (g(t)-g(-t)) dt<\infty
	\ea
	for $q,m \in \R$, $q >0$ and $m \geq 0$. Let $v$ be a function defined in $\Omega \setminus \Sigma$. For $t>0$ and $G \subseteq \Omega \setminus \{0\}$, set
	\bal E_{G}(t):=\{x\in \xO\setminus \{0\}:| v(x)|>t\} \quad \text{and} \quad e_{G}(t):=\int_{E_{G}(t)} |x|^{\tau_+} dx.
	\eal
	Assume that there exists a positive constant $C_0$ such that
	\ba \label{e}
	e_{G}(t) \leq C_0t^{-q}(\ln s)^m, \quad \forall t>e^\frac{2 m}{q}.
	\ea
	Then for any $t_0>e^\frac{2 m}{q},$ there holds
	\ba\label{53}
	\int_{G}g(|v|)|x|^{\tau_+}dx &\leq g(t_0)\int_{G} |x|^{\tau_+} dx+ C_0 q\int_{t_0}^\infty t^{-q-1}(\ln t)^{m} g(t) dt, \\ \label{53-a}
	-\int_{G}g(-|v|)|x|^{\tau_+}dx&\leq  - g(-t_0)\int_{G} |x|^{\tau_+} dx-C_0 q\int_{t_0}^\infty  t^{-q-1}(\ln t)^{m} g(-t) dt.
	\ea
\end{lemma}
\begin{proof}[\textbf{Proof}]
	We note that $g(|v|) \geq g(0)=0$. Let $G \subseteq \Omega \setminus \{0\}$ and $t_0>1$ to be determined later on. Using the fact that $g$ is nondecreasing, we obtain
	\bal \begin{aligned}
		\int_{G} g(|v|) |x|^{\tau_+} dx &\leq \int_{G  \setminus E_{t_0}(v)} g(|v|)|x|^{\tau_+} dx + \int_{E_{t_0}(v)} g(|v|)|x|^{\tau_+} dx \\
		&\leq g(t_0)e(t) -\int_{t_0}^{\infty}g(t) de(t).
	\end{aligned} \eal
	From \eqref{subcd0}, we deduce that there exists an increasing sequence $\{T_n\}$ such that
	\ba \label{Tn}
	\lim_{T_n \to \infty}T_n^{-q}(\ln T_n)^m g(T_n) = 0.
	\ea
	For $T_n > t_0$, we have
	\bal 
	\begin{aligned}
		-\int_{t_0}^{T_n}g(t)de(t) &= -g(T_n)e(T_n) + g(t_0)e(t_0) + \int_{t_0}^{T_n} e(t) dg(t) \\
		&\leq -g(T_n)e(T_n) + g(t_0)e(t_0)+C_0\int_{t_0}^{T_n}t^{-q}(\ln t)^m dg(t)\\
		&\leq (CT_n^{-q}(\ln T_n)^m - e(T_n))g(T_n)   - C_0\int_{t_0}^{T_n} (t^{-q}(\ln t)^m)' g(t) dt.
	\end{aligned} 
\eal
	Here in the last estimate, we have used \eqref{e}.
	Note that if we choose $t_0>e^{\frac{2 m}{q}}$ then
	\ba \label{deriv} -q t^{-q-1}(\ln t)^m  < (t^{-q}(\ln t)^m)' < -\frac{q}{2}t^{-q-1}(\ln t)^m \quad \forall t \geq t_0.
	\ea
	Combining \eqref{Tn}--\eqref{deriv} and then letting $n \to \infty$, we obtain
	\bal
	-\int_{t_0}^{\infty}g(t) de(t) < C_0q \int_{t_0}^{\infty} t^{-q-1}(\ln t)^m g(t) dt.
	\eal
	Thus we have proved estimate \eqref{53}. By  applying estimate \eqref{53} with $g$ replaced by $\tilde g(t)=-g(-t)$, we obtain \eqref{53-a}.
\end{proof}

\begin{lemma} \label{lem:equi}
Assume that $\{g_n\}_{n \in \N}$ is a sequence of nondecreasing continuous functions vanishing at $0$ such that
\ba \label{gk} \BAL
&g_n(t) - g_n(-t) \leq C_1(t), \quad \forall t \geq 1, \, \forall n \in \N, \\
&\int_{1}^\infty (g_n(t)-g_n(-t))t^{-q-1}dt <  C_2, \quad \forall n \in \N,
\EAL \ea
for some $q>1$ and positive constants $C_1,C_2$ independent of $n$. Let $\{v_n\}_{n \in \N}$ be a uniformly bounded in $L_w^q(\Omega;|x|^{\tau_+})$. Then the sequence $\{g_n(v_n)\}_{n \in \N}$ is uniformly bounded and equi-integrable in $L^1(\Omega;|x|^{\tau_+})$.	
\end{lemma}
\begin{proof}[\textbf{Proof}] By the assumption, there exists a positive constant $C$ such that
\bal
\| v_n \|_{L_w^q(\Omega \setminus \{0\};|x|^{\tau_+})} \leq C, \quad \forall n \in \N,
\eal
which implies that
\bal
\int_{ \{x \in \Omega \setminus \{0\}: |v_n(x)|> t  \} }|x|^{\tau_+}dx \leq t^{-q}\| v_n \|_{L_w^q(\Omega \setminus \{0\};|x|^{\tau_+})} \leq Ct^{-q}, \quad \forall t>0.
\eal
Take arbitrarily $t_0 \geq 1$ and $G \subseteq \Omega \setminus \{0\}$. By Lemma \ref{subcrcon}, we obtain
\ba \label{uni-equi} \BAL
\int_{G}  |g_n(v_n)||x|^{\tau_+}dx &\leq \int_{\Omega \setminus \{0\}} (g_n(|v_n|)-g_n(-|v_n|))|x|^{\tau_+}dx \\
& \leq (g_n(t_0)-g_n(-t_0))\int_{G} |x|^{\tau_+} dx+ Cq\int_{t_0}^\infty  t^{-q-1}(g_n(t)-g_n(-t)) dt \\
&\leq C(N,\Omega,s,\mu,q,C_1(t_0),C_2).
\EAL \ea
Here in the last estimate we have used the assumption \eqref{gk}. Therefore, by taking $t_0=1$ and $G=\Omega \setminus \{0\}$ in \eqref{uni-equi}, we derive that $\{g_n(v_n)\}$ is uniformly bounded in $L^1(\Omega;|x|^{\tau_+})$. From \eqref{uni-equi}, we can also deduce that $\{g_n(v_n)\}$ is equi-integrable in $L^1(\Omega;|x|^{\tau_+})$.
\end{proof}

\begin{proof}[\textbf{Proof of Theorem \ref{existence-semi-1}}]
Set $g_n:=\max\{-n,\min\{g,n\}\}$ with $n\in\N$ and  let $\{\zeta_{\delta}\}_{\delta>0}$ be the sequence of standard mollifiers.
Let $r_0>0$ small  such that $B_{4r_0}(0)\subset\xO$. For $\delta< \frac{1}{8}r <\frac{1}{8} r_0$, put
\bal \xn_{r,\delta}=\zeta_{\delta} \ast (\1_{\xO\setminus B_r(0)}\xn) , \quad  \xn_{1,r,\delta}=\xn_{r,\delta}^+, \quad  \xn_{2,r,\delta}=\xn_{r,\delta}^-.
\eal
Then $\{ \nu_{r,\delta}\}$, $\{ \nu_{1,r,\delta}\}$, {\color{blue}$\{ \nu_{2,r, \delta}\}$} converge weakly to $\1_{\xO\setminus B_r(0)}\nu$, $\1_{\xO\setminus B_r(0)}\nu^+$, $\1_{\xO\setminus B_r(0)}\nu^-$ in $\GTM(\Omega\setminus \{0\};|x|^{\tau_+})$ respectively.
By Theorem \ref{existence4}, there exist weak solutions
$u_{r,\delta,n}$, $u_{1,r,\delta,n}$, $u_{2,r,\delta,n}$, $v_{1,r,\delta}$, $v_{2,r,\delta}$ of the following problems respectively
\bal 
\left\{ \BAL
\CL_\xm^s u+g_n(u)&=\xn_{r,\delta}&&\quad\text{in}\;\xO,\\
u & = 0 &&\quad\text{in}\;\BBR^N\setminus\xO,
\EAL \right.
\eal
\ba \label{gknu1e+} \left\{ \BAL
\CL_\xm^s u +g_n(u)&=\nu_{1,r,\delta}&&\quad\text{in}\;\xO,\\
u&=0&&\quad\text{in}\;\BBR^N\setminus\xO,
\EAL \right.
\ea
\ba \label{gknu2e+} \left\{ \BAL
\CL_\xm^s u -g_n(-u)&=\xn_{2,r,\delta}&&\quad\text{in}\;\xO,\\
u & = 0 &&\quad\text{in}\;\BBR^N\setminus\xO,
\EAL \right.
\ea
and
\bal 
\left\{ \BAL
\CL_\xm^s v &=\xn_{i,r,\delta}&&\quad\text{in}\;\xO,\\
v & = 0 &&\quad\text{in}\;\BBR^N\setminus\xO.
\EAL \right.
\eal
Moreover, for any $b<2s-\tau_+$, $u_{r,\delta,n}$, $u_{1,r,\delta,n}$, $u_{2,r,\delta,n}$, $v_{1,r,\delta}$, $v_{2,r,\delta} \in L^1(\Omega;|x|^{-b})$. We also have $u_{i,r,\delta,n} \geq 0$, $v_{i,r,\delta} \geq 0$ and
\ba \label{order-2}
-v_{2,r,\delta}\leq-u_{2,r,\delta,n}\leq u_{r,\delta,n}\leq u_{1,r,\delta,n}\leq v_{1,r,\delta}.
\ea

Moreover, since $g_n$ is nondecreasing with $g_n(0)=0$, we deduce that $\{u_{i,r,\delta,n}\}$, $i=1,2$, is nonincreasing with respect to $n$. Therefore, there exist positive functions $u_{i,r,\delta}$, $i=1,2$, such that $u_{i,r,\delta,n} \downarrow u_{i,r,\delta}$ a.e. in $\Omega \setminus \{0\}$ and in $L^1(\Omega;|x|^{-b})$ as $n \to +\infty$ for any $b<2s-\tau_+$. Using Dini's lemma with respect to the sequence $\{g_n\}_{n \in \N}$ in compact sets of $(0,+\infty)$, we deduce that $g_n(u_{i,r,\delta,n}) \to g(u_{i,r,\delta})$ a.e. in $\Omega $ as $n \to +\infty$. Since $0 \leq u_{i,r,\delta,n} \leq v_{i,r,\delta} \in L_w^{\frac{N}{N-2s}}(\Omega \setminus \{0\};|x|^{\tau_+})$ (due to Proposition \ref{Marcin-3}), it follows that $\{u_{i,r,\delta,n} \}$ is uniform bounded in $L_w^{\frac{N}{N-2s}}(\Omega \setminus \{0\};|x|^{\tau_+})$ with respect to $n$. From Lemma \ref{lem:equi}, we infer that the sequences $\{g_n(u_{i,r,\delta,n})\}_{n \in \N}$, $i=1,2$, are uniformly bounded and equi-integrable in  $L^1(\Omega;|x|^{\tau_+})$ with respect to $n$. By Vitali's convergence theorem, $g_n(u_{i,r,\delta,n}) \to g(u_{i,r,\delta})$ in $L^1(\Omega;|x|^{\tau_+})$ as $n \to +\infty$. By letting $n \to +\infty$ in the weak formulation for problem \eqref{gknu1e+} and \eqref{gknu2e+}, we derive that $u_{1,r,\delta}$ and $u_{2,r,\delta}$ are respectively weak solutions to the problems
\ba \label{gknu1e} \left\{ \BAL
\CL_\xm^s u +g(u)&=\nu_{1,r,\delta}&&\quad\text{in}\;\xO,\\
u&=0&&\quad\text{in}\;\BBR^N\setminus\xO,
\EAL \right.
\ea
\ba \label{gknu2e} \left\{ \BAL
\CL_\xm^s u -g(-u)&=\xn_{2,r,\delta}&&\quad\text{in}\;\xO,\\
u & = 0 &&\quad\text{in}\;\BBR^N\setminus\xO.
\EAL \right.
\ea

From \eqref{order-2}, we deduce that
\ba \label{urkebound-1} \BAL
|u_{r,\delta,n}| &\leq u_{1,r,\delta,n} + u_{2,r,\delta,n} \leq v_{1,r,\delta} + v_{2,r,\delta}, \\
|g_n(u_{r,\delta,n})| &\leq g_n(u_{1,r,\delta,n}) - g_n(-u_{2,r,\delta,n}),
\EAL \ea
which in turn implies that $\{u_{r,\delta,n}\}_{n \in \N}$ is uniformly bounded in $L^1(\Omega;|x|^{-b})$ for any $b<2s-\tau_+$, as well as $\{g_n(u_{i,r,\delta,n})\}_{n \in \N}$ is uniformly bounded in $L^1(\Omega;|x|^{\tau_+})$.

By using a similar argument as in the proof of Theorem \ref{existence4}, we derive that $\{u_{r,\delta,n}\}$ converges to a function $u_{r,\delta}$ a.e. in $\Omega \setminus \{0\}$ and in $L^1(\Omega;|x|^{-b})$ as $n \to +\infty$.
We note that $g_n^+ =\min\{g^+,n\}$ and $g_n^- = \min\{g^-,n\}$, hence $\{g_n^+\}$ and $\{g_n^-\}$ are nondecreasing sequences of bounded continuous nonnegative functions and $g_n^{\pm} \uparrow g^{\pm}$ as $n \to \infty$. Using Dini's lemma with respect to $\{g_n^\pm\}$ in compact sets of $(0,+\infty)$, we deduce that $g_n^{\pm}(u_{r,\delta,n}) \to g^{\pm}(u_{r,\delta})$ a.e. in $\Omega \setminus \{0\}$ as $n \to +\infty$, consequently, $g_n(u_{r,\delta,n}) \to g(u_{r,\delta})$ a.e. in $\Omega$. By \eqref{urkebound-1} and the generalized dominated convergence theorem, we deduce that  $g_n(u_{r,\delta,n}) \to g(u_{r,\delta})$ in $L^1(\Omega;|x|^{\tau_+})$ as $n \to +\infty$. Passing to the limit, we deduce that $u_{r,\delta}$ is a weak solution of
\ba \label{gknue} \left\{ \BAL
\CL_\xm^s u+g(u)&=\xn_{r,\delta}&&\quad\text{in}\;\xO,\\
u & = 0 &&\quad\text{in}\;\BBR^N\setminus\xO.
\EAL \right.
\ea
From \eqref{order-2} and \eqref{urkebound-1}, we deduce that $u_{i,r,\delta} \geq 0$ and
\bal 
 \BAL
  &-v_{2,r,\delta}\leq-u_{2,r,\delta}\leq u_{r,\delta}\leq u_{1,r,\delta}\leq v_{1,r,\delta}, \\
&|u_{r,\delta}| \leq u_{1,r,\delta} + u_{2,r,\delta} \leq  v_{1,r,\delta} + v_{2,r,\delta}, \\
&|g(u_{r,\delta})| \leq g(u_{1,r,\delta}) - g(-u_{2,r,\delta}).
\EAL \eal
By Theorem \ref{existence2}, for any $b<2s-\tau_+$, we have
\bal
\| v_{i,r,\delta} \|_{L^1(\Omega;|x|^{-b})} &\leq C(N,\Omega,s,\mu,b)\| \nu_{i,r,\delta} \|_{L^1(\Omega;|x|^{\tau_+} )} \leq C(N,\Omega,s,\mu,b,r) \| \nu_{i,r} \|_{\GTM(\Omega \setminus \{0\};|x|^{\tau_+} )}.
\eal
Therefore
\ba \label{apriori-2}
\| u_{r,\delta} \|_{L^1(\Omega;|x|^{-b})} + \sum_{i=1,2}\| u_{i,r,\delta} \|_{L^1(\Omega;|x|^{-b})}
\leq C(N,\Omega,s,\mu,b,r) \| \nu_{i,r} \|_{\GTM(\Omega \setminus \{0\};|x|^{\tau_+} )}.
\ea
As above, we can show that the sequences $\{g(u_{r,\delta})\}_{\delta}$, $\{g(u_{i,r,\delta})\}_{\delta}$, $i=1,2$, are uniformly bounded and equi-integrable in  $L^1(\Omega;|x|^{\tau_+})$ with respect to $\delta$. Again, by a similar argument as in the proof of Theorem \ref{existence4}, we deduce that there exist functions $u_r$, $u_{1,r}$, $u_{2,r} $ such that $\{u_{r,\delta}\}_{\delta>0}$, $\{u_{1,r,\delta}\}_{\delta>0}$, $\{u_{2,r,\delta}\}_{\delta>0}$ converge to $u_r$, $u_{1,r}$ and $u_{2,r}$ a.e. in $\Omega \setminus \{0\}$ and in $L^1(\Omega;|x|^{-b})$ for any $b<2s-\tau_+$ respectively as $\delta \to 0$. This implies that $\{g(u_{r,\delta})\}_{\delta>0}$, $\{g(u_{1,r,\delta})\}_{\delta>0}$, $\{g(u_{2,r,\delta})\}_{\delta>0}$ converge to $g(u_r)$, $g(u_{1,r})$ and $g(u_{2,r})$ a.e. in $\Omega \setminus \{0\}$ as $\delta \to 0$ respectively. Employing an analogous argument as above, we can show that $\{g(u_{2,r,\delta})\}_{\delta>0}$ converge to $g(u_r)$, $g(u_{1,r})$ and $g(u_{2,r})$ in $L^1(\Omega;|x|^{\tau_+})$ as $\delta \to 0$.
By passing to the limit in the weak formulation for problems \eqref{gknue}, \eqref{gknu1e}, \eqref{gknu2e}, we deduce that $u_r$, $u_{1,r}$, $u_{2,r}$ are respectively weak solutions to the following problems
\bal 
\left\{ \BAL
\CL_\xm^s u+g(u)&=\xn_{r}&&\quad\text{in}\;\xO,\\
u & = 0 &&\quad\text{in}\;\BBR^N\setminus\xO,
\EAL \right.
\eal 
\bal 
\left\{ \BAL
\CL_\xm^s u+g(u)&=\xn_{1,r}&&\quad\text{in}\;\xO,\\
u & = 0 &&\quad\text{in}\;\BBR^N\setminus\xO,
\EAL \right.
\eal
\bal 
\left\{ \BAL
\CL_\xm^s u+g(u)&=\xn_{2,r}&&\quad\text{in}\;\xO,\\
u & = 0 &&\quad\text{in}\;\BBR^N\setminus\xO.
\EAL \right.
\eal
Proceeding as in {\it Step 2} of the proof of Theorem \ref{existence4}, we can show that there exist functions $u,u_i,v_i $, $i=1,2$, such that $\{u_r\}$, {\color{blue}$\{u_{i,r}\}$, $\{v_{i,r}\}$} converge to $u$, $u_i$, $v_i$, $i=1,2$. Moreover, $u$, $u_1$, $u_2$, $v_i$ are the unique weak solutions to problems \eqref{prob:gu}, \eqref{prob:gu1}, \eqref{prob:gu2}, \eqref{linear-vi} respectively.

Finally, estimate \eqref{apriori-1} follows from \eqref{apriori-2} by letting $\delta \to 0$ and $r \to 0$ successively.
\end{proof}

\section{Semilinear problem with Dirac measures concentrated at $0$}

Let $\ell \in\BBR$ and $\xO'$ be an open bounded domain such that $\xO\Subset\xO'$ and $\xe\leq \frac{d(0)}{16}.$  Set $H(x):= \ell \xF_{s,\xm}^\xO(x)|x|^{-\tau_+}$ for $x \neq 0$, where $\Phi_{s,\mu}^{\Omega}$ is given in \eqref{PhiOmega}.

Inspired by \cite{KKP2}, we will prove the existence using the respective monotone operator. Set $\xO_\xe'=\xO'\setminus B_{\xe}(0)$ and $\xO_\xe=\xO\setminus B_{\xe}(0).$ For any $u,v\in H^s(\xO_{\frac{\xe}{2}}'),$ we define the operators
\bal
\mathcal{A}_1u(v):=\frac{C_{N,s}}{2}\int_{\xO_{\frac{\xe}{2}}'}\int_{\xO_{\frac{\xe}{2}}'}\frac{(u(x)-u(y))(v(x)-v(y))}{|x-y|^{N+2s}}|y|^{\tau_+} dy |x|^{\tau_+} dx
\eal
and
\bal
\mathcal{A}_2u(v)&:=C_{N,s}\int_{\BBR^N\setminus\xO_{\frac{\xe}{2}}'}\int_{\xO_\xe}\frac{(u(y)-H(x))v(y)}{|x-y|^{N+2s}}|y|^{\tau_+} dy |x|^{\tau_+} dx+\int_{\xO_\xe} g(|x|^{\tau_+} u)v|x|^{\tau_+} dx,
\eal
where $g\in C(\BBR)\cap L^\infty(\xO)$ is a nondecreasing function such that $g(0)=0.$ We set
\ba \label{AA1A2}
\mathcal{A}u(v):=\mathcal{A}_1u(v)+\mathcal{A}_2u(v)\quad \text{for } u,v\in H^s(\xO_{\frac{\xe}{2}}').
\ea
Finally, set $M=\sup_{t\in\BBR}|g(t)|$.

\begin{lemma} \label{dual}
	Let $u\in H^s(\xO_{\frac{\xe}{2}}'),$ then $\mathcal{A}u\in (H^s(\xO_{\frac{\xe}{2}}'))^*.$
\end{lemma}
\begin{proof}[\textbf{Proof}]
	Let $v\in H^s(\xO_{\frac{\xe}{2}}'),$ then using the H\"{older} inequality, we can easily show that
	\bal
	|\mathcal{A}_1u(v)|\leq C(N,s,\xO',\xe)\norm{u}_{H^s(\xO_{\frac{\xe}{2}}')}\norm{v}_{H^s(\xO_{\frac{\xe}{2}}')}.
	\eal
	
	By the H\"{o}lder inequality and the definition of $H$, we have
	\bal
	&\quad \bigg|\int_{\BBR^N\setminus\xO_{\frac{\xe}{2}}'}\int_{\xO_\xe}\frac{(u(y)-H(x))v(y)}{|x-y|^{n+2s}}|y|^{\tau_+} dy |x|^{\tau_+} dx\bigg|\\
	&\leq C(\xO,\xO',\xe,N,s)\bigg(\norm{u}_{H^s(\xO_{\frac{\xe}{2}}')} +\int_{\BBR^N\setminus B_{d(0)}(0)}\frac{H(x)|x|^{\tau_+}}{1+|x|^{N+2s}}\bigg)\norm{v}_{H^s(\xO_{\frac{\xe}{2}}')} \\
	&\leq C(\xO,\xO',\xe,N,s,\ell)(\norm{u}_{H^s(\xO_{\frac{\xe}{2}}')} +1)\norm{v}_{H^s(\xO_{\frac{\xe}{2}}')}.
	\eal
	
	Again, by the H\"older inequality, we also obtain
	\bal
	\bigg|\int_{\xO_\xe} g(|x|^{\tau_+} u)v|x|^{\tau_+} dx\bigg|\leq C(\xO,M,\xe) \norm{v}_{H^s(\xO_{\frac{\xe}{2}}')}.
	\eal
	From the above preceding estimates and \eqref{AA1A2}, we obtain
	\bal |\mathcal{A}u(v)|\leq C(N,s,\Omega,\xO',\xe,\ell,M,\norm{u}_{H^s(\xO_{\frac{\xe}{2}}')})\norm{v}_{H^s(\xO_{\frac{\xe}{2}}')},
	\eal
	which implies the desired result.
\end{proof}
Set
\ba \label{Ke}
\mathcal{K}_{\xe}&:=\big\{v\in H^s(\xO_{\frac{\xe}{2}}'):\;v=H \text{ a.e. in }\;\BBR^N\setminus \Omega_{\varepsilon} \big\}.
\ea
We note that $\mathcal{K}_{\xe}$ is not empty since $H\in \mathcal{K}_{\xe}$ convex and closed.

\begin{lemma}\label{monotone}
	The operator $\mathcal{A}$ is monotone, coercive and weakly continuous on the set $\mathcal{K}_\xe.$
\end{lemma}
\begin{proof}[\textbf{Proof}] We split the proof into three steps. \medskip
	
\noindent \textbf{Step 1.} \emph{We claim that the operator $\mathcal{A}$ is monotone on the set $\mathcal{K}_\xe$}, namely
 \ba \label{Amonotone}
	\mathcal{A}u(u-v)-\mathcal{A}v(u-v)\geq 0, \quad\forall u,v\in \mathcal{K}_\xe.
 	\ea
	
	Indeed, we have
	\ba\label{15+}\BAL
	\mathcal{A}_1u(u-v)&-\mathcal{A}_1v(u-v)=
	\frac{C_{N,s}}{2}\int_{\xO_{\frac{\xe}{2}}'}\int_{\xO_{\frac{\xe}{2}}'}\frac{|u(x)-v(x)-(u(y)-v(y))|^2}{|x-y|^{N+2s}}|y|^{\tau_+} dy |x|^{\tau_+} dx \geq 0
	\EAL
	\ea
	and
	\ba\label{16}\BAL
	\mathcal{A}_2u(u-v)-\mathcal{A}_2v(u-v)&=C_{N,s}\int_{\BBR^N\setminus\xO_{\frac{\xe}{2}}'}\int_{\xO_\xe}\frac{|u(y)-v(y)|^2}{|x-y|^{N+2s}}|y|^{\tau_+} dy |x|^{\tau_+} dx\\
	&\quad +\int_{\xO_\xe} (g(|x|^{\tau_+} u)-g(|x|^{\tau_+} v))(u-v)|x|^{\tau_+} dx\\
	&\geq C_{N,s}\int_{\BBR^N\setminus\xO_{\frac{\xe}{2}}'}\int_{\xO_\xe}\frac{|u(y)-v(y)|^2}{|x-y|^{N+2s}}|y|^{\tau_+} dy |x|^{\tau_+} dx \geq 0,
	\EAL
	\ea
	since $g$ is nondecreasing. Combining \eqref{15+} and \eqref{16} leads to \eqref{Amonotone}.
	\medskip
	
\noindent \textbf{Step 2.} \emph{We claim that the operator $\mathcal{A}$ is coercive on the set $\mathcal{K}_\xe$}, namely, there exists $v\in \mathcal{K}_\xe$ such that
	\bal
	\frac{\mathcal{A}u_j(u_j-v)-\mathcal{A}v(u_j-v)}{\norm{u_j-v}_{H^s(\xO_{\frac{\xe}{2}}')}}\to +\infty\quad \text{as}\;\;\norm{u_j-v}_{H^s(\xO_{\frac{\xe}{2}}')}\to +\infty.
	\eal
	We fix a function $v\in \mathcal{K}_\xe.$ In view of the proof of \eqref{15+} and \eqref{16}, we have
	\bal
	\mathcal{A}u_j(u_j-v)-\mathcal{A}v(u_j-v)&\geq \frac{C_{N,s}}{2}\int_{\xO_{\frac{\xe}{2}}'}\int_{\xO_{\frac{\xe}{2}}'}\frac{|u_j(x)-v(x)+v(y)-u_j(y)|^2}{|x-y|^{n+2s}}|y|^{\tau_+} dy |x|^{\tau_+} dx\\
	&\geq C(N,s,\Omega,\Omega',\xe)\norm{u_j-v}_{H^s(\xO_{\frac{\xe}{2}}')}^2,
	\eal
	where in the last inequality we have used the standard fractional Sobolev embedding since $u_j-v\in W^{s,2}_0(\xO_{\frac{\xe}{2}}')$.
	
	\medskip
	
\noindent \textbf{Step 3.} \emph{We claim that the operator $\mathcal{A}$ is weakly continuous on the set $\mathcal{K}_\xe$}. To this purpose, we need to show that if $\{u_j\}_{j=1}^\infty\subset\mathcal{K}_\xe$ such that $u_j\to u\in\mathcal{K}_\xe $ in $H^s(\xO_{\frac{\xe}{2}}')$ then $\mathcal{A}u_j(v)-\mathcal{A}u(v)\to 0$ for any $v\in H^s(\xO_{\frac{\xe}{2}}').$
	
	First we note that, for any $v\in H^s(\xO_{\frac{\xe}{2}}')$,
	\ba\label{15}\BAL
	&\quad |\mathcal{A}_1u_j(v)-\mathcal{A}_1u(v)|\\
	&\leq
	\frac{C_{N,s}}{2}\int_{\xO_{\frac{\xe}{2}}'}\int_{\xO_{\frac{\xe}{2}}'}\frac{|u_j(x)-u(x)-(u_j(y)-u(y))||v(x)-v(y)|}{|x-y|^{N+2s}}|y|^{\tau_+} dy |x|^{\tau_+} dx \\
	&\leq \frac{C_{N,s}}{2} \| u_j-u \|_{H^s(\Omega'_{\frac{\varepsilon}{2}})} \| v \|_{H^s(\Omega'_{\frac{\varepsilon}{2}})} \to 0 \quad \text{as } j \to +\infty.
	\EAL
	\ea
	
	Next we have
\bal
	\BAL
	|\mathcal{A}_2u_j (v) - \mathcal{A}_2 u(v)| &\leq C_{N,s}\int_{\BBR^N\setminus\xO_{\frac{\xe}{2}}'}\int_{\xO_\xe}\frac{|u_j(y)-u(y)||v(y)|}{|x-y|^{N+2s}}|y|^{\tau_+} dy |x|^{\tau_+} dx\\
	&\quad +\int_{\xO_\xe} (g(|x|^{\tau_+} u_j(x))-g(|x|^{\tau_+} u(x)))v(x)|x|^{\tau_+} dx.
	\EAL
\eal
	
	Since $u_j\to u\in\mathcal{K}_\xe $ in $H^s(\xO_{\frac{\xe}{2}}'),$ we have that $u_j\to u$ in $L^2(\xO_\xe).$
	By the H\"older inequality and the estimate (due to the fact $\frac{2s-N}{2} \leq \tau_+ <2s<N$)
	\bal
	\int_{\BBR^N\setminus\xO_{\frac{\xe}{2}}'} \frac{|x|^{\tau_+}}{|x-y|^{N+2s}}dx < C(N,\Omega,\Omega',s,\mu,\varepsilon), \quad \forall y \in \Omega_{\varepsilon},
	\eal
we deduce that
\bal 
	\BAL
	&\quad \int_{\BBR^N\setminus\xO_{\frac{\xe}{2}}'}\int_{\xO_\xe}\frac{|u_j(y)-u(y)||v(y)|}{|x-y|^{N+2s}}|y|^{\tau_+} dy |x|^{\tau_+} dx \\ &\leq C(s,\mu,\varepsilon)\int_{\xO_\xe}|u_j(y)-u(y)||v(y)| \int_{\BBR^N\setminus\xO_{\frac{\xe}{2}}'} \frac{|x|^{\tau_+}}{|x-y|^{N+2s}}dx dy \\
	&\leq C(N,\Omega,\Omega',s,\mu,\varepsilon)\|u_j-u  \|_{L^2(\Omega_{\varepsilon})}\|v \|_{L^2(\Omega_{\varepsilon})} \to 0 \quad \text{as } j \to \infty.
	\EAL \eal
Since $g$ is bounded and uniformly continuous in $\BBR,$ $\xe<|x|<\mathrm{diam}(\xO)$ for any $x \in \xO_\xe$ and $u_\xe\to u$ in measure in $\xO_\xe,$ we may deduce that
	\bal
	\int_{\xO_\xe} |g(|x|^{\tau_+} u_j)-g(|x|^{\tau_+} u)|^2dx\to 0 \quad \text{as } j \to +\infty. 
	\eal
	Therefore, by the H\"older inequality, we obtain
	\ba \label{17-b} \BAL
	&\quad \bigg|\int_{\xO_\xe} (g(|x|^{\tau_+} u_j)-g(|x|^{\tau_+} u))v|x|^{\tau_+} dx\bigg|\\
	&\leq C(s,\mu,\xe,\xO)\left(\int_{\xO_\xe} |g(|x|^{\tau_+} u_j)-g(|x|^{\tau_+} u)|^2dx\right)^{\frac{1}{2}}\norm{v}_{L^2(\xO_\xe)} \to 0 \quad \text{as } j \to +\infty.
	\EAL \ea
	
	Combining \eqref{15}--\eqref{17-b}, we conclude that $\mathcal{A}u_j(v)-\mathcal{A}u(v)\to 0$ as $j \to +\infty$. The proof is complete.
\end{proof}

\begin{lemma}\label{aproxsol}
	Assume $\ell>0$ and $g\in C(\BBR)\cap L^\infty(\xO)$ is a nondecreasing function such that $g(0)=0$. There exists a function $u_\xe\in W^{s,2}(\xO_{\frac{\xe}{2}}')$ such that $u_{\varepsilon}=\ell \xF_{s,\xm}^\xO$ in $\BBR^N\setminus \xO_\xe$ and
	\ba \label{aprox-1}
	\frac{C_{N,s}}{2}\int_{\BBR^N}\int_{\BBR^N}\frac{(u_\xe(x)-u_\xe(y))(\xf(x)-\xf(y))}{|x-y|^{N+2s}}dy dx+\xm\int_{\xO_\varepsilon}\frac{u_\xe(x)\xf(x)}{|x|^{2s}}dx +\int_{\xO_\xe} g(u_\xe)\xf dx=0
	\ea
	for all $\xf\in C_0^\infty(\xO_{\xe})$.
	Furthermore there holds
	\ba \label{uplow-1} (\ell \xF_{s,\xm}^\xO-\psi_M)^+\leq u_{\varepsilon} \leq \ell \xF_{s,\xm}^\xO \quad \text{in } {\color{blue}\xO_{\varepsilon}},
	\ea
	where $M=\sup_{t\in\BBR}|g(t)|$, $\psi_M$ is the nonnegative weak solution of
	\ba \label{psiM} \left\{ \BAL
	\CL_\xm^s \psi&=M&&\quad\text{in}\;\, \xO,\\
	u&=0&&\quad\text{in}\;\, \BBR^N\setminus\xO.
	\EAL \right.
	\ea

\end{lemma}
\begin{proof}[\textbf{Proof}]
	By Lemma \ref{monotone} and the standard theory of monotone operators (see, e.g., \cite[Proposition 17.2]{HKM}), there exists $v_\xe\in \mathcal{K}_\xe$ such that
	\bal
	\mathcal{A}v_\xe(\xz-v_\xe)\geq0,
	\eal
	for any $\xz\in \mathcal{K}_\xe.$ Set $\xz_\pm=\pm\tilde\xf+v_\xe$ for $\tilde\xf\in C_0^\infty(\xO_\xe),$ then  $\xz_\pm\in \mathcal{K}_\xe$ and by the above inequality we can easily show that
	\ba \label{aprox-2} \BAL
	0&=\mathcal{A}v_\xe(\tilde\xf) \\
	&=\frac{C_{N,s}}{2}\int_{\xO_{\frac{\xe}{2}}'}\int_{\xO_{\frac{\xe}{2}}'}\frac{(v_\xe(x)-v_\xe(y))(\tilde\xf(x)-\tilde\xf(y))}{|x-y|^{N+2s}}|y|^{\tau_+} dy|x|^{\tau_+} dx\\
	&\quad +C_{N,s}\int_{\BBR^N\setminus\xO_{\frac{\xe}{2}}'}\int_{\xO_\xe}\frac{(v_\xe(y)-H(x))\tilde\xf(y)}{|x-y|^{N+2s}}|y|^{\tau_+} dy |x|^{\tau_+} dx+\int_{\xO_\xe} g(|x|^\tp v_\xe)\tilde\xf|x|^{\tau_+} dx\\
	&=\frac{C_{N,s}}{2}\int_{\BBR^N}\int_{\BBR^N}\frac{(v_\xe(x)-v_\xe(y))(\tilde\xf(x)-\tilde\xf(y))}{|x-y|^{N+2s}}|y|^{\tau_+} dy|x|^{\tau_+} dx+\int_{\xO_\xe} g(|x|^{\tau_+} v_\xe)\tilde\xf|x|^{\tau_+} dx.
	\EAL \ea
	Setting $u_\xe=|x|^{\tau_+} v_\xe$ and $\xf=|x|^{\tau_+}\tilde\xf,$ we obtain \eqref{aprox-1}.
	
	We have
	\ba \label{aprox-3} \BAL
	\langle H, \tilde \phi \rangle_{s,\tau_+}
	+\int_{\xO_\xe} g(|x|^{\tau_+} H)\tilde\xf|x|^{\tau_+} dx\geq 0, \quad \forall \; 0 \leq  \tilde\xf\in C_0^\infty(\xO_{\varepsilon}).
	\EAL \ea
	
	Put $w_{\varepsilon}=v_{\varepsilon}-H$ then from \eqref{aprox-2} and \eqref{aprox-3}, we have
	\ba \label{aprox-4} 
	\langle w_{\varepsilon}, \tilde \phi \rangle_{s,\tau_+}  +\int_{\xO_\xe} (g(|x|^{\tau_+} v_{\varepsilon}) - g(|x|^{\tau_+} H))\tilde\xf|x|^{\tau_+} dx \leq 0, \quad \forall \; 0 \leq \tilde\xf\in C_0^\infty(\xO_{\varepsilon}).
	\ea
	Note that $w_{\varepsilon}^+\in W^{s,2}(\xO_{\frac{\xe}{2}}')$ and $w_{\varepsilon}^+=0$ in $\BBR^N\setminus \xO_\xe.$ Since $\xO_\xe$ is smooth, we deduce that $w_{\varepsilon}^+\in W^{s,2}_0(\xO_\xe),$ hence we may use it as a test function in \eqref{aprox-4} and the standard density argument together with the monotonicity assumption on $g$ to obtain that
\bal 
	0  \leq \langle w_{\varepsilon}, w_\varepsilon^+ \rangle_{s,\tau_+}
	+\int_{\xO_\xe} (g(|x|^{\tau_+} v_{\varepsilon}) - g(|x|^{\tau_+} H))w_{\varepsilon}^+|x|^{\tau_+} dx \leq 0.
\eal
	This implies
	$w_{\varepsilon}^+=0$ a.e. in $\BBR^N \setminus \{0\}$, hence $v_\xe\leq H$ a.e. in $\Omega_{\varepsilon}$. This implies the upper bound in \eqref{uplow-1}.
	
	As for the lower bound in \eqref{uplow-1}, by taking $\tilde \phi = v_{\varepsilon}^- \in W_0^{s,2}(\Omega_\varepsilon)$ in \eqref{aprox-2}, the standard density argument and the assumption that $g$ is nondecreasing and $g(0)=0$, we have
\bal 
	0=\langle v_{\varepsilon}, v_{\varepsilon}^- \rangle_{s,\tau_+} +\int_{\xO_\xe} g(|x|^{\tau_+} v_\xe)v_{\varepsilon}^-|x|^{\tau_+} dx \leq 0,
\eal
	which implies $v_\varepsilon^-=0$ a.e. in $\R^N \setminus \{0\}$. Therefore $v_{\varepsilon} \geq 0$ in $\Omega_{\varepsilon}$, and hence $u_{\varepsilon} \geq 0$ a.e. in $\Omega_{\varepsilon}$.
	
	Since  $\psi_M$ is the nonnegative solution of \eqref{psiM}, we have \bal \CL_\xm^s(\ell\xF_{s,\xm}^\xO-\psi_M)+g(\ell\xF_{s,\xm}^\xO-\psi_M) \leq 0
	\eal
	in the sense of distribution in $\xO_\xe.$ By a similar argument as above, we may show that $u_\xe\geq \ell\xF_{s,\xm}^\xO-\psi_M$ a.e. in $\Omega_{\varepsilon}$. Thus $u \geq (\ell\xF_{s,\xm}^\xO-\psi_M)^+$ a.e. in $\Omega_{\varepsilon}$.
\end{proof}

\begin{definition} \label{sol:semieq-Dirac} {\it
	Assume $\ell \in \R$ and $g: \R \to \R$ is a nondecreasing continuous function such that $g(0)=0$. A function $u$ is called a \textit{weak solution} to problem 	
	\ba \label{eq:g(u)-kdirac} \left\{\BAL
	\CL_\xm^s u+g(u)&=\ell \delta_0&&\quad\text{in}\;\,\xO,\\
	u&=0&&\quad\text{in}\;\,\BBR^N\setminus\xO,
	\EAL \right.
	\ea
	if for any $b<2s-\tau_+$, $u\in L^1(\xO;|x|^{-b}),$ $g(u)\in L^1(\xO;|x|^{\tau_+})$ and
	\ba \label{sol:g(u)-kdirac}
	\int_\xO u(-\xD)^s_{\tau_+}\psi dx+\int_\xO g(u)\psi|x|^{\tau_+} dx=\ell\int_{\xO}\xF_{s,\xm}^\xO (-\xD)^s_{\tau_+}\psi dx,\quad\forall\psi\in \mathbf{X}_\xm(\xO;|x|^{-b}),
	\ea	
where $\mathbf{X}_\xm(\xO;|x|^{-b})$ is defined in Definition \ref{def:weaksol}.}
\end{definition}


\begin{theorem} \label{existence-gLinf}
	Assume $\ell >0$ and $g\in C(\BBR)\cap L^\infty(\xO)$ is a nondecreasing function such that $g(0)=0$. Then there exists a unique weak solution $u\in W^{s,2}_{\loc}(\xO'\setminus\{0\})\cap C(\xO\setminus\{0\})$ of \eqref{eq:g(u)-kdirac}. The solution $u$ satisfies
	\ba\label{est}
	(\ell \xF_{s,\xm}^\xO-\psi_M)^+ \leq u\leq \ell\xF_{s,\xm}^\xO\quad \text{a.e. in } \;\xO \setminus \{0\}
	\ea
	and
	\ba \label{assym-1}
	\lim_{\Omega \ni  x \to 0}\frac{u(x)}{\xF_{s,\xm}^\xO(x)}=\ell.
	\ea
\end{theorem}
\begin{proof}[\textbf{Proof}]
	{\sc Uniqueness.} Suppose that $u_1,u_2$ are two weak solutions of \eqref{eq:g(u)-kdirac}. Then for any $b<2s-\tau_+$, by Theorem \ref{existence2},  for any nonnegative $\psi \in \mathbf{X}_{\mu}(\Omega;|x|^{-b})$, 
	\ba \label{u1-u2}
	\int_{\Omega}(u_1-u_2)^+ (-\Delta)_{\tau_+}^s \psi dx + \int_{\Omega} (g(u_1)- g(u_2))(\sign^+(u_1-u_2))\psi |x|^{\tau_+} dx \leq 0. 
	\ea
Taking $\psi=\xi_b$, the solution of \eqref{xib}, in \eqref{u1-u2} and noting that $g$ is nondecreasing, we deduce $(u_1-u_2)^+ = 0$ in $\Omega$. This implies $u_1 \leq u_2$ a.e. in $\Omega$. Similarly, $u_2 \leq u_1$ a.e. in $\Omega$. Thus $u_1 = u_2$ a.e. in $\Omega$.

	\medskip
	
	{\sc Existence.} For $\varepsilon>0$, let $u_\xe$ be the solution in Lemma \ref{aproxsol} and $\eta \in C^\infty(\R)$ such that $0\leq\eta\leq1$, $\eta(t)=0$ for any $|t|\leq 1$ and $\eta(t)=1$ for any $|t|\geq 2$. For $\xe>0$, set $\eta_{\xe}(x)=\eta(\xe^{-1} |x|)$. Consider $\xe$ small enough such that $B_{16\xe}(0)\subset \xO.$ Using $\eta_\xe^2(x)u_\xe(x)$ as a test function in \eqref{aprox-1}, we obtain
	\ba\label{18} \BAL
	&\quad \frac{C_{N,s}}{2}\int_{\BBR^N}\int_{\BBR^N}\frac{(u_\xe(x)-u_\xe(y))(\eta^2_\xe(x)u_\xe(x)-\eta^2_\xe(y)u_\xe(y))}{|x-y|^{N+2s}}dy \\
	&= -\int_{\xO_\xe} g(u_\xe)u_\xe\eta^2_\xe dx-\xm\int_{\xO_\xe}\frac{u_\xe^2(x)\eta^2_\xe(x)}{|x|^{2s}}dx.
	\EAL \ea
	
	By \eqref{uplow-1}, it is easy to see that
\bal 
	\left|\int_{\xO_\xe} g(u_\xe)u_\xe\eta^2_\xe dx \right| + \left|\xm\int_{\xO_\xe}\frac{u_\xe^2(x)\eta^2_\xe(x)}{|x|^{2s}}dx \right|\leq   C(N,s,\xm,\xO,\xe,\ell,M),
\eal
where $M=\sup_{t\in\BBR}|g(t)|$.	
	We write
\bal \BAL
	&\quad \int_{\BBR^N}\int_{\BBR^N}\frac{(u_\xe(x)-u_\xe(y))(\eta^2_\xe(x)u_\xe(x)-\eta^2_\xe(y)u_\xe(y))}{|x-y|^{N+2s}}dydx\\
	&=
	\int_{\xO_{\frac{\xe}{2}}'}\eta^2_\xe(x)\int_{\xO_{\frac{\xe}{2}}'}\frac{|u_\xe(x)-u_\xe(y)|^2}{|x-y|^{N+2s}}dydx\\
	&\quad+\int_{\xO_{\frac{\xe}{2}}'}\int_{\xO_{\frac{\xe}{2}}'}u_\xe(y)\frac{(u_\xe(x)-u_\xe(y))(\eta^2_\xe(x)-\eta^2_\xe(y))}{|x-y|^{N+2s}}dydx  +2\int_{\BBR^N\setminus\xO_{\frac{\xe}{2}}'}\int_{\xO_{\frac{\xe}{2}}'}
	\frac{\eta^2_\xe(y)u_\xe^2(y)}{|x-y|^{N+2s}}dydx.
	\EAL
\eal
	By Young's inequality, we have
\bal 
	\BAL
	&\quad \left| \int_{\xO_{\frac{\xe}{2}}'}\int_{\xO_{\frac{\xe}{2}}'}u_\xe(y)\frac{(u_\xe(x)-u_\xe(y))(\eta^2_\xe(x)-\eta^2_\xe(y))}{|x-y|^{N+2s}}dydx \right| \\
	&\leq \frac{1}{8}\int_{\xO_{\frac{\xe}{2}}'}\int_{\xO_{\frac{\xe}{2}}'}\frac{(u_\xe(x)-u_\xe(y))^2 (\eta_{\varepsilon}(x)+\eta_{\varepsilon}(y))^2}{|x-y|^{N+2s}}dydx
	+8\int_{\xO_{\frac{\xe}{2}}'}\int_{\xO_{\frac{\xe}{2}}'}u_\xe^2(y)\frac{(\eta_\xe(x)-\eta_\xe(y))^2}{|x-y|^{N+2s}}dydx \\
	&\leq\frac12 \int_{\xO_{\frac{\xe}{2}}'}\eta^2_\xe(x)\int_{\xO_{\frac{\xe}{2}}'}\frac{(u_\xe(x)-u_\xe(y))^2}{|x-y|^{N+2s}}dydx + 8\int_{\xO_{\frac{\xe}{2}}'}\int_{\xO_{\frac{\xe}{2}}'}u_\xe^2(y)\frac{(\eta_\xe(x)-\eta_\xe(y))^2}{|x-y|^{N+2s}}dydx
\EAL \eal
	and by Lemma \ref{aproxsol} and the regularity of $\eta_{\varepsilon}$, we deduce
	\ba \label{21-b}
	\int_{\xO_{\frac{\xe}{2}}'}\int_{\xO_{\frac{\xe}{2}}'}u_\xe^2(y)\frac{(\eta_\xe(x)-\eta_\xe(y))^2}{|x-y|^{N+2s}}dydx \leq C(N,\Omega,s,\mu,\varepsilon,\ell).
	\ea
	Moreover, by using the fact that $\eta_{\varepsilon}=0$ in $B_{\varepsilon}(0)$ and $u_{\varepsilon}=0$ in $\R^N \setminus \Omega$, we have
	\ba \label{21-c}
	\int_{\BBR^N\setminus\xO_{\frac{\xe}{2}}'}\int_{\xO_{\frac{\xe}{2}}'}
	\frac{\eta^2_\xe(y)u_\xe^2(y)}{|x-y|^{N+2s}}dydx = \int_{\BBR^N\setminus\xO_{\frac{\xe}{2}}'}\int_{\xO_\xe}
	\frac{\eta^2_\xe(y)u_\xe^2(y)}{|x-y|^{N+2s}}dydx \leq C(N,s,\mu,\Omega,\Omega',\varepsilon,\ell).
	\ea
	Combining \eqref{18}--\eqref{21-c} yields
	\ba \label{21-d}
	\int_{\xO_{\frac{\xe}{2}}'}\eta^2_\xe(x)\int_{\xO_{\frac{\xe}{2}}'}\frac{(u_\xe(x)-u_\xe(y))^2}{|x-y|^{N+2s}}dydx \leq C(N,s,\mu,\Omega,\Omega',\varepsilon,\ell,M).
	\ea

	Now, let us estimate the following term
	\ba\label{24}\BAL
	&\int_{\BBR^N}\int_{\BBR^N}\frac{|\eta_\xe(x)u_\xe(x)-\eta_\xe(y)u_\xe(y)|^2}{|x-y|^{N+2s}}dydx \\
	&=\int_{\xO_{\frac{\xe}{2}}'}\int_{\xO_{\frac{\xe}{2}}'}\frac{|\eta_\xe(x)u_\xe(x)-\eta_\xe(y)u_\xe(y)|^2}{|x-y|^{N+2s}}dydx+2\int_{\BBR^N\setminus\xO_{\frac{\xe}{2}}'}\int_{\xO_{\frac{\xe}{2}}'}
	\frac{\eta^2_\xe(y)u_\xe^2(y)}{|x-y|^{N+2s}}dydx.
	\EAL
	\ea
	The first term on the right of \eqref{24} is bounded from above by using \eqref{21-d} and \eqref{21-b}
	\ba\label{25}\BAL
	&\int_{\xO_{\frac{\xe}{2}}'}\int_{\xO_{\frac{\xe}{2}}'}\frac{|\eta_\xe(x)u_\xe(x)-\eta_\xe(y)u_\xe(y)|^2}{|x-y|^{N+2s}}dydx \\
	&\leq
	2\int_{\xO_{\frac{\xe}{2}}'}\eta^2_\xe(x)\int_{\xO_{\frac{\xe}{2}}'}\frac{(u_\xe(x)-u_\xe(y))^2}{|x-y|^{N+2s}}dydx
+2\int_{\xO_{\frac{\xe}{2}}'}\int_{\xO_{\frac{\xe}{2}}'}u_\xe^2(y)\frac{(\eta_\xe(x)-\eta_\xe(y))^2}{|x-y|^{N+2s}}dydx \\
	&\leq C(N,s,\mu,\Omega,\Omega',\varepsilon,\ell,M).
	\EAL
	\ea
	Plugging \eqref{25} and \eqref{21-c} into \eqref{24} leads to
\ba\label{28-1}
	 \int_{\BBR^N}\int_{\BBR^N}\frac{|\eta_\xe(x)u_\xe(x)-\eta_\xe(y)u_\xe(y)|^2}{|x-y|^{N+2s}}dydx  \leq C(N,s,\mu,\Omega,\Omega',\varepsilon,\ell,M).
	 \ea
	
	By the standard fractional compact Sobolev embedding and a diagonal argument, there exists a subsequence $\{u_{\xe_k}\}_{k \in \N}$ such that
	$u_{\xe_k}\to u$ a.e. in $\Omega \setminus \{0\}$. We find that $u\in W^{s,2}_{\loc}(\BBR^N\setminus \{0\})\cap C(\xO\setminus\{0\})$ and $u=0$ in $\BBR^N\setminus\xO$.
	Moreover, for any $\xf\in C_0^\infty(\xO \setminus \{0\})$, there exists $\bar \varepsilon >0$ such that $\phi \in C_0^\infty(\Omega_{\bar \varepsilon})$. Thus, for $2\varepsilon_k \leq \bar \varepsilon$, $\phi \in C_0^\infty(\Omega_{\varepsilon_k})$. Therefore, by the dominated convergence theorem, we obtain
	\bal
	\lim_{k\to\infty}\int_{\BBR^N}\int_{\BBR^N}\frac{(u_{\xe_k}(x)-u_{\xe_k}(y))(\xf(x)-\xf(y))}{|x-y|^{N+2s}}dy dx&=
	\int_{\BBR^N}\int_{\BBR^N}\frac{(u(x)-u(y))(\xf(x)-\xf(y))}{|x-y|^{N+2s}}dy dx, \\
	\lim_{n\to\infty}\int_{\xO_{\varepsilon_k}} \frac{u_{\xe_k}(x)\xf(x)}{|x|^{2s}}dx&=\int_\xO\frac{u(x)\xf(x)}{|x|^{2s}}dx, \\
	\lim_{k\to\infty}\int_{\xO_{\varepsilon_k}} g(u_{\xe_k})\xf dx&=\int_{\xO} g(u)\xf dx.
	\eal
	Thus by passing $k \to \infty$ in \eqref{aprox-1} with $\varepsilon$ replaced by $\varepsilon_k$, we obtain
	\ba\label{eq-usol}\BAL
	\frac{C_{N,s}}{2}\int_{\BBR^N}\int_{\BBR^N}\frac{(u(x)-u(y))(\xf(x)-\xf(y))}{|x-y|^{N+2s}}dy dx+\xm\int_\xO\frac{u(x)\xf(x)}{|x|^{2s}}dx+\int_{\xO} g(u)\xf dx=0
	\EAL
	\ea
	for all $\xf\in C_0^\infty(\xO\setminus\{0\})$.
	
	We see that $u$ satisfies \eqref{est} due to \eqref{uplow-1}. Next by \cite[Lemma 4.4,4.5]{CGN}, $\psi_M \in C^\xb(\xO\setminus\{0\})$ for any $\xb\in (0,2s)$ and
\bal 
	\psi_M(x) \leq C(M,\xO,s,N,\xm) d(x)^s|x|^{\tau_+} \quad \forall x \in \Omega \setminus \{0\},
\eal
	which implies
	\ba \label{psiM/Phi} \lim_{\Omega \ni x \to 0}\frac{\psi_M(x)}{\xF_{s,\xm}^\xO(x)}=0.
	\ea
This and \eqref{est} lead to \eqref{assym-1}.
	
	Since $g \in L^\infty(\R)$, by Theorem \ref{existence2}, there exists a unique weak solution $w \in L^1(\Omega;|x|^{-b})$ of
	\bal \left\{ \BAL
	\CL_\xm^s w &=g(u)&&\quad\text{in}\;\,\xO,\\
	w&=0&&\quad\text{in}\;\, \BBR^N\setminus\xO,
	\EAL \right.
	\eal
	namely
	\ba\label{102}
	\int_\xO w(-\xD)^s_{\tau_+}\psi dx=\int_\xO g(u) \psi |x|^{\tau_+} dx,\quad\forall\psi\in \mathbf{X}_\xm(\xO;|x|^{-b}).
	\ea
	Thus,  together with \eqref{eq-usol}, we obtain that for  $ \xf\in C_0^\infty(\xO\setminus\{0\})$
	\bal
	\frac{C_{N,s}}{2} \int_{\BBR^N}\int_{\BBR^N}\frac{(u(x)+w(x)-u(y)-w(y))(\xf(x)-\xf(y))}{|x-y|^{N+2s}}dy dx +\xm\int_\xO\frac{(u(x)+w(x))\xf(x)}{|x|^{2s}}dx=0.
	\eal
	
	By \cite[Lemma 4.4,  4.5]{CGN}, $w \in C^\xb(\xO\setminus\{0\})$ for any $\xb\in (0,2s)$ and
	\ba \label{wi-est-1a}
	w(x)\leq C(M,\xO,s,N,\xm) d(x)^s|x|^{\tau_+}, \quad \forall x \in \Omega \setminus \{0\}.
	\ea

	Combining \eqref{assym-1}, \eqref{wi-est-1a}, the definition of $\Phi_{s, \mu}$ in \eqref{fu} and the fact that $\tau_+ \geq \tau_-$, we derive
	\bal
	\lim_{|x|\to 0}\frac{u(x)+w(x)}{\xF_{s,\xm}(x)}=\ell.
	\eal
	Hence by \cite[Theorem 4.14]{CW}, we have that $u_i+w_i=\ell\xF_{s,\xm}^\xO$ a.e. in $\Omega \setminus \{0\}$. This and \eqref{102} imply the desired result.
\end{proof}
\begin{remark}\label{remark}
	If $\ell <0$ then we put $\tilde g(t) = -g(-t)$ and consider the problem
	\ba \label{eq:-l} \left\{ \BAL
	\CL_\xm^s u + \tilde g(u)&=-\ell\delta_0&&\quad\text{in}\;\,\xO,\\
	u&=0&&\quad\text{in}\;\, \BBR^N\setminus\xO.
	\EAL \right.
	\ea
 By Theorem  \ref{existence-gLinf}, there exists a unique weak solution $\tilde u$ of problem \eqref{eq:-l}. Therefore, $-\tilde u$ is the unique weak solution of problem \eqref{eq:g(u)-kdirac}. 
\end{remark}


\begin{proof}[\textbf{Proof of Theorem \ref{semi-dirac>}}]
	In view of Remark \ref{remark}, we may assume that $\ell>0$.
	
	The uniqueness follows from Kato inequality \eqref{Kato||-1}.
	
	Next we show the existence of a solution to problem \eqref{eq:g(u)-kdirac}.
	Let $g_n=\max\{-n,\min\{g,n\}\}$ then $g_n \in C(\R) \cap L^\infty(\R)$. By Theorem \ref{existence-gLinf}, there exists a unique positive weak solution $u_n$ of
	\bal \left\{ \BAL
	\CL_\xm^s u_n+g_n(u)&=\ell \delta_0&&\quad\text{in}\;\,\xO,\\
	u&=0&&\quad\text{in}\;\, \BBR^N\setminus\xO,
	\EAL \right.
	\eal
	namely, for any $b<2s-\tau_+$, $u_n\in L^1(\xO;|x|^{-b})$ and
	 \ba\label{solg(un)dirac}
	\int_\xO u_n(-\xD)^s_{\tau_+}\psi dx+\int_\xO g_n(u_n)\psi|x|^{\tau_+} dx=\ell \int_{\xO}\xF_{s,\xm}^\xO (-\xD)^s_{\tau_+}\psi dx,\quad\forall\psi\in \mathbf{X}_\xm(\xO;|x|^{-b}).
	 \ea
	From the above formula and \eqref{Kato:+-1}, we deduce that $u_{n+1}\leq u_{n}$ for any $n \in \N$. Put $u:=\displaystyle\lim_{n \to\infty}u_n$.
	
	For any $n\in\BBN$, by Theorem \ref{existence-gLinf},  we have
	\ba \label{unPhi-}
	u_{n}(x)\leq \ell\xF_{s,\xm}^\xO(x)\leq \ell |x|^{\tau_-},\quad\forall x\in \xO\setminus\{0\},
	\ea
	which implies $u(x) \leq \ell \Phi_{\mu}^\Omega(x)$ for all $x \in \Omega \setminus \{0\}$. By the monotone convergence theorem, $u_n \to u$ in $L^1(\Omega;|x|^{-b})$.
	
	We notice that $g_n(u_n) \to g(u)$ a.e. in $\Omega \setminus \{0\}$ and in view of \eqref{unPhi-} and assumption $\Lambda_g <\infty$, $0 \leq g_n(u_n) \leq g(\ell |x|^{\tau_-}) \in L^1(\Omega;|x|^{\tau_+})$. Hence by the dominated convergence theorem, we have that $g_n(u_n) \to g(u)$ in $L^1(\xO;|x|^{\tau_+})$.
	
	Let $\psi \in \mathbf{X}_{\mu}(\Omega;|x|^{-b})$ then by \cite[Lemma 4.4]{CGN}, $|\psi| \leq Cd^s$ in $\Omega$. By letting $n \to \infty$ in \eqref{solg(un)dirac}, we conclude \eqref{sol:g(u)-kdirac}.
\end{proof}

\begin{proof}[\textbf{Proof of Theorem \ref{semi-dirac=}}] The proof of this Theorem is similar to that of Theorem \eqref{semi-dirac>} with minor modification, hence we omit it.
\end{proof}

\section{Measures on $\Omega$}

\subsection{Bounded absorption}
Let $\ell \in\BBR$ and $\xO'$ be an open bounded domain such that $\xO\Subset\xO'$ and $\xe\leq \frac{\min_{x\in\partial\xO}|x| }{16}$. Recall that $H(x)= \ell \xF_{s,\xm}^\xO(x)|x|^{-\tau_+}$ for $x \neq 0$. Set $\xO_\xe'=\xO'\setminus B_{\xe}(0)$ and $\xO_\xe=\xO\setminus B_{\xe}(0)$.

For any $u,v\in W^{s,2}(\xO_{\frac{\xe}{2}}')$, we define the operators
\bal
\CB u(v): = \int_{\Omega_{\varepsilon}} fv |x|^{\tau_+}dx
 \quad {\rm and}\quad 
\CT u(v):=\CA u(v) - \CB u(v),
\eal
where $g\in C(\BBR)\cap L^\infty(\xO)$ is a nondecreasing function such that $g(0)=0$, $f \in L^\infty(\Omega_{\varepsilon})$ and $\CA$ is defined in \eqref{AA1A2}. Recall that $M=\sup_{t\in\BBR}|g(t)|$ and $\CK_{\varepsilon}$ is the nonempty convex closed set of $W^{s,2}(\Omega_{\frac{\varepsilon}{2}}')$ defined in \eqref{Ke}.

Using an analogous argument as in the proof of Lemma \ref{dual}, Lemma \ref{monotone}, we can show that $\CT$ is monotone, coercive and weakly continuous on $\CK_{\varepsilon}$. By proceeding as in the proof of Lemma \ref{aproxsol}, we can obtain the following result.
\begin{lemma}\label{aproxsol-sum}
	Assume $\ell \geq 0$ and $g\in C(\BBR)\cap L^\infty(\xO)$ is a nondecreasing function such that $g(0)=0$ and $0 \leq f \in L^\infty(\Omega_{\varepsilon})$. There exists a function $u_\xe = u_{\varepsilon,f,\ell}\in W^{s,2}(\xO_{\frac{\xe}{2}}')$ such that $u_{\varepsilon,f,\ell}=\ell \xF_{s,\xm}^\xO$ in $\BBR^N\setminus \xO_\xe$ and
	\ba \label{aprox-1-sum} \BAL
	\frac{C_{N,s}}{2}\int_{\BBR^N}\int_{\BBR^N}&\frac{(u_{\varepsilon,f,\ell}(x)-u_{\varepsilon,f,\ell}(y))(\xf(x)-\xf(y))}{|x-y|^{N+2s}}dy dx\\
	& +\xm\int_{\xO_\varepsilon}\frac{u_{\varepsilon,f,\ell}(x)\xf(x)}{|x|^{2s}}dx +\int_{\xO_\xe} g(u_{\varepsilon,f,\ell})\xf dx = \int_{\xO_\xe} f\xf dx
	\EAL \ea
	for all $\xf\in C_0^\infty(\xO_{\xe})$.
	Furthermore there holds
	\ba \label{uplow-1-sum} \max\{(\ell\xF_{s,\xm}^\xO-\psi_M)^+,u_{\varepsilon,f,0}\} \leq u_{\varepsilon,f,\ell} \leq \ell \xF_{s,\xm}^\xO + u_{\varepsilon,f,0} \quad \text{in } \xO_{\xe},
	\ea
	where $\psi_M$ is the nonnegative solution of \eqref{psiM} and
$u_{\varepsilon,f,0}$ satisfies \eqref{aprox-1-sum} and $u_{\varepsilon,f,0}=0$ in $\R^N \setminus \Omega_{\varepsilon}$.
\end{lemma}
\begin{proof}[\textbf{Proof}]
	By the standard theory of monotone operators (see, e.g., \cite[Proposition 17.2]{HKM}), there exists $v_{\varepsilon,f,\ell}\in \mathcal{K}_\xe$ such that
	\bal
	\mathcal{A}v_{\varepsilon,f,\ell}(\xz-v_{\varepsilon,f,\ell})\geq0,
	\eal
	for any $\xz\in \mathcal{K}_\xe.$ Set $\xz_\pm=\pm\tilde\xf+v_{\varepsilon,f,\ell}$ for $\tilde\xf\in C_0^\infty(\xO_\xe),$ then  $\xz_\pm\in \mathcal{K}_\xe$ and by the above inequality we can easily show that
	\ba \label{aprox-2-sum} \BAL
	0&=\mathcal{A}v_{\varepsilon,f,\ell}(\tilde\xf) \\
	&=\frac{C_{N,s}}{2}\int_{\xO_{\frac{\xe}{2}}'}\int_{\xO_{\frac{\xe}{2}}'}\frac{(v_{\varepsilon,f,\ell}(x)-v_{\varepsilon,f,\ell}(y))(\tilde\xf(x)-\tilde\xf(y))}{|x-y|^{N+2s}}|y|^{\tau_+} dy|x|^{\tau_+} dx\\
	&\quad +C_{N,s}\int_{\BBR^N\setminus\xO_{\frac{\xe}{2}}'}\int_{\xO_\xe}\frac{(v_{\varepsilon,f,\ell}(y)-H(x))\tilde\xf(y)}{|x-y|^{N+2s}}|y|^{\tau_+} dy |x|^{\tau_+} dx \\
	&\quad +\int_{\xO_\xe} g(|x|^{\tau_+} v_{\varepsilon,f,\ell})\tilde\xf|x|^{\tau_+} dx - \int_{\xO_\xe} f \tilde\xf|x|^{\tau_+} dx\\
	&=\langle v_{\varepsilon,f,\ell}, \tilde \phi \rangle_{s,\tau_+} + \int_{\xO_\xe} g(|x|^{\tau_+} v_{\varepsilon,f,\ell})\tilde\xf|x|^{\tau_+} dx - \int_{\xO_\xe} f \tilde\xf|x|^{\tau_+} dx.
	\EAL \ea
	Setting $u_{\varepsilon,f,\ell}=|x|^{\tau_+} v_{\varepsilon,f,\ell}$ and $\xf=|x|^{\tau_+}\tilde\xf,$ we obtain \eqref{aprox-1-sum}. Moreover, since $v_{\varepsilon,f,\ell} \in \CK_{\varepsilon}$, $v_{\varepsilon,f,\ell} = H$ a.e. in $\R^N \setminus \Omega_{\varepsilon}$, hence $u_{\varepsilon,f,\ell} = \ell\Phi_{s,\mu}^\Omega$ a.e. in $\R^N \setminus \Omega_{\varepsilon}$.
	
	By taking $\tilde \phi = (v_{\varepsilon,f,\ell})^- \in W_0^{s,2}(\Omega_\varepsilon)$ in \eqref{aprox-2-sum}, the standard density argument and the assumption that $g$ is nondecreasing and $g(0)=0$, we have
\bal 
	\BAL
	0&=\langle v_{\varepsilon,f,\ell}, (v_{\varepsilon,f,\ell})^- \rangle_{s,\tau_+} +\int_{\xO_\xe} g(|x|^{\tau_+} v_{\varepsilon,f,\ell})(v_{\varepsilon,f,\ell})^-|x|^{\tau_+} dx - \int_{\xO_\xe} f(v_{\varepsilon,f,\ell})^-|x|^{\tau_+} dx \leq 0,
	\EAL \eal
	which implies $(v_{\varepsilon,f,\ell})^-=0$ a.e. in $\R^N$. Therefore $v_{\varepsilon,f,\ell} \geq 0$, and hence $u_{\varepsilon,f,\ell} \geq 0$ a.e. in $\Omega_{\varepsilon}$.
	
In particular, $v_{\varepsilon,f,0}$ satisfies \eqref{aprox-2-sum} with $\ell=0$, namely
\ba \label{vef0}
	\langle v_{\varepsilon,f,0}, \tilde \phi \rangle_{s,\tau_+} + \int_{\xO_\xe} g(|x|^{\tau_+} v_{\varepsilon,f,0})\tilde\xf|x|^{\tau_+} dx - \int_{\xO_\xe} f \tilde\xf|x|^{\tau_+} dx = 0, \quad \forall \tilde \phi \in C_0^\infty(\Omega_{\varepsilon}).
\ea
	
Moreover, $0 \leq v_{\varepsilon,f,0} \in W^{2,s}(\Omega_{\frac{\varepsilon}{2}}')$ and $v_{\varepsilon,f,0}=0$ a.e. in $\R^N \setminus \Omega_{\varepsilon}$, which in turn implies $v_{\varepsilon,f,0} \in W_0^{2,s}(\Omega_{\varepsilon})$. Therefore, by a density argument, we can take $v_{\varepsilon,f,0}$ as a test function in \eqref{vef0} and use estimate $g(t)t \geq 0$ for any $t \in \R$, and by the embedding inequalities \eqref{subcritsobolev0}--\eqref{subcritsobolev1}, we obtain
\bal 
\| v_{\varepsilon,f,0} \|_{s,\tau_+}^2 \leq C(N,s,\Omega,\mu)\| f \|_{L^\infty(\Omega_{\varepsilon})}.
\eal
Put $u_{\varepsilon,f,0} = |x|^{\tau_+}v_{\varepsilon,f,0}$ then $u_{\varepsilon,f,0} = 0$ a.e. in $\R^N \setminus \Omega_{\varepsilon}$,  and
\ba \label{uef0-2}
\| u_{\varepsilon,f,0} \|_{\mu}^2 \leq  C(N,s,\Omega,\mu)\| f \|_{L^\infty(\Omega_{\varepsilon})}.
\ea

	
	We have
	\ba \label{aprox-3-sum} \BAL
	\langle H, \tilde \phi \rangle_{s,\tau_+}
	= 0, \quad \forall \; \tilde\xf\in C_0^\infty(\xO_{\varepsilon}).
	\EAL \ea
	Therefore,
	\bal
	\langle v_{\varepsilon,f,0} + H, \tilde \phi \rangle_{s,\tau_+} + \int_{\Omega_{\varepsilon}} g(|x|^{\tau_+} (\tilde v_{\varepsilon,f,0}+H)) \tilde\xf|x|^{\tau_+} dx \geq \int_{\Omega_{\varepsilon}} f \tilde \phi |x|^{\tau_+}dx, \quad \forall 0 \leq \tilde \phi \in C_0^\infty(\Omega_{\varepsilon}).
	\eal
	Put $w_{\varepsilon}=v_{\varepsilon,f,\ell}-v_{\varepsilon,f,0}- H$ then from \eqref{aprox-2-sum} and \eqref{aprox-3-sum}, we have
	\ba \label{aprox-4-sum} \BAL
	\langle w_{\varepsilon}, \tilde \phi \rangle_{s,\tau_+}  +\int_{\xO_\xe} (g(|x|^{\tau_+} v_{\varepsilon,f,\ell}) - g(|x|^{\tau_+} (\tilde v_{\varepsilon,f,0}+H)))\tilde\xf|x|^{\tau_+} dx  \leq 0, \; \forall \; 0 \leq \tilde\xf\in C_0^\infty(\xO_{\varepsilon}).
	\EAL \ea
	Note that $w_{\varepsilon}^+\in W^{s,2}(\xO_{\frac{\xe}{2}}')$ and $w_{\varepsilon}^+=0$ in $\BBR^N\setminus \xO_\xe.$ Since $\xO_\xe$ is smooth, we deduce that $w_{\varepsilon}^+\in W^{s,2}_0(\xO_\xe),$ hence, by the density argument, we may use it as test function in \eqref{aprox-4-sum}  together with the monotonicity assumption on $g$ to obtain that
	\ba \label{aprox-5-sum} \BAL
	0& \leq \langle w_{\varepsilon}, w_\varepsilon^+ \rangle_{s,\tau_+}
	+\int_{\xO_\xe} (g(|x|^{\tau_+} v_{\varepsilon,f,\ell}) - g(|x|^{\tau_+}(v_{\varepsilon,f,0} + H)))w_{\varepsilon}^+|x|^{\tau_+} dx \leq 0.
	\EAL \ea
	This implies
	$(w_{\varepsilon})^+=0$ a.e. in $\BBR^N$, hence $v_{\varepsilon,f,\ell} \leq H+ v_{\varepsilon,f,0}$ a.e. in $\Omega_{\varepsilon}$. This implies
	\bal
	u_{\varepsilon,f,\ell} \leq \ell\xF_{s,\xm}^\xO + u_{\varepsilon,f,0} \quad \text{a.e. in } \Omega_{\varepsilon}.
	\eal

	Next we show the lower bound in \eqref{uplow-1-sum}. Since $\displaystyle M=\sup_{t\in\BBR}|g(t)|$ and $\psi_M$ is the nonnegative solution of \eqref{psiM}, we have \bal \CL_\xm^s(\ell\xF_{s,\xm}^\xO-\psi_M)+g(\ell \xF_{s,\xm}^\xO-\psi_M) \leq 0
	\eal
	in the sense of distribution in $\xO_\xe.$ By a similar argument as above, we may show that $u_{\varepsilon,f,\ell} \geq \ell \xF_{s,\xm}^\xO-\psi_M$ a.e. in $\Omega_{\varepsilon}$. Thus $u_{\varepsilon,f,\ell} \geq (\ell \xF_{s,\xm}^\xO-\psi_M)^+$ a.e. in $\Omega_{\varepsilon}$.
	
	From \eqref{aprox-2-sum} and \eqref{vef0}, we see that
	 \ba \label{ve-vef-sum}
	\langle v_{\varepsilon,f,0} -  v_{\varepsilon,f,\ell}, \tilde \phi \rangle_{s,\tau_+} + \int_{\Omega_{\varepsilon}} (g(|x|^{\tau_+} v_{\varepsilon,f,0}) - g(|x|^{\tau_+} v_{\varepsilon,f,\ell})) \tilde \phi |x|^{\tau_+}dx = 0, \quad \forall \tilde \phi \in C_0^\infty(\Omega_{\varepsilon}).
	 \ea
	We note that $(v_{\varepsilon,f,0} - v_{\varepsilon,f,\ell})^+ \in W_0^{s,2}(\Omega_{\varepsilon})$. By density argument, we can take $(v_{\varepsilon,f,0} - v_{\varepsilon,f,\ell})^+$ as a test function in \eqref{ve-vef-sum} to deduce $(v_{\varepsilon,f,0} - v_{\varepsilon,f,\ell})^+ = 0$ in $\Omega_{\varepsilon}$. Therefore  $v_{\varepsilon,f,0} \leq v_{\varepsilon,f,\ell}$ a.e. in $\Omega_{\varepsilon}$.
	The proof is complete.
\end{proof}


\begin{theorem} \label{existence-gLinf-sum}
	Assume $\ell >0$ and $g\in C(\BBR)\cap L^\infty(\xO)$ is a nondecreasing function such that $g(0)=0$ and $\nu \in \GTM^+(\Omega \setminus \{0\}; |x|^{\tau_+})$. Then there exists a unique weak solution $u_{\nu,\ell}\in W^{s,2}_{\loc}(\xO'\setminus\{0\})\cap C(\xO\setminus\{0\})$ of \eqref{eq:g(u)-kdirac-sum}. The solution $u$ satisfies
	\ba\label{est-sum}
	\max\{(\ell \xF_{s,\xm}^\xO-\psi_M)^+, u_{\nu,0} \} \leq u_{\nu,\ell} \leq u_{\nu,0} + \ell \xF_{s,\xm}^\xO\quad \text{a.e. in } \;\xO \setminus \{0\},
	\ea
where $u_{\nu,0}$ is the weak solution to
\bal 
\left\{ \BAL
\CL_\xm^s u + g(u)&= \nu &&\quad\text{in}\;\, \Omega,\\
u&=0&&\quad\text{in}\;\, \BBR^N\setminus\Omega.
\EAL \right.
\eal
\end{theorem}
\begin{proof}[\textbf{Proof}]

{\sc Uniqueness.} Suppose that $u_1,u_2$ are two weak solutions of \eqref{eq:g(u)-kdirac-sum}. Then by Theorem \ref{existence2},
	\ba \label{u1-u2-sum}
	\int_{\Omega}(u_1-u_2)^+ (-\Delta)_{\tau_+}^s \psi dx + \int_{\Omega} (g(u_1)- g(u_2))\psi |x|^{\tau_+} dx \leq 0, \quad \forall \, 0 \leq \psi \in \mathbf{X}_{\mu}(\Omega;|x|^{-b}).
	\ea
	Taking $\psi=\xi_b$, the solution of \eqref{xib}, in \eqref{u1-u2-sum} and noting that $g$ is nondecreasing, we deduce $(u_1-u_2)^+ = 0$ in $\Omega$. This implies $u_1 \leq u_2$ a.e. in $\Omega$. Similarly, $u_2 \leq u_1$ a.e. in $\Omega$. Thus $u_1 = u_2$ a.e. in $\Omega$.

	\medskip
	
{\sc Existence.} The proof of the existence is divided into three steps. \medskip

\noindent \textbf{Step 1.} First we consider the case $\nu=f \in L^\infty(\Omega)$.

For $\varepsilon>0$, let $u_{\varepsilon,f,\ell}$ be the solution in Lemma \ref{aproxsol-sum} and $\eta \in C^\infty(\R)$ such that $0\leq\eta\leq1$, $\eta(t)=0$ for any $|t|\leq 1$ and $\eta(t)=1$ for any $|t|\geq 2$. For $\xe>0$, set $\eta_{\xe}(x)=\eta(\xe^{-1} |x|)$. Consider $\xe$ small enough such that $B_{16\xe}(0)\subset \xO.$ Using $\eta^2_\xe(x)u_\xe(x)$ as a test function in \eqref{aprox-1-sum}, we obtain
\bal
	\BAL
	&\frac{C_{N,s}}{2}\int_{\BBR^N}\int_{\BBR^N}\frac{(u_{\varepsilon,f,\ell}(x)-u_{\varepsilon,f,\ell}(y))(\eta^2_\xe(x)u_{\varepsilon,f,\ell}(x)-\eta^2_\xe(y)u_{\varepsilon,f,\ell}(y))}{|x-y|^{N+2s}}dy \\
	&= -\int_{\xO_\xe} g(u_{\varepsilon,f,\ell})u_{\varepsilon,f,\ell}\eta^2_\xe dx-\xm\int_{\xO_\xe}\frac{u_{\varepsilon,f,\ell}^2(x)\eta^2_\xe(x)}{|x|^{2s}}dx + \int_{\Omega_{\varepsilon}} f u_{\varepsilon,f,\ell}\eta^2_\xe dx.
	\EAL \eal
	
	By Lemma \ref{aproxsol-sum}, it is easy to see that
\bal
	\BAL
	&\Big|\int_{\xO_\xe} g(u_{\varepsilon,f,\ell})u_{\varepsilon,f,\ell}\eta^2_\xe dx \Big| + \Big|\xm\int_{\xO_\xe}\frac{u_{\varepsilon,f,\ell}^2(x)\eta^2_\xe(x)}{|x|^{2s}}dx  \Big| + \Big| \int_{\Omega_{\varepsilon}} f u_{\varepsilon,f,\ell}\eta^2_\xe dx \Big| \\[1mm]
	& \leq C(N,s,\xm,\xO,\xe,\ell,M, \| f\|_{L^\infty(\Omega)}).
	\EAL \eal
By (\ref{28-1}), we have that
 \bal
	&\int_{\BBR^N}\int_{\BBR^N}\frac{|\eta_\xe(x)u_{\varepsilon,f,\ell}(x)-\eta_\xe(y)u_{\varepsilon,f,\ell}(y)|^2}{|x-y|^{N+2s}}dydx  \leq C(N,s,\mu,\Omega,\Omega',\varepsilon,\ell,M, \| f \|_{L^\infty(\Omega)}).
 \eal
	
	By the standard fractional compact Sobolev embedding and a diagonal argument, there exists a subsequence $\{u_{\varepsilon_{n},f,\ell}\}_{n \in \N}$ such that
	$u_{\varepsilon_{n},f,\ell}\to u_{f,\ell}$ a.e. in $\Omega \setminus \{0\}$ as $\varepsilon_n \to 0$. We find that $u_{f,\ell}\in W^{s,2}_{\loc}(\BBR^N\setminus \{0\})\cap C(\xO\setminus\{0\})$ and $u=0$ in $\BBR^N\setminus\xO$.
	Moreover, for any $\xf\in C_0^\infty(\xO \setminus \{0\})$, there exists $\bar \varepsilon >0$ such that $\phi \in C_0^\infty(\Omega_{\bar \varepsilon})$. Thus, for $2\varepsilon_n \leq \bar \varepsilon$, $\phi \in C_0^\infty(\Omega_{\varepsilon_n})$. Therefore, by dominated convergence theorem, we obtain
	\bal
	&\lim_{n\to\infty}\int_{\BBR^N}\int_{\BBR^N}\frac{(u_{\varepsilon_{n},f,\ell}(x)-u_{\varepsilon_{n},f,\ell}(y))(\xf(x)-\xf(y))}{|x-y|^{N+2s}}dy dx \\
	&\qquad =
	\int_{\BBR^N}\int_{\BBR^N}\frac{(u_{f,\ell}(x)-u_{f,\ell}(y))(\xf(x)-\xf(y))}{|x-y|^{N+2s}}dy dx, \\
	&\lim_{n\to\infty}\int_{\xO_{\varepsilon_n}} \frac{u_{\varepsilon_{n},f,\ell}(x)\xf(x)}{|x|^{2s}}dx=\int_\xO\frac{u_{f,\ell}(x)\xf(x)}{|x|^{2s}}dx, \\
	&\lim_{n\to\infty}\int_{\xO_{\varepsilon_n}} g(u_{\varepsilon_{n},f,\ell})\xf dx=\int_{\xO} g(u_{f,\ell})\xf dx, \\
	&\lim_{n \to \infty} \int_{\Omega_{\varepsilon_n}} f  \phi dx = \int_{\Omega} f \phi dx.
	\eal
	Thus by passing $n \to \infty$ in \eqref{aprox-1-sum} with $\varepsilon$ replaced by $\varepsilon_n$, we obtain
	 \ba\label{eq-usol-sum}
	\BAL
	&\frac{C_{N,s}}{2}\int_{\BBR^N}\int_{\BBR^N}\frac{(u_{f,\ell}(x)-u_{f,\ell}(y))(\xf(x)-\xf(y))}{|x-y|^{N+2s}}dy dx \\
	&\qquad \qquad +\xm\int_\xO\frac{u_{f,\ell}(x)\xf(x)}{|x|^{2s}}dx +\int_{\xO} g(u_{f,\ell})\xf dx= \int_{\Omega} f \phi dx, \quad \forall \xf\in C_0^\infty(\xO\setminus\{0\}).
	\EAL
 	\ea

Employing estimate \eqref{uplow-1-sum} with $\varepsilon$ replaced by $\varepsilon_n$ and letting $\varepsilon_n \to 0$, we find that
\ba\label{est-ufk}
\max\{(\ell \xF_{s,\xm}^\xO-\psi_M)^+, u_{f,0} \} \leq u_{f,\ell} \leq u_{f,0} + \ell \xF_{s,\xm}^\xO\quad \text{a.e. in } \;\xO \setminus \{0\}.
\ea

Next, by \eqref{uef0-2},  $u_{f,0} \in \Hm$ and by \cite[Lemma 4.4,4.5]{CGN}, we deduce that $u_{f,0} \in C^\xb(\xO\setminus\{0\})$ for any $\xb\in (0,2s)$ and
\bal 
	u_{f,0}(x) \leq C(N,\Omega,s,\mu) d(x)^s|x|^{\tau_+}\| f\|_{L^\infty(\Omega)} \quad \forall x \in \Omega \setminus \{0\},
\eal
	which implies
	\ba \label{uf0/Phi} \lim_{|x| \to 0}\frac{u_{f,0}(x)}{\Phi_{  \mu}^\Omega(x)}=0.
	\ea
Combining \eqref{est-ufk}, \eqref{psiM/Phi} and \eqref{uf0/Phi} lead to
\bal 
\lim_{|x| \to 0}\frac{u_{f,\ell}(x)}{\Phi_{\mu}(x)}=\ell.
\eal

Since $g \in L^\infty(\R)$, by Theorem \ref{existence2}, there exists a unique weak solution $w_{f,\ell}$ of
	\bal \left\{ \BAL
	\CL_\xm^s w &=g(u_{f,\ell}) + f &&\quad\text{in}\;\xO,\\
	w&=0&&\quad\text{in}\;\BBR^N\setminus\xO,
	\EAL \right.
	\eal
namely for any $b<2s-\tau_+$, $w_{f,\ell} \in L^1(\Omega;|x|^{-b})$ and
	\ba \label{sol-w}
	\int_\xO w_{f,\ell}(-\xD)^s_{\tau_+}\psi dx=\int_\xO g(u_{f,\ell}) \psi |x|^{\tau_+} dx + \int_\xO f \psi |x|^{\tau_+} dx,\quad\forall\psi\in \mathbf{X}_\xm(\xO;|x|^{-b}).
	\ea
	This and \eqref{eq-usol-sum} imply
	\bal
	\frac{C_{N,s}}{2}&\int_{\BBR^N}\int_{\BBR^N}\frac{(u_{f,\ell}(x)+w_{f,\ell}(x)-u_{f,\ell}(y)-w_{f,\ell}(y))(\xf(x)-\xf(y))}{|x-y|^{N+2s}}dy dx\\
	&+\xm\int_\xO\frac{(u_{f,\ell}(x)+w_{f,\ell}(x))\xf(x)}{|x|^{2s}}dx=0,\quad\forall \xf\in C_0^\infty(\xO\setminus\{0\}).
	\eal
Since $g \in L^\infty(\R)$ and $f \in L^\infty(\Omega)$, in view of the proof of Theorem \ref{existence2}, $w_{f,\ell} \in \Hm$. It follows by \cite[Lemma 4.4,4.5]{CGN} that  $w_{f,\ell} \in C^\xb(\xO\setminus\{0\})$ for any $\xb\in (0,2s)$ and
	\ba \label{wi-est-1}
	w_{f,\ell}(x)\leq C(M,\xO,s,N,\xm) d(x)^s|x|^{\tau_+}, \quad \forall x \in \Omega \setminus \{0\}.
	\ea

	Combining \eqref{assym-1}, \eqref{wi-est-1}, the definition of $\Phi_{s,\mu}$  in \eqref{fu} and the fact that $\tau_+ \geq \tau_-$, we derive
	\bal
	\lim_{|x|\to 0}\frac{u_{f,\ell}(x)+w_{f,\ell}(x)}{{\xF_{s,\xm}(x)}}=\ell.
	\eal
	Hence by \cite[Theorem 4.14]{CW}, we have that $u_{f,\ell}+w_{f,\ell}=\ell \xF_{s,\xm}^\xO$ a.e. in $\Omega \setminus \{0\}$. Plugging it into \eqref{sol-w} leads to \eqref{sol:g(u)-kdiracdef}, namely $u_{f,\ell}$ is a weak solution of \eqref{eq:g(u)-kdirac-sum}. \medskip
	
\noindent \textbf{Step 2.} We assume that $\nu \in \GTM^+(\Omega \setminus \{0\}; |x|^{\tau_+})$ has compact support in $\Omega \setminus \{0\}$. Let $\{ \zeta_\delta\}$ be the sequence of standard mollifiers. Put $\nu_\delta = \zeta_\delta \ast \nu$ then $0 \leq \nu_\delta \in C_0^\infty(D)$ where $D \Subset \Omega \setminus \{0\}$, then \eqref{nuntonu} and \eqref{nun<nu} hold.

Let $u_{\nu_\delta,\ell}$ be the unique weak solution of \eqref{eq:g(u)-kdirac-sum} with $\nu$ replaced by $\nu_\delta$, namely for any $b<2s-\tau_+$, 
$u_{\nu_\delta,\ell} \in L^1(\xO;|x|^{-b}),$ $g(u_{\nu_\delta,\ell})\in L^1(\xO;|x|^{\tau_+})$ and
 \ba \label{sol:g(u)-kdirac-sum}
\int_\xO u_{\nu_\delta,\ell}(-\xD)^s_{\tau_+}\psi dx+\int_\xO g(u_{\nu_\delta,\ell})\psi|x|^{\tau_+} dx= \int_{\Omega \setminus \{0\} } \psi |x|^{\tau_+}\nu_\delta dx  +  \ell \int_{\xO}\xF_{s,\xm}^\xO (-\xD)^s_{\tau_+}\psi dx,
 \ea
for any $\psi\in \mathbf{X}_\xm(\xO;|x|^{-b})$.

Since $g \in L^\infty(\R)$, by Theorem \ref{existence2}, there exists a unique weak solution $w_{\nu_\delta,\ell}$ of
\bal \left\{ \BAL
\CL_\xm^s w &=g(u_{\nu_\delta,\ell}) + \nu_\delta  &&\quad\text{in}\;\xO,\\
w&=0&&\quad\text{in}\;\BBR^N\setminus\xO.
\EAL \right.
\eal
As in step 1, $u_{\nu_\delta,\ell} + w_{\nu_\delta,\ell} = \ell  \xF_{s,\xm}^\xO$ in $\Omega \setminus \{0\}$ for any $n \in \N$.

By proceeding analogously as in the proof of Theorem \ref{existence-semi-1} and employing \eqref{nun<nu} and the fact that $g \in L^\infty(\R)$, we may show that there exists $w_{\nu,\ell}$ such that, up to a subsequence, $w_{\nu_\delta,\ell} \to w_{\nu,\ell}$ a.e. in $\Omega \setminus \{0\}$ and in $L^1(\Omega;|x|^{-b})$ as $\delta \to 0$ for any $b<2s-\tau_+$. Put $u_{\nu,\ell} = \ell\xF_{s,\xm}^\xO - w_{\nu,\ell}$ then $u_{\nu_\delta,\ell} \to u_{\nu,\ell}$ a.e. in $\Omega \setminus \{0\}$ and in $L^1(\Omega;|x|^{-b})$ as $\delta \to 0$ for any $b<2s-\tau_+$. Since $g \in L^\infty(\Omega) \cap C(\R)$, by the dominated convergence theorem, $g(u_{\nu_\delta,\ell}) \to g(u_{\nu,\ell})$ a.e. in $\Omega \setminus \{0\}$ and in $L^1(\Omega;|x|^{\tau_+})$ as $\delta \to 0$.

Therefore, letting $\delta \to 0$ in \eqref{sol:g(u)-kdirac-sum} leads to
\bal 
\int_\xO u_{\nu,\ell}(-\xD)^s_{\tau_+}\psi dx+\int_\xO g(u_{\nu,\ell})\psi|x|^{\tau_+} dx= \int_{\Omega \setminus \{0\} } \psi |x|^{\tau_+}\nu dx  +  \ell \int_{\xO}\xF_{s,\xm}^\xO (-\xD)^s_{\tau_+}\psi dx
\eal
for any $\psi\in \mathbf{X}_\xm(\xO;|x|^{-b})$. It means $u_{\nu,\ell}$ is a weak solution of \eqref{eq:g(u)-kdirac-sum}. \medskip

\noindent \textbf{Step 3.} We consider $\nu \in \GTM^+(\Omega \setminus \{0\}; |x|^{\tau_+})$. Put $\nu_r = \1_{\Omega \setminus B_r(0)}\nu$ and $\nu_{r,\delta}=\zeta_\delta \ast (\1_{\Omega \setminus B_r(0)}\nu)$. Denote by $u_{\nu_{r},\ell}$ and $u_{\nu_{r,\delta},\ell}$ the nonnegative weak solutions of \eqref{eq:g(u)-kdirac-sum} with $\nu$ replaced by $\nu_r$ and by $\nu_{r,\delta}$ respectively. By step 2, $u_{\nu_{r,\delta},\ell} \to u_{\nu_{r},\ell}$ a.e. in $\Omega \setminus \{0\}$ and in $L^1(\Omega;|x|^{-b})$ as $\delta \to 0$ for any $b<2s-\tau_+$.

Since $\nu_{r} \geq \nu_{r'} \geq 0$ for $0<r \leq r'$, it follows that $\nu_{r,\delta} \geq \nu_{r',\delta}$. By \eqref{Kato:+-1} and the monotonicity of $g$, we deduce that $u_{\nu_{r,\delta},\ell} \geq u_{\nu_{r',\delta},\ell} \geq 0$ for any $0<r<r'$ and $\delta > 0$. Letting $\delta \to 0$ yields $u_{\nu_{r},\ell} \geq u_{\nu_{r'},\ell} \geq 0$ for any $0<r<r'$. Employing the monotonicity convergence theorem, we derive that $u_{\nu_{r},\ell} \uparrow u_{\nu,\ell}$ a.e. in $\Omega$ and in $L^1(\Omega;|x|^{-b})$ as $r \to 0$ for any $b<2s-\tau_+$. Consequently, $g(u_{\nu_{r},\ell}) \uparrow g(u_{\nu,\ell})$ a.e. in $\Omega$ and in $L^1(\Omega;|x|^{\tau_+})$ as $r \to 0$. Passing to the limit, we conclude that $u_{\nu,\ell}$ is a weak solution to \eqref{eq:g(u)-kdirac-sum}.
\end{proof}


\subsection{Unbounded absorption}

\begin{proof}[\textbf{Proof of Theorem \ref{dirac>-sum}}]
	The uniqueness follows from Kato inequality \eqref{Kato||-1}.
	
	Next we show the existence of a solution to problem \eqref{eq:g(u)-kdirac-sum}.
	Let $g_n=\max\{-n,\min\{g,n\}\}$ then $g_n \in C(\R) \cap L^\infty(\R)$. For $r>0$ and $0<\delta<\frac{r}{4}$, put $\nu_r = \1_{\Omega \setminus B_r(0)}\nu$ and $\nu_{r,\delta} = \zeta_{\delta} \ast (\1_{\Omega \setminus B_r(0)}\nu)$. By Theorem \ref{existence-gLinf-sum}, there exists a unique positive weak solution $u_{\nu_{r,\delta},\ell,n}$ of
	\bal \left\{ \BAL
	\CL_\xm^s u +g_n(u)&= \nu_{r,\delta} + \ell \delta_0 &&\quad\text{in}\;\xO,\\
	u&=0&&\quad\text{in}\;\BBR^N\setminus\xO,
	\EAL \right.
	\eal
namely, for any $b<2s-\tau_+$,  $u_{\nu_{r,\delta},\ell,n} \in L^1(\Omega;|x|^{-b})$ and
	\ba\label{solg(un)-rldirac}
	\int_\xO u_{\nu_{r,\delta},\ell,n}(-\xD)^s_{\tau_+}\psi dx+\int_\xO g_n(u_{\nu_{r,\delta},\ell,n})\psi|x|^{\tau_+} dx= \int_{\Omega \setminus \{0\}}\psi |x|^{\tau_+}\nu_{r,\delta}dx +  \ell\int_{\xO}\xF_{s,\xm}^\xO (-\xD)^s_{\tau_+}\psi dx
	\ea
	for all $\psi\in \mathbf{X}_\xm(\xO;|x|^{-b})$.
	From the above formula and \eqref{Kato:+-1}, we deduce that $u_{\nu_{r,\delta},\ell,n+1}\leq u_{\nu_{r,\delta},\ell,n}$ for any $n \in \N$. Put $u_{\nu_{r,\delta},\ell}:=\lim_{n \to\infty}u_{\nu_{r,\delta},\ell,n}$.
	
	For any $n\in\BBN$, by Theorem \ref{existence-gLinf-sum},  we have
\ba \label{ulnPhi}
\max\{ (\ell \xF_{s,\xm}^\xO - \psi_M)^+, u_{\nu_{r,\delta},0,n} \} \leq u_{\nu_{r,\delta},\ell,n} \leq u_{\nu_{r,\delta},0,n} + \ell \xF_{s,\xm}^\xO \quad \text{a.e. in } \xO\setminus\{0\},
\ea
which implies \eqref{est-sum}.

By Theorem \ref{existence4}, $u_{\nu_{r,\delta},0,n} \leq v_{\nu_{r,\delta}}$ a.e. in $\Omega \setminus \{0\}$, where $v_{\nu_{r,\delta}}$ is the unique solution of
\bal 
\left\{ \BAL
\CL_\xm^s v&=\xn_{r,\delta}&&\quad\text{in}\;\, \xO,\\
v&=0&&\quad\text{in}\;\,\BBR^N\setminus\xO.
\EAL \right.
\eal
The existence and uniqueness of $v_{\nu_r,\varepsilon}$ is guaranteed by Theorem \ref{existence2}. Moreover $v_{\nu_r,\delta} \geq 0$ and for any $b<2s-\tau_+$,
\bal
\| v_{\nu_r,\delta} \|_{L^1(\Omega;|x|^{-b})} \leq C(N,\Omega,s,\mu,b)\| \nu_{r,\delta} \|_{L^1(\Omega;|x|^{\tau_+})} \leq C(N,\Omega,s,\mu,b,r)\| \nu \|_{\GTM(\Omega \setminus \{0\};|x|^{\tau_+})}.
\eal
We also note that $\xF_{s,\xm}^\xO \in L^1(\Omega;|x|^{-b})$. Therefore, by the monotone convergence theorem, $u_{\nu_{r,\delta},\ell,n} \to u_{\nu_{r,\delta},\ell}$ in $L^1(\Omega;|x|^{-b})$ as $n \to \infty$.

By Lemma \ref{Marcin-3},
\bal
\| v_{\nu_{r,\delta}} \|_{L_w^{\frac{N}{N-2s}}(\Omega \setminus \{0\};|x|^{\tau_+})} \leq C(N,\Omega,s,\mu,r) \| \nu_{r,\delta} \|_{L^1(\Omega \setminus \{0\};|x|^{\tau_+})},
\eal
which implies
\ba \label{Marcin-4}
\| u_{\nu_{r,\delta},0,n} \|_{L_w^{\frac{N}{N-2s}}(\Omega \setminus \{0\};|x|^{\tau_+})} \leq C(N,\Omega,s,\mu,r) \| \nu_{r,\delta} \|_{L^1(\Omega \setminus \{0\};|x|^{\tau_+})}, \quad \forall n \in \N.
\ea
We can also check that
\ba \label{Marcin-5}
\| \xF_{s,\xm}^\xO \|_{L_w^{\frac{N+\tau_+}{-\tau_-}}(\Omega \setminus \{0\};|x|^{\tau_+})} \leq C(N,s,\mu).
\ea
Combining \eqref{ulnPhi}, \eqref{Marcin-4} and \eqref{Marcin-5} yields
\bal
\| u_{\nu_{r,\delta},\ell,n} \|_{L_w^{p_{s,\mu}^*}(\Omega \setminus \{0\};|x|^{\tau_+})} \leq C(N,\Omega,s,\mu,r) \| \nu_{r,\delta} \|_{L^1(\Omega \setminus \{0\};|x|^{\tau_+})} + C(N,s,\mu,\ell), \quad \forall n \in \N.
\eal
Then Lemma \ref{lem:equi} ensures that the sequence $\{ g_n(u_{\nu_{r,\delta},\ell,n}) \}$ is uniformly bounded and equi-integrable in $L^1(\Omega;|x|^{\tau_+})$. On the other hand, we derive that $g_n(u_{\nu_{r,\delta},\ell,n}) \to g(u_{\nu_{r,\delta},\ell})$ a.e. in $\Omega \setminus \{0\}$ as $n \to \infty$. By the Vitali convergence theorem, we conclude that $g_n(u_{\nu_{r,\delta},\ell,n}) \to g(u_{\nu_{r,\delta},\ell})$ in $L^1(\Omega;|x|^{\tau_+})$ as $n \to \infty$.

Therefore, by sending $n \to \infty$ in \eqref{solg(un)-rldirac}, we derive
\ba\label{solg(un)-ldirac}
\int_\xO u_{\nu_{r,\delta},\ell}(-\xD)^s_{\tau_+}\psi dx+\int_\xO g(u_{\nu_{r,\delta},\ell})\psi|x|^{\tau_+} dx= \int_{\Omega \setminus \{0\}}\psi |x|^{\tau_+}\nu_{r,\delta}dx +  \ell\int_{\xO}\xF_{s,\xm}^\xO (-\xD)^s_{\tau_+}\psi dx
\ea
for all $\psi\in \mathbf{X}_\xm(\xO;|x|^{-b})$.

By a similar argument as above, we can show that $u_{\nu_{r,\delta},\ell} \to u_{\nu_{r},\ell}$ a.e. in $\Omega \setminus \{0\}$ and in $L^1(\Omega;|x|^{-b})$ for any $b<2s-\tau_+$ and $g(u_{\nu_{r,\delta},\ell}) \to g(u_{\nu_{r},\ell})$ a.e. in $\Omega \setminus \{0\}$ and in $L^1(\Omega;|x|^{\tau_+})$ as $\delta \to 0$. By letting $\delta \to 0$ in \eqref{solg(un)-ldirac}, we obtain
\ba\label{solg(un)-rdirac}
\int_\xO u_{\nu_{r},\ell}(-\xD)^s_{\tau_+}\psi dx+\int_\xO g(u_{\nu_{r},\ell})\psi|x|^{\tau_+} dx= \int_{\Omega \setminus \{0\}}\psi |x|^{\tau_+}d\nu_{r} +  \ell\int_{\xO}\xF_{s,\xm}^\xO (-\xD)^s_{\tau_+}\psi dx
\ea
for all $\psi\in \mathbf{X}_\xm(\xO;|x|^{-b})$.

Next we see that $\nu_{r',\delta} \geq \nu_{r,\delta}$ for any $0<r<r'$. By the Kato type inequality \eqref{Kato:+-1} and the monotonicity of $g$, we deduce that $u_{\nu_{r',\delta},\ell} \leq u_{\nu_{r,\delta},\ell}$. It follows that $u_{\nu_{r'},\ell} \leq u_{\nu_{r},\ell}$ for any $0<r<r'$. Put $u_{\nu,\ell}=\lim_{r \to 0}u_{\nu_{r},\ell}$.

Next by taking $\psi=\xi_b$, the solution of problem \eqref{xib}, we deduce that for any $b<2s-\tau_+$,
\bal
\| u_{\nu_{r},\ell} \|_{L^1(\Omega;|x|^{-b})} + \| g(u_{\nu_{r},\ell}) \|_{L^1(\Omega;|x|^{\tau_+})} \leq \| \nu \|_{\GTM(\Omega \setminus \{0\};|x|^{\tau_+})}  + C(N,s,\mu,\ell).
\eal
Therefore, from the monotone convergence theorem and the monotonicity of $g$, we deduce that $u_{\nu_{r},\ell} \to u_{\nu,\ell}$ in $L^1(\Omega;|x|^{-b})$ and $g(u_{\nu_{r},\ell}) \to g(u_{\nu,\ell})$ in $L^1(\Omega;|x|^{\tau_+})$ as $r \to 0$. Thus, letting $r \to 0$ in \eqref{solg(un)-rdirac}, we obtain
\bal
\int_\xO u_{\nu,\ell}(-\xD)^s_{\tau_+}\psi dx+\int_\xO g(u_{\nu,\ell})\psi|x|^{\tau_+} dx= \int_{\Omega \setminus \{0\}}\psi |x|^{\tau_+}d\nu +  \ell\int_{\xO}\xF_{s,\xm}^\xO (-\xD)^s_{\tau_+}\psi dx
\eal
for all $\psi\in \mathbf{X}_\xm(\xO;|x|^{-b})$. It means $u_{\nu,\ell}$ is a weak solution to problem \eqref{eq:g(u)-kdirac-sum}.
\end{proof}

\begin{proof}[\textbf{Proof of Theorem \ref{dirac=-sum}}] The proof of this theorem can be proceeded similarly to that of Theorem \ref{dirac>-sum}.
\end{proof}



\end{document}